\documentclass[11pt]{article}
\usepackage[top=1in, bottom=1in, left=1in, right=1in]{geometry}
\usepackage{tikz}
\usetikzlibrary{shapes,arrows,matrix,patterns,positioning}
\usepackage{color}
\usepackage{epstopdf}
\usepackage{graphicx}
\usepackage{algorithmic}
\usepackage{amssymb}
\usepackage{epstopdf}
\usepackage[norelsize,boxed,linesnumbered,vlined]{algorithm2e}
\usepackage{float}
\usepackage{amsmath,amsthm,bm,color,epsfig,enumerate,caption}
\usepackage{hyperref}
\usepackage{booktabs}
\usepackage{multirow}
\usepackage{array}
\usepackage{todonotes}
\usepackage{xspace}
\usepackage{subcaption}
\usepackage[normalem]{ulem}

\newcommand{\BEA}{\begin{eqnarray}}
\newcommand{\EEA}{\end{eqnarray}}
\newcommand{\comment}[1]{}
\usepackage{amssymb,amsmath,txfonts}
\newtheorem{theorem}{Theorem}

\newtheorem{lemma}{Lemma}
\newtheorem{assumptions}{Assumption}

\makeatletter
\newcommand*{\extendadd}{
  \mathbin{
    \mathpalette\extend@add{}
  }
}
\newcommand*{\extend@add}[2]{
  \ooalign{
    $\m@th#1\leftrightarrow$%
    \vphantom{$\m@th#1\updownarrow$}
    \cr
    \hfil$\m@th#1\updownarrow$\hfil
  }
}
\makeatother

\begin{document}

\title{Machine Learning for Prediction with Missing Dynamics}
\author{John Harlim \\ Department of Mathematics\\ Department of Meteorology and Atmospheric Science\\ Institute for Computational and Data Sciences \\ The Pennsylvania State University, PA 16802, USA \\ \href{mailto:jharlim@psu.edu}{jharlim@psu.edu}
     \and Shixiao W. Jiang \\ Department of Mathematics\\ The Pennsylvania State University, PA 16802, USA \\ \href{mailto:suj235@psu.edu}{suj235@psu.edu}
       \and Senwei Liang \\ Department of Mathematics\\ Purdue University\footnote{Current institute.}, IN 47907, USA\\
       Department of Mathematics\\ National University of Singapore\footnote{Part of the work was done at the National University of Singapore}, Singapore\\ \href{mailto:liang339@purdue.edu}{liang339@purdue.edu}
         \and Haizhao Yang \\ Department of Mathematics\\ Purdue University, IN 47907, USA\\ \href{mailto:haizhao@purdue.edu}{haizhao@purdue.edu} }

\maketitle

\begin{abstract}
This article presents a general framework for recovering missing dynamical systems using available data and machine learning techniques. The proposed framework reformulates the prediction problem as a supervised learning problem to approximate a map that takes the memories of the resolved and identifiable unresolved variables to the missing components in the resolved dynamics.  We demonstrate the effectiveness of the proposed framework with a strong convergence error bound of the resolved variables up to finite time and numerical tests on prototypical models in various scientific domains. These include the 57-mode barotropic stress models with multiscale interactions that mimic the blocked and unblocked patterns observed in the atmosphere, the nonlinear Schr\"{o}dinger equation which found many applications in physics such as optics and Bose-Einstein-Condense, the Kuramoto-Sivashinsky equation which spatiotemporal chaotic pattern formation models trapped ion mode in plasma and phase dynamics in reaction-diffusion systems. While many machine learning techniques can be used to validate the proposed framework, we found that recurrent neural networks outperform kernel regression methods in terms of recovering the trajectory of the resolved components and the equilibrium one-point and two-point statistics. This superb performance suggests that a recurrent neural network is an effective tool for recovering the missing dynamics that involves approximation of high-dimensional functions.
\end{abstract}



\section{Introduction}

{T}he problem of missing dynamics is ubiquitous in any scientific domain that concerns with prediction through computational models. This long-standing problem has been posted under various names, including model error, sub-grid scale parameterization, closure modeling \cite{crommelin2008subgrid,kbm:10,bh:14,hmm:14,kwasniok2012data,lu2016comparison,lu2017accounting,mh:13,wilks2005effects}.  Another relevant topic of broad interest is the reduced-order modeling whose ultimate goal is to systematically deduce a computationally efficient model to predict the evolution of the resolved variables when the full underlying model is too expensive to solve \cite{givon2004extracting,majda1999models,majda2001mathematical,eweinan2007heterogeneous,chk:02,chorin2007problem,gouasmi2017priori,hl:15,kondrashov2015data}. In our context, the proposed framework adopted here does not require any knowledge of the full equations that govern the underlying dynamical systems. The proposed approach that we consider assumes that only the dynamical components corresponding to the resolved variables are given. As in \cite{jh:19}, the missing components will be learned from a historical time series of the resolved variables and the identifiable unresolved variables, where the latter serves as feedback from the interaction of the resolved and unresolved scales. In essence, the proposed approach is to learn a dynamical model for the identifiable unresolved variables that depend on both the resolved and unresolved variables.

The success of deep learning as a supervised learning algorithm has drawn tremendous interest on reduced-order modeling applications. A closely related approach to the modeling framework in this paper is presented in \cite{pan2018data}. They proposed a Feedforward Neural Network (FNN) as a representation of the dynamics of the irrelevant variables. These authors also provide a linear control theory perspective to justify the identifiability of their dynamical representation on a class of nonlinear systems with a dual linear closure. In this article, we consider the closure modeling framework with a nonparametric formulation and provide a strong convergence error bound of the estimation of the resolved variables for the first time. For discrete dynamics obtained from a temporal discretization of differential equations, we found that when the unresolved variables are fully identifiable, the error rate is $\mathcal{O}({T^2\Delta^2\epsilon})$, under mild assumptions. Here, $T>0$ denotes the prediction time index, $\Delta$ denotes the discrete time step, and $\epsilon>0$ denotes the approximation error of the missing dynamics. Recalling the theory of nonparametric regression \cite{stone1982optimal}, if the missing dynamics is a function of a Sobolev class, $H^\beta$, where $\beta>0$ denotes the regularity parameter, the learning rate $\epsilon$ of any nonparametric regression algorithm with i.i.d data has an optimal global error rate of an order $\epsilon = \mathcal{O}(N^{-\frac{\beta}{2\beta+d}})$, where $N$ denotes the length of training data and $d$ denotes the dimension of the domain of the function. Hence, nonparametric regression algorithms suffer from the curse of dimensionality unless when the regularity parameter $\beta=O(d)$. However, even for the case of $\beta=O(d)$, there are no efficient tools to carry out the computation for high-dimensional problems.

Fortunately, recent advances in the theoretical analysis of deep neural networks show that they can avoid the curse of dimensionality in terms of approximation error in both the case of sufficiently smooth functions \cite{barron1993,Hadrien,HadrienYangDu,HadrienYang,Shen3}, and even for continuous functions \cite{Shen4}. Also, there is no curse of dimensionality of deep neural networks in terms of generalization error when the target functions admit sufficient smoothness \cite{Weinan2018}, when the data are sampled on a low-dimensional manifold \cite{Masaaki}, or in the case of classification functions \cite{DBLP:journals/corr/abs-1902-01384}. While the generalization error for deep neural networks on general functions is an open problem, empirical numerical evidence has indicated that deep neural networks together with their stochastic training algorithms (e.g., batch-based stochastic gradient descent) are automatic tools that can identify the ``low complexity'' of the underlying systems, e.g., the smoothness or the low-dimensional domain that leads to no curse of dimensionality  \cite{barron1993,HadrienYang,HadrienYangDu,ShenYangZhang2,Shen3,Weinan2018}. In particular, recurrent neural networks as a special case of deep neural networks has the potential to avoid the curse of dimension when learning and predicting discrete dynamical systems with low complexity structures.

While the closure modeling framework can be numerically realized using any approximation/regression methods, we will consider a special type of recurrent neural networks called the Long-Short-Term-Memory (LSTM) \cite{Hochreiter_1997}. We will show that this approach can overcome the curse of dimension suffered by the standard nonparametric regression method such as the kernel mean embedding approximation used in \cite{jh:19}. Our choice of using the LSTM is also encouraged by the success of it in recent closure modeling applications as proposed in \cite{ma2018model,vlachas2018data,maulik2019time}. We should stress that these existing approaches \cite{ma2018model,vlachas2018data,maulik2019time} share a similarity, that is, they specify the closure model as a function of only the memory of the resolved variables and motivate their framework using a heuristic connection with the Mori-Zwanzig formalism \cite{mori:65,zwanzig:61}. In contrast,  we will demonstrate that it is critical for the closure model to also depend on the memory of the identifiable unresolved variables in addition to the resolved components. We will demonstrate the effectiveness of our framework on several tough prototype complex systems that arise in geophysical fluid dynamics, optics, quantum fluid such as Bose-Einstein-Condensate, and plasma physics, in addition to theoretical justification.

The rest of the paper is organized as follows. In Section~\ref{intro}, we formulate the problem, provide a simple example to elucidate the proposed formulation, and discuss the theoretical aspect of the proposed approach. In Section~\ref{lstm}, we provide a short discussion on the numerical aspect of LSTM as a special class of RNN. In Section~\ref{numerics}, we report the numerical results on three prototypical examples of different types of dynamical systems. In Section~\ref{summary}, we close the paper with a summary. The technical proofs of the theoretical result will be reported in Appendices.

\section{Data-Driven Modeling for Missing Dynamics}\label{intro}

Throughout this paper, we will describe the closure modeling approach in the context of discrete maps that naturally arise from numerical discretization of partial or ordinary differential equations. Stochastic differential equation will be discussed in the following sections. Let the resolved, $x_t\in\mathcal{X}$, and unresolved, $y_t\in\mathcal{Y}$, variables be the solution of the following deterministic discrete dynamical systems,
\begin{align}
\begin{split}
x_{t+1} = \mathcal{F}(x_t,y_t), \quad y_{t+1} = \mathcal{G}(x_t,y_t), \label{fulldyn}
\end{split}
\end{align}
given initial conditions $x_0,y_0$.

\begin{assumptions}\label{assumption1} Furthermore, we assume that the full system is ergodic with a unique invariant measure $\mu$. For the measure space $(\mathcal{X}\times\mathcal{Y},\mathcal{B},\mu)$, where $\mathcal{B}$ denotes the $\sigma$-algebra of set $\mathcal{X}\times\mathcal{Y}$, the map defined by
$\Phi:=(\mathcal{F},\mathcal{G})$ is invariant under measure $\mu$. That is, $\mu(\Phi^{-1}(B))=\mu(B)$, for $B\in\mathcal{B}$.

\end{assumptions}


Under this assumption, the missing dynamics problem is to predict $\{x_t:t\in\mathbb{N}\}$ and its statistics, such as, the mean, covariance, and auto-correlation functions, given only partial dynamics, $\mathcal{F}$. Basically, the absence of $\mathcal{G}$ means that we are missing the unresolved dynamics for $y$. Our goal is to reconstruct the missing dynamics in \eqref{fulldyn} from the given historical time series, $\{x_t,\theta_t\}_{t=1,\ldots,N}$, where $\theta_t:=\Theta(x_t,y_t)$ is the identifiable unresolved variable. Here, the observables $X:\mathcal{X}\times\mathcal{Y} \to \mathcal{X}$ and $\Theta:\mathcal{X}\times\mathcal{Y} \to \mathcal{W}$ are random variables defined as $X(x_t,y_t)=x_t$ and $\Theta(x_t,y_t)=\theta_t$, respectively. 

In particular, $\theta_t$ is the component of the unresolved variables that can be identified from $\mathcal{F}(x,y):=\mathcal{F}(x,\Theta(x,y))$ in \eqref{fulldyn} and observations $\{x_t\}$.  With this definition, we abuse the notation $\mathcal{F}$ to emphasize its dependence on $\theta$ (and suppress its dependence on $y$), so that in general, $\theta\neq y$ (see e.g, \eqref{Fadditive}). This restriction is motivated by practical constraints where only the resolved variables are observed. For example, any $\mathcal{F}$ can be decomposed into,
\BEA
\mathcal{F}(x_t,y_t) = \bar{\mathcal{F}}(x_t) +  \Theta(x_t,y_t) =  \bar{\mathcal{F}}(x_t) + \theta_t,\label{Fadditive}
\EEA
for some $\bar{\mathcal{F}}$ that involves only the resolved variables and the remainder term is the identifiable unresolved variable. Given $\{x_t\}$, one can extract a time series of $\theta_t = x_{t+1} - \bar{\mathcal{F}}(x_t)$ by a direct subtraction. In the case when the observed $x_t$ is noisy, one can also use appropriate filtering methods \cite{hmm:14,bh:16jcp}. We should point out that our formulation below also holds even if $\theta$ depends only on the unresolved variables $y$, so long as the time series of $\theta$ is available. In this case, the identifiability of $\theta$ will be related to the notion of reachability/observability in the control theory (e.g. see~Chapters 3 and 4 of \cite{chui2012linear}). In the numerical simulations shown below, we assume that a historical timeseries of $\{x_t,\theta_t\}_{t=1,\ldots,N}$ is available to us.

Our goal is to predict $\{x_t:t\in\mathbb{N}\}$ and its statistics, such as the mean, covariance, and auto-correlation functions, with the above constraints.  We also aim to reconstruct the missing dynamics in \eqref{fulldyn}.
Define $\bm{z}_{t,m}:=(\bm{x}_{t-m:t},\bm{\theta}_{t-m:t})\in \mathcal{Z}$ with $\bm{x}_{t-m:t}:=(x_{t-m},x_{t-m+1},\ldots ,x_{t})$ and $\bm{\theta}_{t-m:t}:=(\theta_{t-m},\theta_{t-m+1},\ldots ,\theta_{t})$ for some integer $m\geq 0$ which characterizes the memory length. We consider a general closure model of the following form,
\begin{align}
\begin{split}
\hat{x}_{t+1} = \mathcal{F}(\hat{x}_t,\hat{\theta}_t), \quad \hat{\theta}_{t+1} = \mathbb{E}^\epsilon[\Theta_{t+1}|{\bm{\hat{z}}_{t,m}}] +\xi_{t+1}, \label{approxdyn}
\end{split}
\end{align}
where $\hat{\cdot}$ is used to denote the numerical approximation of the corresponding variable in the closure model. In \eqref{approxdyn}, the notation $\mathbb{E}^\epsilon [\Theta_{t+1}|\cdot]:\mathcal{Z}\to\mathcal{W}$ is to denote an estimator to the target function $\mathbb{E} [\Theta_{t+1}|\cdot]:\mathcal{Z}\to\mathcal{W}$ with error $\epsilon>0$ in appropriate sense. Here, the random variable $\Theta_{t+1}:= S^{t+1}\circ\Theta$, where $S:L^2(\mu,\mathcal{W})\to L^2(\mu,\mathcal{W})$ denotes the associated Koopman operator that is defined as, $Sf := f\circ \Phi$, for all function $f:\mathcal{X}\times\mathcal{Y}\to\mathcal{W}$ of the Hilbert space $L^2(\mu,\mathcal{W})$, equipped with the inner product $\langle f, g\rangle_{L^2(\mu)} = \int_{\mathcal{X}\times\mathcal{Y}} \langle f(x,y),g(x,y)\rangle_{\mathcal{W}}d\mu(x,y)$. Here, the map $\Phi:=(\mathcal{F},\mathcal{G})$ denotes the full dynamics. With this definition, one can verify that evaluating $\Theta_{t+1}$ on initial condition $(x_0,y_0)$ produces $\Theta_{t+1}(x_0,y_0) = S^{t+1}\Theta(x_0,y_0) = \Theta\circ \Phi^{t+1}(x_0,y_0) = \Theta(x_{t+1},y_{t+1})=\theta_{t+1}$.

While the proposed framework suggests that one can use any supervised learning method to construct an estimator $\mathbb{E}^\epsilon  [\Theta_{t+1}|\cdot]$ (that can be in the form of parametric or nonparametric as we shall discuss in subsection~\ref{sec22}), to guarantee an accurate estimation of the path $x_t$, one should consider a consistent estimator, that is,  $\mathbb{E}^\epsilon  [\Theta_{t+1}|\cdot]\to \mathbb{E} [\Theta_{t+1}|\cdot]$ as $\epsilon\to 0$ in $L^2$ sense, as we shall discussed below in subsections~\ref{sec21} and \ref{sec:partialobs}. Another question that will be clarified in these two subsections is to which target function does the proposed estimator converge to. In other words, what is the target function  $\mathbb{E}[\Theta_{t+1}|\cdot]$? As we shall see later, this depends on the observable $\Theta$.

In \eqref{approxdyn}, the noise $\xi_t$ is added to account for the residual due to misspecification of hypothesis space for the target function $\mathbb{E}[\Theta_{t+1}|\cdot]$. In fact, we shall see from our numerical experiments below that this additional noise is not needed for the deterministic problems when LSTM is used as the estimator $\mathbb{E}^\epsilon[\Theta_{t+1}|\cdot]$. For simplicity, we only consider $\xi_t\sim\Xi$ to be Gaussian with variance,
\BEA
\mathbb{E}[\Xi^2] := \mathbb{E}\Big[(\Theta_{t+1}-\mathbb{E}^\epsilon[\Theta_{t+1}|Z_{t,m}] )^2 \Big],\label{noisebalancedmain}
\EEA
where we used the notation $Z_{t,m}$ to denote the random variables associated to the realizations $\bm{z}_{t,m}$.
Here, the variance will also be estimated by Monte-Carlo approximation to \eqref{noisebalancedmain} using the historical solutions of the ergodic system in \eqref{fulldyn}.

To summarize, the closure model is reformulated as a supervised learning problem to learn the map $\hat{\bm{z}}_{t,m}\mapsto \mathbb{E}[\Theta_{t+1}|{\bm{\hat{z}}_{t,m}}]$ and the variance $\mathbb{E}[\Xi^2]$ of the residual. In the next subsection, we discuss a simple case for which $\Theta(x,y) = y$. Subsequently, we discuss the notion of parametric and nonparametric estimators, $\mathbb{E}^\epsilon[\Theta_{t+1}|\cdot]$. We finally close this section with a study on the general case of $\Theta$, which constitutes a more complicated target function, $\mathbb{E}[\Theta_{t+1}|\cdot]$, obtained using a discrete Dyson formula.

\subsection{Fully identifiable unresolved variables}\label{sec21}
To give an intuition, suppose that the entire unresolved variables can be identified, that is, $\Theta(x,y)=y$. Since $\theta_t = y_t$, let us replace the notation $\Theta_{t+1}$ with $Y_{t+1}:=S^{t+1}\circ Y$ defined with the random variable $Y:\mathcal{X}\times\mathcal{Y} \to \mathcal{Y}$ in $L^2(\mu,\mathcal{Y})$ with $Y(x_t,y_t) = y_t$. In this case, it is clear that the target function is nothing but the full missing dynamics, namely, $\mathbb{E}[\Theta_{t+1}|{\bm{z}_{t,0}}] = \mathbb{E}[Y_{t+1}|x_t,y_t] =\mathcal{G}(x_t,y_t)$ such that one can rewrite \eqref{fulldyn} as,
\BEA
x_{t+1} = \mathcal{F}(x_t,y_t), \quad y_{t+1} =  \mathbb{E}[Y_{t+1}|x_t,y_t]. \label{fulldyn2}
\EEA

If $\mathbb{E}^\epsilon [Y_{t+1}|\cdot]$ is a consistent estimator of $\mathbb{E} [Y_{t+1}|\cdot]$ with variance error rate $\epsilon^2$, it is clear from
\eqref{noisebalancedmain} that $\mathbb{E}[\Xi^2]=\mathcal{O}(\epsilon^2)$. In this case, we can show that
\begin{theorem}
Let $\mathcal{F}$ and $\mathcal{G}$ be Lipschitz in $x$ and $y$. Let ${x}_{t+1}$ be the solutions of \eqref{fulldyn2} and $\hat{x}_{t+1}$ be the solutions of,
\begin{align}
\begin{split}
\hat{x}_{t+1} =\mathcal{F}(\hat{x}_t,\hat{y}_t),\quad \hat{y}_{t+1} = \mathbb{E}^\epsilon [Y_{t+1}|\hat{x}_{t},\hat{y}_t] +\xi_{t+1}, \label{approxdyn2}
\end{split}
\end{align}
under the same initial conditions $x_0=\hat{x}_0, y_0=\hat{y}_0$. Under the Assumption~\ref{assumption1}, if $\mathbb{E}^\epsilon [Y_{t+1}|\cdot] \to \mathbb{E} [Y_{t+1}|\cdot]$ as $\epsilon\to 0$ with variance error of order $\epsilon^2$, then
\BEA
\mathbb{E}\Big[\max_{t\in \{0,\ldots,T\}}|\hat{x}_{t} - x_{t}|\Big] =\mathcal{O}( a^T \epsilon).\label{bound1}
\EEA
for some constant $a>1$ that is independent of $T$ and $\epsilon$.
\end{theorem}
\begin{proof}
See Appendix~\ref{app:A}.
\end{proof}

This rather pessimistic error bound (exponential on $T$) is not surprising since the assumption on $\mathcal{F}$ and $\mathcal{G}$ is mild, Lipschitz continuous. This error bound is basically an extension of a classical result in dynamical system theory, the continuous dependence of the solutions of uniformly perturbed vector field (e.g., Chapter 17.5 in \cite{hirsch2012differential}). To obtain an improved result (such as polynomial on $T$), one should consider the structure on $\mathcal{F},\mathcal{G}$. For example, when the discrete dynamical system is a result of the Euler-Maruyama discretization on a system of stochastic differential equations,
\begin{align}
dx = f(x,y)\,dt + \sigma_x dW_{x,t}, \quad\quad dy = g(x,y)\,dt + \sigma_y dW_{y,t}\nonumber,
\end{align}
where $dW_{x,t}$ and $dW_{y,t}$ denote independent Gaussian white noises, we have:
\begin{align}\label{EM}
x_{t+1} =  x_t + f(x_t,y_t)\Delta + \sigma_x \Delta^{1/2} \xi_{x,t+1}, \quad\quad y_{t+1} =  y_t + g(x_t,y_t)\Delta + \sigma_y\Delta^{1/2} \xi_{y,t+1},
\end{align}
where $\Delta$ denotes the time step. Here, $\xi_{x},\xi_y \sim\mathcal{N}(0,\mathcal{I})$ are samples of the standard Gaussian white noises. When $g$ and $\sigma_y$ are unknown, we can directly estimate these terms and obtain a sharper error bound:
\begin{theorem}
{blue}Let $\mathcal{F}$ and $\mathcal{G}$ be Lipschitz in $x$ and $y$.
Let ${x}_{t+1}$ be the solutions of \eqref{EM} and $\hat{x}_{t+1}$ be the solutions of,
\begin{align}
\begin{split}
\hat{x}_{t+1} &= \hat{x}_t + f(\hat{x}_t,\hat{y}_t)\,\Delta + \sigma_x \Delta^{1/2}\xi_{x,t+1}, \\
\hat{y}_{t+1} &= \hat{y}_{t} + \Delta \mathbb{E}^\epsilon \Big[Y^\Delta_{t+1} |\hat{x}_{t},\hat{y}_t\Big] + \hat\sigma_y \Delta^{1/2} \xi_{y,t+1}, \label{approxdyn3}
\end{split}
\end{align}
under {the same} initial conditions $x_0=\hat{x}_0, y_0=\hat{y}_0$. Here, we have defined $Y^\Delta_{t+1}:=\frac{Y_{t+1}-Y_t}{\Delta}$ and the noise variance,
\BEA
\hat\sigma_{y}^2 \Delta := \mathbb{E}\Big[\Big(Y_{t+1}-Y_t-\Delta \mathbb{E}^\epsilon\Big[Y^\Delta_{t+1}|X_t,Y_t\Big]\Big)^2\Big]\nonumber
\EEA
is estimated from the training data. Suppose that the learning variance error rate is,
\BEA
\mathbb{E}[ ( \mathbb{E}[Y^\Delta_{t+1}| X_t, Y_t] - \mathbb{E}^\epsilon[Y^\Delta_{t+1}| X_t, Y_t] )^2 ] \leq C\epsilon^2,\nonumber
\EEA
for some $C>0$. Let $f$ and $g$ be Lipschitz continuous in $x$ and $y$, then,
\BEA
\mathbb{E}\Big[\max_{t\in \{0,\ldots,T\}}|\hat{x}_{t} - x_{t}|\Big] = \mathcal{O}(\epsilon T^2\Delta^2).\nonumber
\EEA
\end{theorem}
\begin{proof}
See Appendix~\ref{app:B}.
\end{proof}

This result suggests that the solution of the proposed approximate dynamics in \eqref{approxdyn3} strongly converges to that of \eqref{EM} up to finite time. The convergence rate suggests that one can expect a path-wise prediction with an accuracy of order learning rate error, $\epsilon$, up to order-one model unit time, $(T\Delta)^2 = \mathcal{O}(1)$. In other words, the length of accurate path-wise prediction is inversely proportional to the square root of the learning error rate, $T\Delta \approx \epsilon^{-1/2}$. Using a consistent learning algorithm, $\epsilon \to 0$, one can control the polynomial error growth. In the next section, we will use this error rate to estimate an accurate prediction time as a function of the number of training data, using a class of hypothesis space with well-known optimal (in the sense of bias-variance tradeoff) learning rate.



\subsection{Parametric versus nonparametric closure models} \label{sec22}

The essence of parametric closure modeling is to simply specify $\mathbb{E}^\epsilon$ in \eqref{approxdyn} with a specific choice of function $\mathcal{P}(\hat{\bm{z}}_{t,m};W)$ that depends on a finite-dimensional parameter $W$. The choice of ansatz $\mathcal{P}$ is usually based on physical intuition; see e.g., \cite{mh:13,hmm:14,bh:14,lu2016comparison,lu2017accounting}. Once the model is specified, the hyper-parameter $W$ can be obtained by regressing the pairs  $\{\bm{z}_{t,m},\theta_{t+1}\}$. Subsequently, the variance of $\Xi$ is estimated as in \eqref{noisebalancedmain}. In the later section, we will consider the Long-Short-Term-Memory model for $\mathcal{P}$.

In \cite{jh:19}, a non-parametric closure model is considered. Specifically, the conditional expectation in \eqref{approxdyn} is estimated using the kernel mean embedding of conditional distribution formula \cite{song2009,song2013}. In the implementation, they assume that $\mathbb{E}[\theta_{t+1}|\cdot]$ belongs to the reproducing kernel Hilbert space (RKHS) $\mathcal{H}\subset L^2(\mathcal{Z},q)$ with a well-defined Mercer-type kernel, constructed based on orthonormal basis functions $\{\varphi_k\}$ of this $L^2$-space, weighted by an arbitrary positive $q\in L^1(\mathcal{Z})$ \cite{jh:19}. For non-compact domain, such construction was studied in \cite{ZHL:19}. The advantage of RKHS beyond inheriting the orthogonality from the $L^2$-space is that any sequence of functions that converges in $\mathcal{H}$-norm also converges uniformly and any function in $\mathcal{H}$ inherits the regularity of the kernel. In fact, one can show that for appropriate choice of kernel, the RKHS $\mathcal{H}$ is dense in $C_b(\mathcal{Z})$ for compact $\mathcal{Z}$ \cite{steinwart2001influence} or $C_0(\mathcal{Z})$ for non-compact $\mathcal{Z}$ \cite{ZHL:19}.  Then, for any $\bm{z}_t\in\mathcal{Z}$, we can represent:
\BEA
\mathbb{E}[\theta_{t+1}|\bm{z}_t] = \sum_{k=0}^\infty c_k \varphi_k(\bm{z}_t),\label{RKHS}
\EEA
where the coefficients $c_k$ are precomputed from the available training data $\{\theta_t,x_t\}$; see \cite{jh:19} for details. The key point is that this nonparametric formulation turns the problem of choosing the closure model into a problem of constructing basis functions of a Hilbert space (i.e., choosing an appropriate hypothesis space). Since $\mathcal{H}$ is dense in the space of continuous bounded functions, then any bounded continuous parametric function $\mathcal{P}(\bm{z};W)$ can be consistently approximated by a truncation of the series expansion in \eqref{RKHS} for appropriate choice of basis functions. In this sense, choosing a parametric-based model is somewhat equivalent to specifying an appropriate RKHS space.

Now, let us demonstrate the effectiveness of this approach in the following simple yet nontrivial example. \vspace{5pt}

\noindent{\bf Example:} Consider a Langevin dynamics
\begin{align}
\begin{split}
dx &=y\,dt,  \label{Eqn:linear_Gauss} \\ dy & =(-\nabla V\left( x\right) -\gamma y)\,dt+\sigma _{y}\, dW_t,
\end{split}
\end{align}
where $x\in\mathbb{R}$ is the displacement, $y\in\mathbb{R}$ is the velocity, $V\left( x\right)
=-x^{2}/2+x^{4}/4$ is the double-well potential, $\gamma =1$ is the damping
coefficient, $dW$ is a standard Gaussian white noise, and $\sigma
_{y}=3\sqrt{2}/10$ is the driving strength. We observe
the trajectories of the variables $(x_t,y_t)$ at every time step $\Delta = 0.01$, obtained
using the Euler-Maruyama discretization scheme. In our previous notation, $\theta(y) = y$
and we consider a closure model in \eqref{approxdyn} with $\mathbb{E}[\theta_{t+1}|{\bm{\hat{z}}_{t,0}}] = \mathbb{E}\left[ Y_{t+1}|\hat{x}_{t},\hat{y}_{t}\right]$.  This example is nontrivial due to the transition state induced by the double-well potential $%
V\left( x\right) $. Note that only when the driving strength $\sigma _{y}$
is in a reasonable region, the transition state phenomenon can be observed
for this double-well potential system (see Fig. \ref{figure1} for
example).

In Fig.~\ref{figure1}, we compare the result obtained using the RKHS approximation in \eqref{RKHS} with the true trajectories and the statistics from the full model in \eqref{Eqn:linear_Gauss}. In this comparison, we apply the formula in \eqref{RKHS} with a tensor product of $50\times50$ Hermite polynomials.
We compare the prediction of the trajectory up to finite-time, marginal density of $x$, and auto-covariance function $\mathbb{E}[x_\tau x_0]$ of the true dynamics in \eqref{Eqn:linear_Gauss} with those from the closure models, trained using $N=5\times 10^4$ and $5\times 10^5$ data points. Notice that using large enough training data, we are not only accurately recovering the trajectory path-wise longer in time but also the density and auto-covariance function. Compare to the optimal learning rate $\epsilon = \mathcal{O}(N^{-\frac{\beta}{2\beta+d}})$ of \cite{stone1982optimal} for very smooth function with $\beta=\infty$, the empirical prediction length (as shown in Fig.~\ref{figure1}) is on the order of the theoretical prediction length $T\Delta = \epsilon^{-1/2}\approx 14.95$ for $N=5\times 10^4$ and is slightly longer than the conservative estimate $T\Delta = \epsilon^{-1/2}\approx 26.59$ for $N=5\times 10^5$. In Table~\ref{table1}, we also see the agreement of several statistics that are commonly used to characterize {metastable} dynamics. When the training data set is large, we see a relatively accurate estimation of the mean exit time $\bar{\tau}_0$ of a particle to escape one of the wells \cite{gdls:2014} and the reaction rate $\nu_{R}$, defined as the limit of $N_T/T$ as $T\to\infty$ where $N_T$ is the number of trajectories to escape a well in the time interval $T$ \cite{vanden2010transition}.

\begin{figure}
\minipage{1\textwidth}
\minipage{0.5\textwidth}
\subcaption{Trajectories ($N=50,000$)}
\vspace{-5pt}
\includegraphics[width=\textwidth]{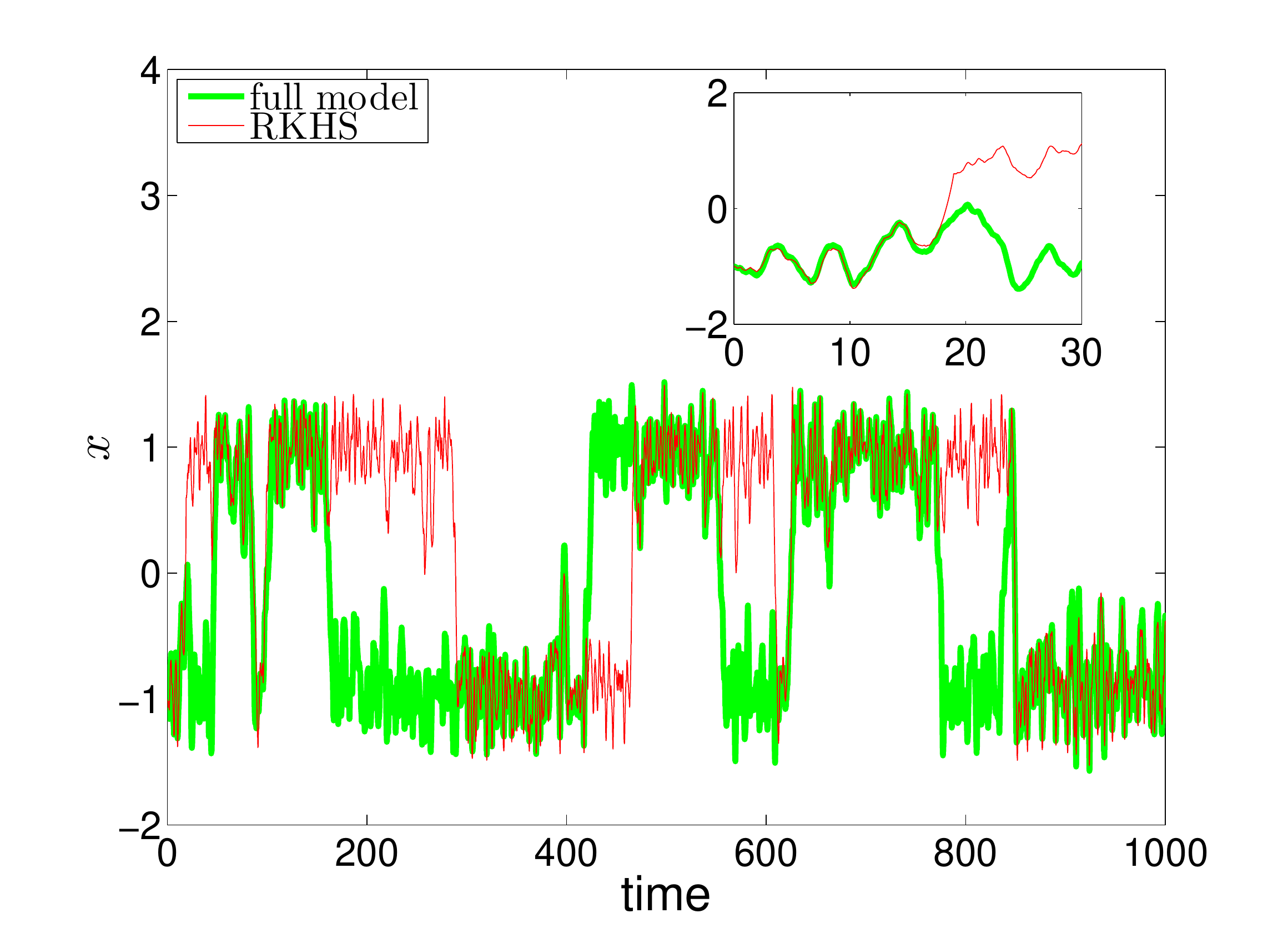}
\endminipage\hfill
\minipage{0.5\textwidth}
\subcaption{Trajectories ($N=500,000$)}
\vspace{-5pt}
\includegraphics[width=\textwidth]{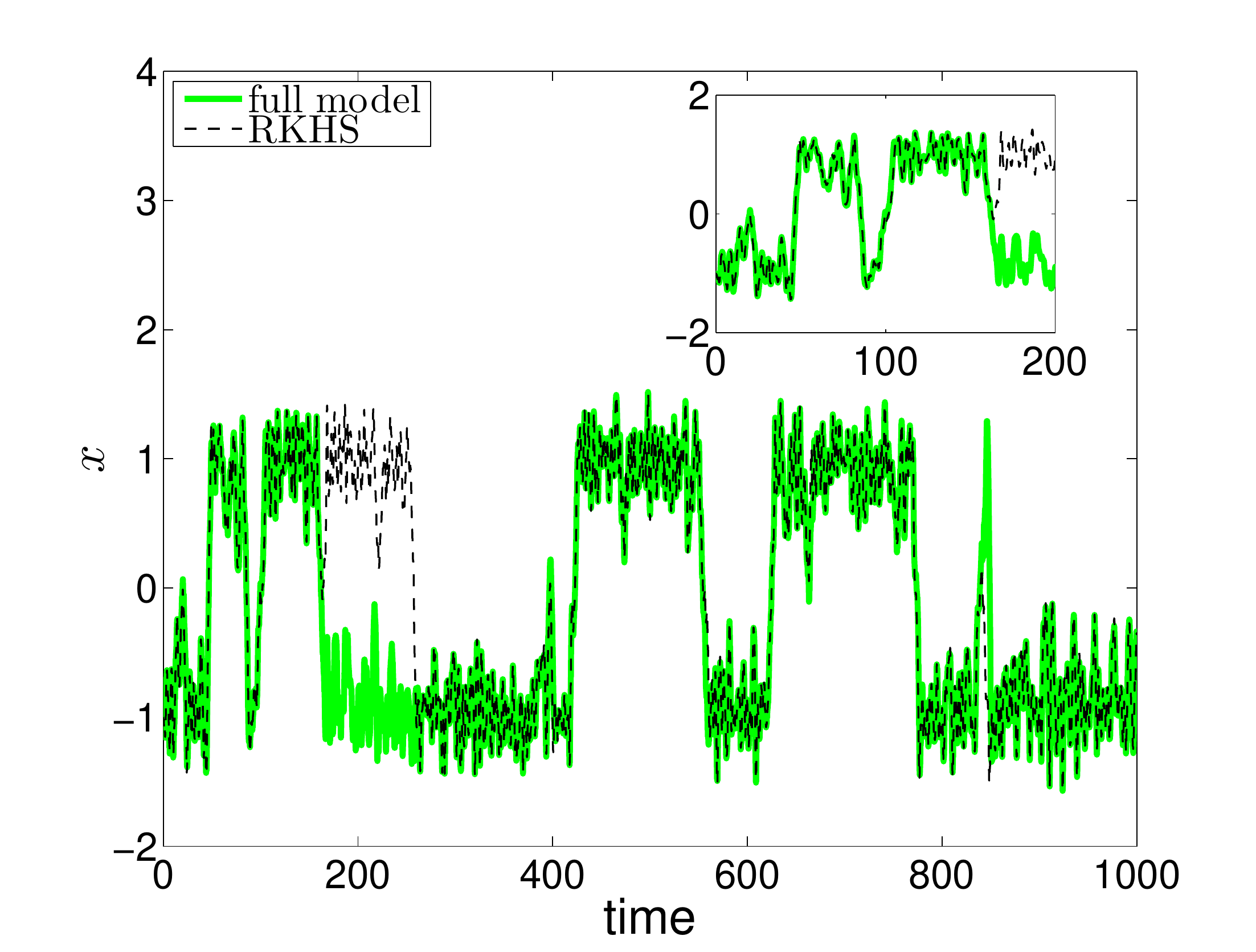}
\endminipage
\endminipage\hfill
\minipage{1\textwidth}
\minipage{0.5\textwidth}
\subcaption{Densities}
\vspace{-5pt}
\includegraphics[width=\textwidth]{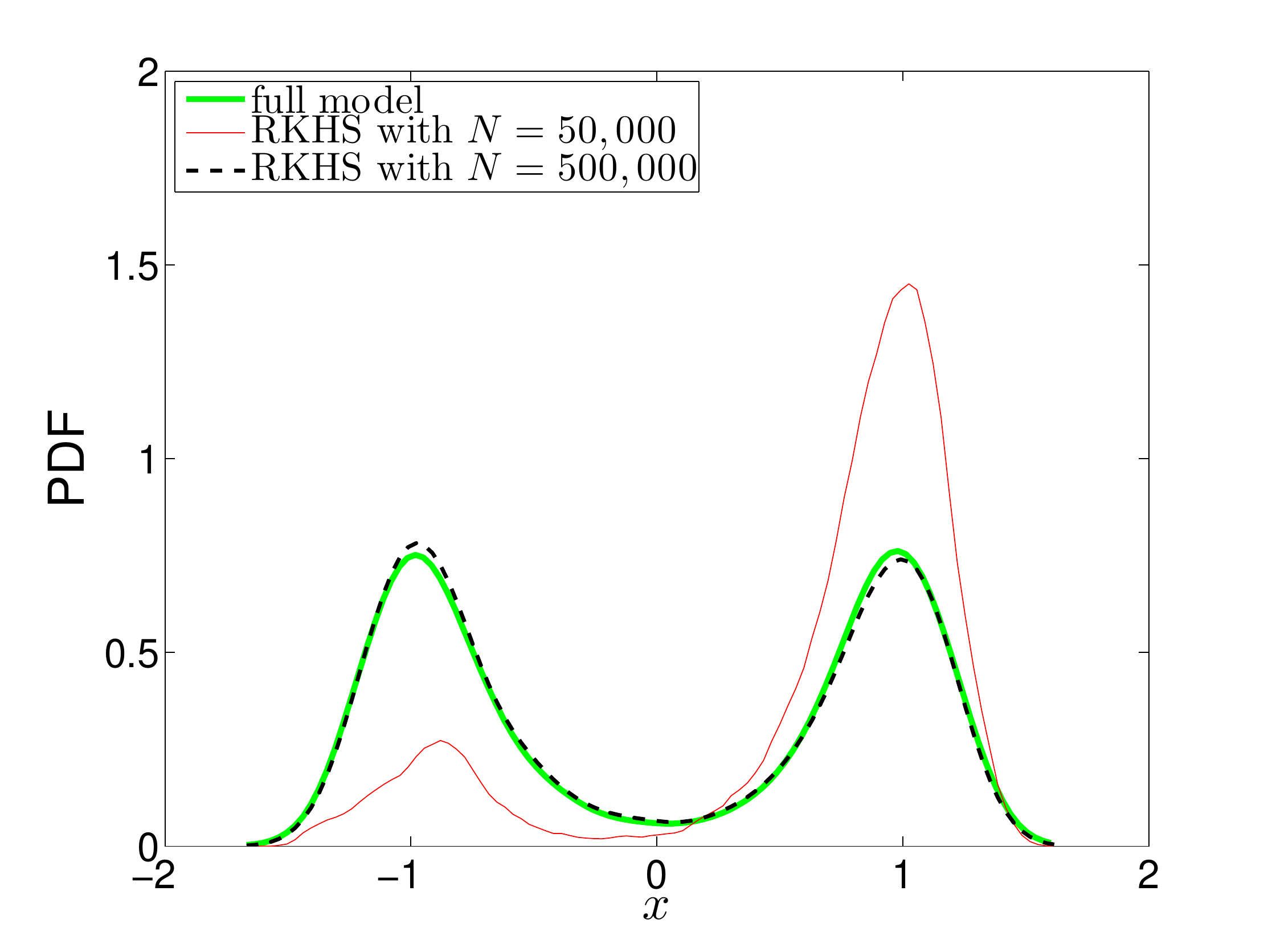}
\endminipage
\minipage{0.5\textwidth}
\subcaption{ACVs}
\vspace{-5pt}
\includegraphics[width=\textwidth]{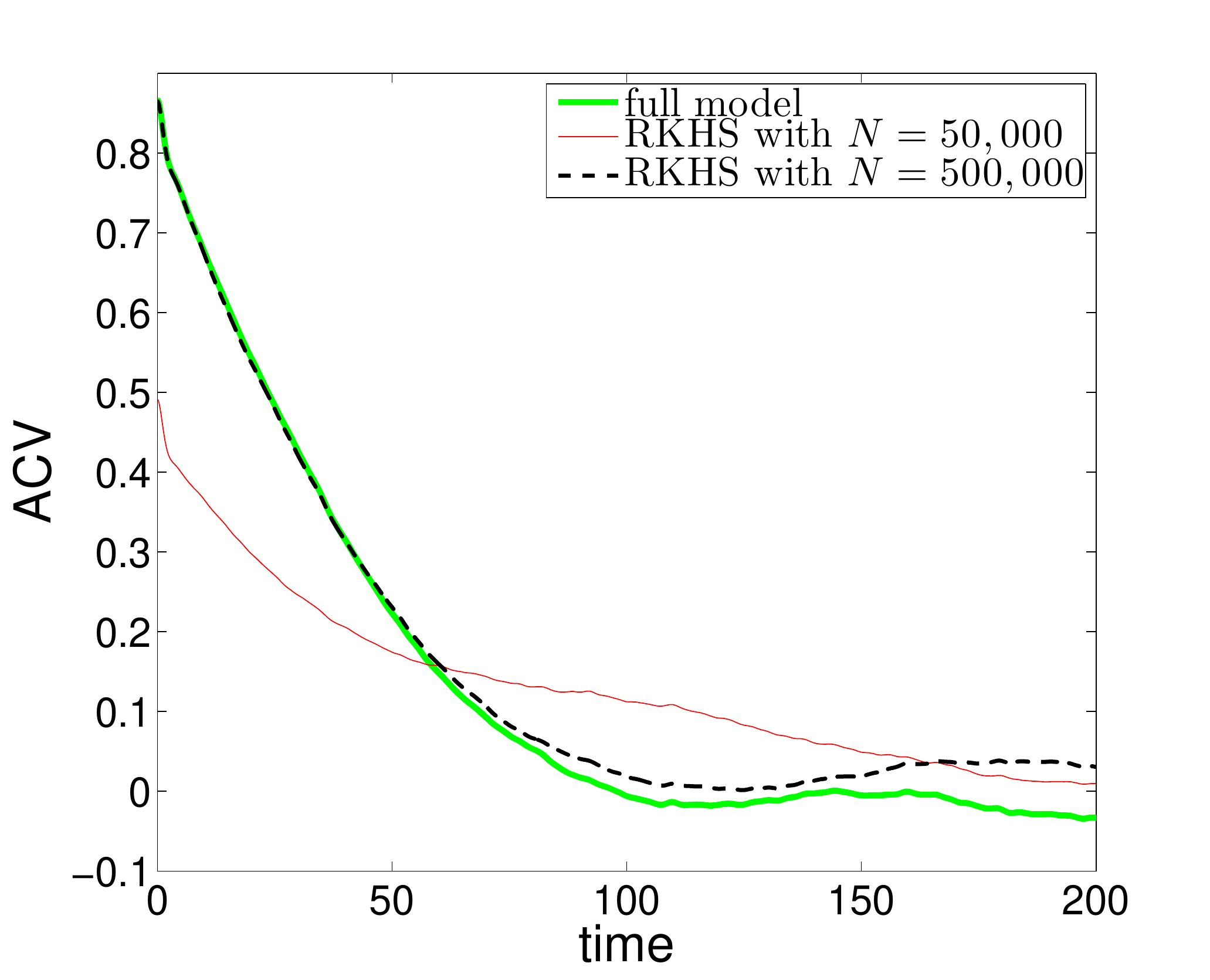}
\endminipage
\endminipage\hfill

\caption{The top panels display comparison of the
trajectories between the full model (green) and RKHS closure model using different size of training
dataset. The bottom panels compare
PDFs and auto-covariance functions (ACVs) among the full model (green), RKHS closure models.}
\label{figure1}
\end{figure}

\begin{table}
{\caption{Comparisons of mean exit time $\bar{\tau}_0$ and reaction rate $\nu_{R}$ between the full model and closure models.}\label{table1}}
\begin{center}
{\begin{tabular}{|c|c|c|c|}
\hline
& True & RKHS $N=50,000$ & RKHS $N=500,000$ \\
\hline
$\bar{\tau}_0$ & 99.2 & 69.1& 102.7 \\
$\nu_R$ & 0.0079 & 0.0040 & 0.0075 \\ \hline
\end{tabular}}
\end{center}
\end{table}

\subsection{Partially identifiable unresolved variables} \label{sec:partialobs}

While the closure model in \eqref{approxdyn2} is theoretically consistent and the example above shows a very promising approach, in real applications, the function $\Theta(x,y)\neq y$ since the unresolved variables, $y$, are usually not identifiable from the data $\{x_t\}$ and the map $\mathcal{F}(x_t,\theta_t)$. Even if the full data of $y$ are available, they are very high-dimensional relative to $\theta$. What is usually identified is $\theta$ in the sense of \eqref{Fadditive}, which yields the same dimensionality as the resolved variable $x$.

In this case, let us rewrite the underlying dynamics in \eqref{fulldyn} as a function of $(x,\theta)$. To do this, we consider the Koopman operator $S:\mathcal{H}\to\mathcal{H}$ defined as, $S \circ f = f \circ \Phi $, for $f \in \mathcal{H}$, a space of $(\mathcal{X}\times\mathcal{W})$-valued functions $f:\mathcal{X}\times\mathcal{Y}\to \mathcal{X}\times\mathcal{W}$, equipped with an inner product, $\langle f,g\rangle_{\mathcal{H}} := \int_{\Omega} \langle f(\omega),g(\omega)\rangle_{\mathcal{X}\times\mathcal{W}}\, d\mu(x)$, weighted by the invariant measure $\mu$. The main interest is of observable function $\pi\in \mathcal{H}$ defined as $\pi:=(X,\Theta)$, such that $\pi(x,y) = (X(x,y),\Theta(x,y))=(x,\theta), \forall (x,y)\in\mathcal{X}\times\mathcal{Y}$, which can be interpreted as a map that changes coordinate.
For identifiable unresolved component as defined in \eqref{Fadditive}, then $\mathcal{X}=\mathcal{W} = \mathbb{R}^{d_x}$ for $d_x$-dimensional real-valued observable; if $\mathcal{Y}=\mathbb{R}^{d_y}$, where $d_y\gg d_x$, then the random variable $\pi$ maps vectors in $\mathbb{R}^{d_x+d_y}$ into vectors in $\mathbb{R}^{2d_x}$. Let $P:\mathcal{H}\to\mathcal{H}$ be an orthogonal projection operator to a closed set of functions of $(x_0,\theta_0)$, namely $\mathcal{V}=\{f\in \mathcal{H}: f=g\circ \pi \,|\, g: \mathcal{X}\times\mathcal{W} \to \mathcal{X}\times\mathcal{W} \}$. To this end, we also define $Q:=I-P$ be the projection map to the orthogonal space $\mathcal{V}^\perp\subset\mathcal{H}$. We now rewrite the dynamics of the observables $(x,\theta)$ by applying the Dyson formula \cite{darve2009computing,lin2019data},
\BEA
S^{t+1} = \sum_{k=0}^t S^{t-k}PS(QS)^k + (QS)^{t+1},\nonumber
\EEA
on $\pi$, evaluated at initial condition $(x_0,y_0)$ with distribution $\mu$. The left-hand-side term, $S^{t+1} \pi(x_0,y_0) = \pi (\Phi^{t+1}(x_0,y_0)) = \pi(x_{t+1},y_{t+1})=(X(x_{t+1},y_{t+1}),\Theta(x_{t+1},y_{t+1}))= (x_{t+1},\theta_{t+1})$. Evaluating the first term on the right-hand-side on $\pi(x_0,y_0)$, we obtain $S^{t-k} PS(QS)^k \pi(x_0,y_0) =PS(QS)^k \pi(x_{t-k},y_{t-k})$. For $k=0$, we obtain the Markovian term, 
\BEA
PS\pi(x_{t},y_{t}) =P(\pi\circ\Phi(x_{t},y_{t}))  = P(\pi\circ(\mathcal{F}(x_{t},y_{t}),\mathcal{G}(x_{t},y_{t}))). \label{Markovian}
\EEA
In order to write \eqref{Markovian} in terms of $(x_t,\theta_t)$, we use the definition of $\mathcal{V}$. Particularly, since $PS\pi\in \mathcal{V}$, there exists a function $\bar{\Phi}_0: \mathcal{X}\times\mathcal{W} \to \mathcal{X}\times\mathcal{W}$ such that $PS\pi = \bar{\Phi}_0 \circ \pi$. Note that $PS$ is a linear operator whereas $\Phi_0$ is a function (possibly nonlinear) where they share the same range in $\mathcal{V}$. 
Let $\bar{\Phi}_0= (\bar{\mathcal{F}}_0,\bar{\mathcal{G}}_0)$, where the maps $\bar{\mathcal{F}}_0: \mathcal{X}\times\mathcal{W} \to \mathcal{X}$ and $\bar{\mathcal{G}}_0: \mathcal{X}\times\mathcal{W} \to \mathcal{W}$ denote the $x-$ and $\theta-$components of the Markovian term in \eqref{Markovian}, respectively. 
Since $\pi$ is identity in the $x$-direction, then the $x-$component of \eqref{Markovian} is simply $P\mathcal{F}(x_{t},y_{t})$. But since $\mathcal{F}(x,y) := \mathcal{F}(x,\theta)$ as we explained prior to Eq.~\eqref{Fadditive}, it is clear that $P\mathcal{F}(x_{t},y_{t}) = \mathcal{F}(x_{t},\theta_{t})$, which is nothing but $\bar{\mathcal{F}}_0(x_{t},\theta_{t})$.

For the memory terms, $k>0$, since $PS(QS)^k\pi \in \mathcal{V}$, there exists $\bar{\Phi}_k: \mathcal{X}\times\mathcal{W} \to \mathcal{X}\times\mathcal{W}$ such that $PS(QS)^k\pi   = \bar{\Phi}_k \circ \pi$, by the definition of $\mathcal{V}$. As before, we let $\bar{\Phi}_k= (\bar{\mathcal{F}}_k,\bar{\mathcal{G}}_k)$ such that the maps $\bar{\mathcal{F}}_k: \mathcal{X}\times\mathcal{W} \to \mathcal{X}$ and $\bar{\mathcal{G}}_k: \mathcal{X}\times\mathcal{W} \to \mathcal{W}$ denote the $x-$ and $\theta-$components of the memory functions, respectively. Since $Q\mathcal{F}=0$, the $x-$component of the dynamics has no memory term and therefore $\bar{\mathcal{F}}_k=0, k>0$. Putting all these together, we have
\begin{align}
\begin{split}\label{fulldyn3}
x_{t+1} &= \mathcal{F}(x_t,\theta_t), \\
\theta_{t+1} &= \bar{\mathcal{G}}_0(x_t,\theta_t) +  \sum_{k=1}^t \bar{\mathcal{G}}_k (x_{t-k},\theta_{t-k}) + (QS)^{t+1}(x_0,y_0).
\end{split}
\end{align}
Here the dynamics of $\theta$ inherits Markovian, memory and orthogonal terms, where the last term is usually regarded as noise with the randomness corresponding to the distribution of the initial condition, $\mu$. 

Conceptually, one can consider learning the entire dynamics of $\theta$ by considering the target function $\mathbb{E}[\Theta_{t+1}|\bm{z}_{t,t}]$, conditioned to the entire history of $\bm{z}_{t,t}=(\bm{x}_{0:t}, \bm{\theta}_{0:t})$, which is not practical. In applications, however, the length of the memory is often finite, $0<m<t$, and it can be estimated as shown in \cite{gouasmi2017priori,parish2017dynamic}. In fact, for nonlinear systems with linear dual closure (that is, the vector field is linear in terms of $\theta$), one can determine the minimum length of $m$ to guarantee the identifiability of $\theta_{t+1}$ from observed data of $\bm{z}_{t,m}:=(\bm{x}_{t-m:t},\bm{\theta}_{t-m:t})\in \mathcal{Z}$  \cite{pan2018data}. We should point out that while the linear dual closure condition is satisfied in the Langevin example above, it is not necessarily satisfied for all examples in Section~\ref{numerics}, that is, while the dependence of $\mathcal{F}$ on $\theta$ is linear as  in \eqref{Fadditive}, the dependence of $\bar{\mathcal{G}}_0$ on $\theta$ in \eqref{fulldyn3} may not be linear. Also, the second fluctuation-dissipation theorem \cite{Zwanzig2001} states that the time correlation of the orthogonal dynamics is proportional to the memory function. This suggests that if the memory term dissipates beyond $m$-lags, then the orthogonal dynamics should also depend on state variables of not longer than $m$-lags. With this in mind, let us rewrite \eqref{fulldyn3} as,
\BEA
\theta_{t+1} = \bar{\mathcal{G}}_0(x_t,\theta_t) +  \sum_{k=1}^m \bar{\mathcal{G}}_k (x_{t-k},\theta_{t-k}) + (QS)^{m+1}\pi(x_{t-m},y_{t-m}) + R_{t+1}, \nonumber
\EEA
where
\BEA
R_{t+1}:= \sum_{k=m+1}^t \bar{\mathcal{G}}_k (x_{t-k},\theta_{t-k}) + (QS)^{m+1}\left((QS)^{t-m}\pi(x_0,y_0)-\pi(x_{t-m},y_{t-m})\right),\label{intrinsicerror}
\EEA
denotes the modeling error due to finite memory assumption. Depending on the choice of learning algorithm and model (e.g., width or depth of the deep neural network), the estimator $\mathbb{E}^\epsilon[\Theta_{t+1}|\bm{z}_{t,m}]$ is usually found by minimizing the variance of $R_t\sim \Sigma$, where $\Sigma$ denotes the random variable of the error. From the learning perspective, this modeling error is usually minimized in the training phase with the hope that the corresponding generalization error will be not much larger than the training error. For the convenient of the analysis below, we first assume that the finite memory approximation holds such that $R_{t+1}=0$. 

For the convergence analysis, we require that the following assumption.

\begin{assumptions}\label{assumption2}
Let $PS$ and $QS$ be bounded linear operators defined with respect to functions $f\in C(\mathcal{X}\times\mathcal{Y})$. 
\end{assumptions}
\noindent With these assumptions, let $(x,\theta), (x',\theta')\in\mathcal{X}\times\mathcal{W}$, where $(x,\theta) = \pi(x,y)$, $(x',\theta') = \pi(x',y')$, and $y,y'\in\mathcal{Y}$. For $\pi\in C(\mathcal{X}\times\mathcal{Y})\cap \mathcal{H}$, notice that $|\bar{\Phi}_k(x,\theta)-\bar{\Phi}_k(x',\theta')| = |PS(QS)^k\pi(x,y) - PS(QS)^k\pi(x',y')|\leq |PS(QS)^k| |\pi(x,y)- \pi(x',y')| \leq C  |(x,\theta)- (x',\theta')|$, where the $|\cdot|$ denotes the appropriate uniform norms. Therefore, this technical assumption says that $\bar{\Phi}_k$ are Lipschitz continuous on $x$ and $\theta$, which is needed to bound the errors in term of $x$ and $\theta$. Notice that if $P$ and $\pi$ are both identity maps and $k=0$, this assumption is equivalent to saying $\Phi=(\mathcal{F},\mathcal{G})$ is Lipschitz in $x$ and $y$, which is assumed in both theorems in previous section. For the decomposition in \eqref{Fadditive}, where $\mathcal{F}$ is linear in $\theta$, this assumption means that $\mathcal{F}$ is also Lipschitz in the $x$.

\begin{theorem}\label{thm3}
Let ${x}_{t+1}$  be the solution of
\begin{align}
\begin{split}\label{MZtrue}
x_{t+1} &= \mathcal{F}(x_t,\theta_t), \\ \theta_{t+1} &= \bar{\mathcal{G}}_0(x_t,\theta_t) +  \sum_{k=1}^m \bar{\mathcal{G}}_k (x_{t-k},\theta_{t-k}) + (QS)^{m+1}\pi(x_{t-m},y_{t-m}) = \mathbb{E}[\Theta_{t+1}|\bm{z}_{t,m}],
\end{split}
\end{align}
where $\pi \in C(\mathcal{X}\times \mathcal{Y})\cap \mathcal{H}$ and the Assumption~2 is satisfied.
Let $\hat{x}_{t+1}$ be the solution of
\BEA
\hat{x}_{t+1} = \mathcal{F}(\hat{x}_t,\hat{\theta}_t), \quad \quad \hat{\theta}_{t+1} = \mathbb{E}^\epsilon[\Theta_{t+1}|\bm{\hat{z}}_{t,m}] +  \xi_{t+1},\label{MZapprox}
\EEA
under {the same} initial conditions $x_{-m:0}=\hat{x}_{-m:0}, \theta_{-m:0}=\hat{\theta}_{-m:0}$. In \eqref{MZapprox}, the noise amplitude $\xi_t\sim\Xi$ is estimated according to \eqref{noisebalancedmain} using the training data. Under the Assumptions~\ref{assumption1}, if $\mathbb{E}^\epsilon [\Theta_{t+1}|\cdot] \to \mathbb{E} [\Theta_{t+1}|\cdot]$ as $\epsilon\to 0$ with variance error of order $\epsilon^2$, then
\BEA
\mathbb{E}\Big[\max_{t\in \{0,\ldots,T+1\}}|\hat{x}_{t} - x_{t}|\Big]  = \mathcal{O}(a^T \epsilon)\nonumber
\EEA
where $a>1$ is a constant that is independent of $T$ and $\epsilon$. 
\end{theorem}
\begin{proof}
See Appendix~\ref{app:C}.
\end{proof}

We should point out that if the underlying dynamic solves the full Mori-Zwanzig equation in \eqref{fulldyn3} and the approximate dynamic in \eqref{MZapprox} commits modeling errors $R_{t+1}$ with error variance of order $\epsilon_{m}^2$, then the error bound above becomes $\mathcal{O}(a^T(\epsilon+\epsilon_m))$. The only change in the proof in Appendix~\ref{app:C} is that the Eqn.~\ref{thm1theta} has an additional order-$\epsilon_m$ term due to the model error.

For dynamical systems driven by stochastic noises, one can rewrite the full dynamics as an autonomous dynamical system by augmenting $(x_n,y_n)$ with the entire history of the noises. See \cite{kunita1997stochastic} for the details or the Appendix of \cite{lin2019data} for the key idea. Subsequently, one can apply the Dyson formula to the resulting autonomous dynamics, defined on appropriate state space that includes $\mathcal{X}\times\mathcal{Y}$ and the space of the history of the noises, and derive an analogous representation as in \eqref{fulldyn3}. We suspect that the result is not different from that in Theorem~\ref{thm3} and thus we will not pursue this derivation.

Again, this error rate is rather loose with unknown coefficients $a>1$ since no other assumptions are included in $\mathcal{F}, \mathcal{G}$. One might achieve an improved error rate by analyzing the eigenvalue problem corresponding to the autoregressive model of order-$m$, which bounds the dynamical equation for the errors between \eqref{fulldyn3} and \eqref{approxdyn}. Another plausible approach is to consider a class of dynamical systems with spatially short-range interaction as studied in \cite{chen2019spatial}, which will require further investigation.

\section{Deep Learning via Long-Short-Term-Memory}\label{lstm}

As a nonlinear type parametric regression method, deep learning outperforms kernel methods including the RKHS approach in terms of generalization error when the target functions are sufficiently smooth \cite{Weinan2018}. Though it is theoretically unclear whether there are any advantages of using deep learning over other nonparametric regression methods for general continuous functions in terms of overcoming the curse of dimension \cite{barron1993,Hadrien,HadrienYangDu,HadrienYang}, deep learning has practical advantages over the RKHS approaches. A significant challenge with the RKHS approximation in \eqref{RKHS} is that there is no a priori guideline for choosing the appropriate hypothesis space. If the orthogonal basis is used, it is practically difficult to even construct these basis functions on very high-dimensional variables $\bm{z}\in\mathcal{Z}$. On the other hand, if arbitrary radial functions are used as a basis, the evaluation of the resulting model on a new point $\bm{z}$ requires evaluating the basis functions on $\|\bm{z}-\bm{z}_{t,m}\|$ for all training data $t=1,\ldots, N$, predicting with \eqref{approxdyn} becomes too costly since we need to evaluate the conditional expectation in \eqref{RKHS} on a new point in each iteration. In contrast, deep learning as a nonlinear parametric regression method is not hampered by these issues, since it is practically just a nonlinear interpolation technique using a composition of nonlinear activation functions and linear transforms. Of course, the main issue with nonlinear regression is whether one can obtain the minimum on such a non-convex optimization problem in the training phase. Recent advances in optimization theory show that simple gradient descent can identify a local minimizer with an arbitrarily small loss and a generalization error without the curse of dimensionality when the network size is sufficiently large for classification problems \cite{DBLP:journals/corr/abs-1902-01384}.  Though there is no existing optimization theory that guarantees good local minimizers in general settings, motivated by many positive numerical results shown in other closure modeling approaches \cite{vlachas2018data,maulik2019time}, we will consider realizing the closure model in \eqref{approxdyn} using recurrent neural networks.

As a special case of recurrent neural networks, Long-Short-Term-Memory (LSTM) is capable of learning multi-scale temporal effects and hence is adopted in our method. The computational flow of the LSTM consists of a sequence of computational cells, each of which is
\begin{eqnarray*}
f_t &=& \sigma\circ NN (h_{t-1},\bm{z}_t;W_f),\quad i_t=\sigma\circ NN (h_{t-1},\bm{z}_t,W_t)  \\
o_t &=& \sigma\circ NN (h_{t-1},\bm{z}_t;W_o),\quad \tilde{C}_t=\text{tanh}(NN(h_{t-1},\bm{z}_t;W_T)),\\
C_t &=& f_t\otimes C_{t-1}+i_t\otimes \tilde{C}_t,\quad h_t=o_t\otimes \text{tanh}(C_t),
\end{eqnarray*}
where $\sigma$ denotes the sigmoid function, $\otimes$ is the pointwise product, and $NN$ denotes a fully connected network which stacks layers of linear transformation and nonlinear activation function. See Fig.~\ref{fig:lstmcell} (left) for an illustration of an LSTM cell. For simplicity, let us denote the above flow as $(h_t,C_t)=\mathcal{P}(\bm{z}_t,h_{t-1},C_{t-1};W)$ with parameters $W$, inputs $(\bm{z}_t,h_{t-1},C_{t-1})$, and outputs $(h_t,C_t)$. LSTM cells can be applied compositionally and we denote the LSTM sequence with $m+1$ cells as $(h_{m+1},C_{m+1})=\mathcal{P}_m(\{\bm{z}_t\}_{t=1}^{m+1},h_0,C_0;W)$ (see Fig.~\ref{fig:lstmcell} (right) for an illustration).

\begin{figure*}
  \centering
  \begin{tabular}{cc}
  \includegraphics[width=0.3\textwidth]{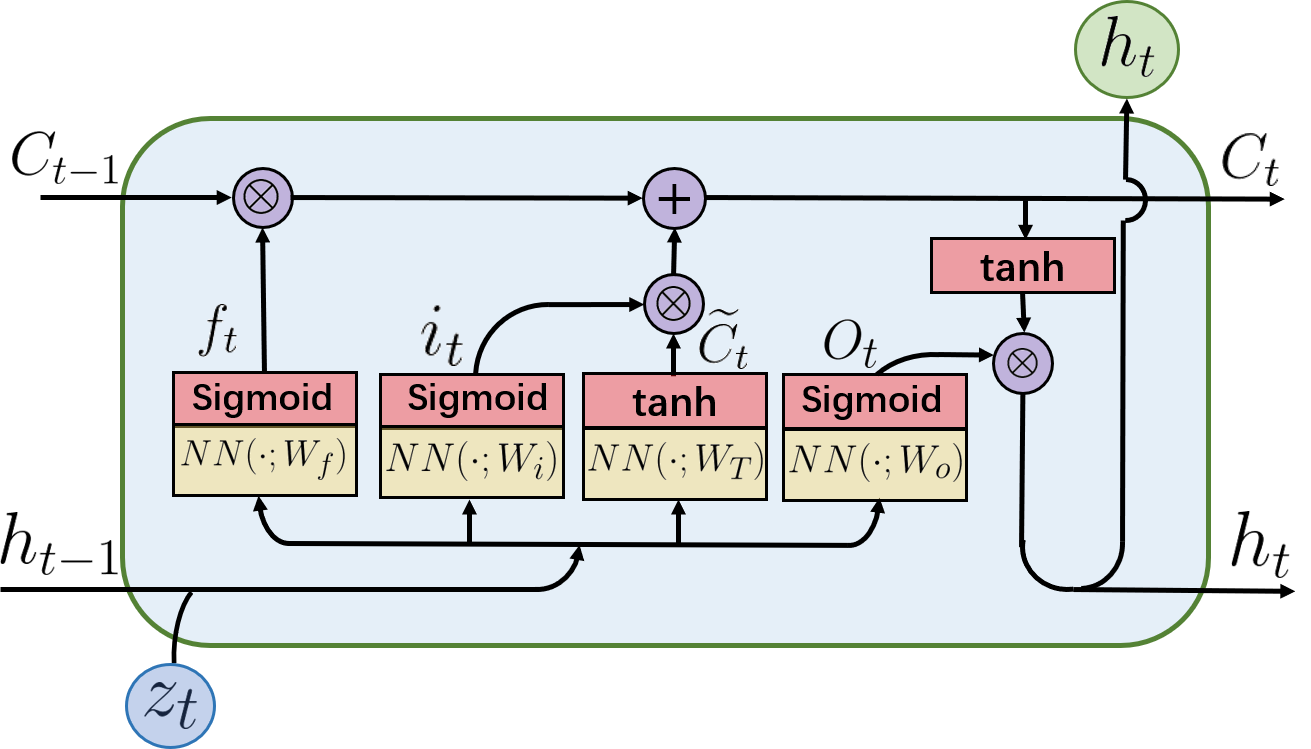} & \includegraphics[width=0.5\textwidth]{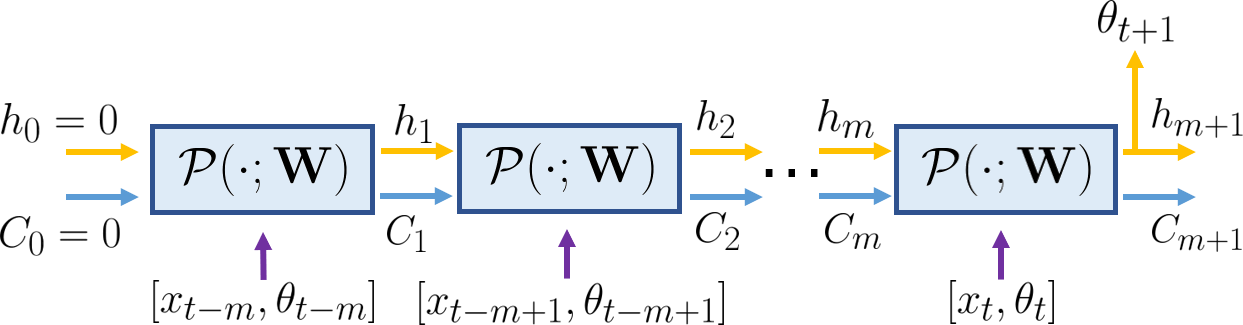}
  \end{tabular}
  \caption{Left: the basic computational flow of a LSTM recurrence.  $+$ and $\bigotimes$ are element-wise addition and multiplication respectively. Right: a sequence of LSTM cells applied compositionally. }\label{fig:lstmcell}
 \end{figure*}

Now let us apply the LSTM to approximate the closure model in \eqref{approxdyn}
with the given training data $\{{\bm{z}}_{t,m},{\theta}_t\}$, where ${\bm{z}}_{t,m}:=({\bm{x}}_{t-m:t},{\bm{\theta}}_{t-m:t})\in \mathcal{Z}$.  We train an $(m+1)$-cell LSTM $(h_{m+1},C_{m+1})=\mathcal{P}_m({\bm{z}}_{t,m},h_0,C_0;W)$ with an input in the $j$-th cell as $({x}_{t-m+j-1},{{\theta}}_{t-m+j-1})$ such that $h_{m+1}$ predicts $\theta_{t+1}$ well. The parameters $h_0$ and $C_0$ are set to be $0$ for simplicity and $W$ is identified via minimizing a mean squared error (MSE) function specified below. In what follows, we adopt the notation $h_{m+1}=\mathcal{P}_m({\bm{z}}_{t,m};W)$ for simplicity. Define the MSE loss as
\begin{equation}\label{eqn:lossW}
\mathcal{L}(W):=\frac{1}{N-m-1}\sum_{s=m+1}^{N-1}(\mathcal{P}_m({\bm{z}}_{s,m};W) - {\theta}_{s+1})^2,
\end{equation}
i.e., we aim to identify a predictor $\mathcal{P}_m({\bm{z}}_{s,m};W) $ such that it can predict the $(s+1)$-th sample in the time series given $(m+1)$ preceding samples. Minimizing \eqref{eqn:lossW}  can be achieved efficiently via a mini-batch stochastic gradient descent (SGD) and backpropagation through time (BPTT) \cite{Mozer:1995:FBA:201784.201791,robinson:utility,WERBOS1988339}. Though the global minimizer of the above highly non-convex optimization might not be available, empirically numerical results \cite{Pascanu:2013} and partial theoretical analysis show that gradient-based algorithms are able to provide a minimizer $W^*$ with a reasonably good generalization capacity in the case of over-parametrized networks \cite{RNNconvergence,DBLP:journals/corr/abs-1902-01028,chen2019on}. Once $W^*$ has been identified, ${\theta}_{t+1}=\mathcal{P}_m({\bm{z}}_{t,m};W^*)$ is applied instead of the conditional expectation in \eqref{approxdyn}. \vspace{4pt}

The computational cost of the proposed framework mainly consists of two parts: training and evaluation. In the training part, suppose the SGD with a batch size $K$ is applied to minimize the loss in \eqref{eqn:lossW}, then the computational cost for evaluating an approximate loss and gradient via BPTT is essentially $O(Km)$ matrix-vector multiplications with a matrix of size $d_L\times d_L$, where $d_L$ is the dimension of the hidden layer of $NN$. According to the approximation theory of DNNs \cite{DBLP:journals/corr/abs-1709-02540,ShenYangZhang2}, $d_L$ is required to be larger than $d$; based on the optimization analysis of DNNs \cite{DBLP:journals/corr/abs-1710-10174,DBLP:journals/corr/abs-1811-04918,DBLP:journals/corr/abs-1902-01384}, a larger $d_L$ admits a simpler optimization problem in deep learning. In the numerical examples shown below, we empirically set $d_L=500$ for problems with dimensions $d=40,80$, and set $d_L=1000$ for a problem that involves $d=480$. Hence, the total computational cost is $O(Kmd_L^2)$ for one iteration in the SGD, which has been parallelized via GPU computing in standard machine learning packages. Recent theoretical analysis shows that the convergence of SGD is linear under mild conditions \cite{RNNconvergence}. Therefore, the upper bound of the iteration number in the SGD to guarantee a loss less than $\epsilon$ is $O(\log(\frac{1}{\epsilon}))$. In our numerical examples below, the total numbers of iterations required for accurate performance are $4\times 10^3$ for the first problem and $4\times10^4$ for the last two problems. In the evaluation part, the cost estimation for predicting one data sample with $m$ pairs of historical data is $O(md_L^2)$, which has been parallelized via GPU.

\section{Numerical Examples}\label{numerics}

In this section, we numerically demonstrate the effectiveness of our proposed closure framework on severely truncated dynamical systems in three prototypical applications. First, the topographic mean flow interaction that mimics the blocked and unblocked patterns observed in the atmosphere \cite{carnevale1987nonlinear,grote1999dynamic,mw:06,franzke2009systematic} is considered. In our test, we will use the stochastic version of the 57-mode model studied in \cite{qi2017low}. Second, the nonlinear Schr\"{o}dinger equation that finds many applications in optics and Bose-Einstein-Condensate (see the references in \cite{kevrekidis2016solitons}) is studied. Finally, the Kuramoto-Shivashinsky equation with a spatiotemporal chaotic pattern formation with applications in trapped ion modes in plasma \cite{laquey1975nonlinear} and phase dynamics in reaction-diffusion systems \cite{10.1143/PTP.55.356}. We shall see that the closure models in these three examples progressively involve approximations of functions of dimensions 40 to 480.

\subsection{Topographic mean flow interaction}

We consider the topographic mean flow interaction that solves a barotropic quasi-geostrophic equation with a large-scale zonal mean flow $u(t)$ on a
two-dimensional $2\pi \times 2\pi $ periodic domain, formulated as in \cite{qi2017low}:
\begin{align}
\begin{split}
\label{Eqn:topo57_1}
\frac{du}{dt} &+\fint\frac{\partial h}{\partial x}\psi  =-%
\bar{d}u+\sigma\mu^{-1/2}\dot{W}_{0},\\
\frac{\partial \omega }{\partial t} &+\nabla ^{\perp }\psi \cdot \nabla
q+u \frac{\partial q}{\partial x}+\beta \frac{\partial \psi }{
\partial x}  =-\mathcal{D} \psi +\Sigma
\dot{W}.
\end{split}
\end{align}
Here, $q=\omega +h$ is the small-scale potential vorticity which is advected
by the velocity field $\mathbf{v}=\nabla ^{\perp }\psi \equiv \left(
-\partial _{y}\psi ,\partial _{x}\psi \right) $; $\omega =\Delta \psi $ and $%
\psi $ are the relative potential vorticity and the stream function,
respectively; $h\left( \mathbf{x}\right) =h\left( x,y\right) $ is the
topography. The parameter $\beta $ is associated with the $\beta $-plane
approximation to the Coriolis force. The integral in \eqref{Eqn:topo57_1}
 is a two-dimensional integral over a periodic box of $[-\pi ,\pi ]\times
\lbrack -\pi ,\pi ]$. On the right hand side of \eqref{Eqn:topo57_1}, the dissipation and forcing operators are applied
on both the small and the large scales. On the small scale, the dissipation operator
is in the form of $\mathcal{D}=-\bar{d}\Delta $ with $\bar{d}\geq
0 $ and $\Delta$ the Laplace operator  corresponding to the Ekman drag dissipation. On the large scale, operator $-\bar{d}u$ represents the momentum
damping. The forcing terms are represented by random Gaussian white noises
(e.g. unresolved baroclinic instability processes on small scales, random
wind stress, etc), where $W\left( t\right) $ and $W_{0}\left( t\right) $ are
standard Wiener processes; $\sigma\mu^{-1/2}>0$ is a constant amplitude and $\Sigma$ is spatially dependent.

Following \cite{mw:06,franzke2009systematic,qi2017low}, we construct a set of special solutions to \eqref{Eqn:topo57_1}, which inherit the nonlinear coupling of the
small-scale vortical modes with the large-scale mean flow via topographic
stress. Consider the truncated spectral expansion of the state variables for
${\psi }$ and $\omega $ with high wavenumber truncation $1\leq \left\vert
\mathbf{k}\right\vert \leq K$ using standard Fourier basis $\exp (i\mathbf{k}%
\cdot \mathbf{x})$ with $\mathbf{k}=\left( k_{x},k_{y}\right) $. We can
rewrite \eqref{Eqn:topo57_1} for the large-scale mean flow $u\left(
t\right) $\ in a truncated Fourier form, as:
\begin{equation}
\begin{split}
\frac{du}{dt}& =i\sum_{1\leq \left\vert \mathbf{k}\right\vert \leq K}\frac{%
k_{x}}{\left\vert \mathbf{k}\right\vert ^{2}}\hat{h}_{\mathbf{k}}^{\ast
}\omega _{\mathbf{k}}-d\left( u-u_{eq}\right) +\sigma \mu ^{-1/2}\dot{W}_{t}, \\
\frac{d\omega _{\mathbf{k}}}{dt}& =\mathcal{P}_{K,\mathbf{k}}\left( \nabla
^{\perp }\psi _{N}\cdot \nabla q_{N}\right) +ik_{x}\left( \frac{\beta }{%
\left\vert \mathbf{k}\right\vert ^{2}}-u\right) \omega _{\mathbf{k}}-ik_{x}%
\hat{h}_{\mathbf{k}}u -d\left( \omega _{\mathbf{k}}-\omega _{eq,\mathbf{k}}\right) +\sigma _{%
\mathbf{k}}\dot{W}_{\mathbf{k},t},\text{ \ \ }1\leq \left\vert \mathbf{k}%
\right\vert \leq K,
\end{split}
\label{Eqn:u_k}
\end{equation}
Here, $\hat{h}_{\mathbf{k}}$ and $\omega _{\mathbf{k}}\ $ are the Fourier
transform of the topography $h\left( \mathbf{x}\right) $ and the relative
potential vorticity $\omega $, respectively; $u_{eq}=-\beta /\mu $ is the
equilibrium mean of $u\left( t\right) $; $\omega _{eq,\mathbf{k}}=-\left\vert \mathbf{k}\right\vert ^{2}\hat{h}%
_{\mathbf{k}}/\left( \mu +\left\vert \mathbf{k}\right\vert ^{2}\right) $ is
the mean relative vorticity; $\sigma _{\mathbf{k}}=\sigma \left( 1+\mu
\left\vert \mathbf{k}\right\vert ^{-2}\right) ^{-1/2}$\ is the forcing
strength for each mode $\mathbf{k}$. The parameter $\sigma $ is chosen
such that $\sigma _{eq}^{2}=\frac{\sigma ^{2}}{2\bar{d}}=1$. The parameters $%
\beta =1$ and $\mu =2$ are fixed in our simulation.


In our implementation, we consider the ground truth as the solution corresponding to the truncation $1\leq \left\vert \mathbf{k%
}\right\vert \leq K$ with $K=17$ such that there are $57$ degrees of freedom for integers $\mathbf{k}=\left( k_{x},k_{y}\right) $. In this
topographic $57$-mode model, we use the standard 4th order Runge-Kutta
method for the time integration up to $5\times 10^7$ time iterations with a time step
$\delta t=2.5$E$-3$, which is small enough to capture the small-scale
dynamics. For the nonlinear advection term, $%
\mathcal{P}_{K,\mathbf{k}}\left( \nabla ^{\perp }\psi _{K}\cdot \nabla
q_{K}\right) $, the 2/3 rule is applied for de-aliasing \cite{qi2017low}.
The noise is added using the standard Euler-Maruyama scheme.
Here, the initial condition, $\psi (\bm{x},0)$, is a sample of Gaussian distribution with random phases and amplitudes consistent with the
ensemble mean and enstrophy as in  \cite{mtv:03}. The observed data are recorded at every $20$ time steps, that
is, we observe the data at every $\Delta=0.05$ time unit. Taking half of this data set for training, $N=1.25\times 10^6$ samples.
For the topography $h\left( \mathbf{x}\right) $, we use a simple layered
topography with variation only in the $x$-direction, $h\left( \mathbf{x}\right)
=H\left( \cos \left( x\right) +\sin \left( x\right) \right) ,$ where $H$
denotes the topography amplitude.


We now present the closure model for the large-scale mean flow $u\left(
t\right) $ in Eq. (\ref{Eqn:u_k}). The application of the
Euler-Maruyama scheme for the large-scale mean flow $%
\hat{u}\left( t\right) $ gives%
\begin{align}
\begin{split}
\hat{u}_{t+1} &=\hat{u}_{t}+\Delta \, \hat\theta_{t}-\Delta \bar{d}\left( \hat{u}_{t}-u_{eq}\right) +\sqrt{%
\Delta}\sigma \mu ^{-1/2} \eta_{t+1},  \label{Eqn:topo_closure} \\
\hat\theta_{t+1} &=\mathbb{E}^\epsilon\Big[ \Theta_{t+1}|\hat{u}_{t-m:t},\hat\theta_{t-m:t}\Big] + \xi_{t+1},
\end{split}
\end{align}%
where the time step $\Delta =0.05$ and $\hat\theta_t$ is an estimator of the identifiable unresolved variable. In this case, $\theta = \theta(\omega _{\mathbf{k}}) =i\sum_{1\leq \left\vert \mathbf{k}\right\vert \leq K}\frac{k_{x}}{\left\vert \mathbf{k}\right\vert ^{2}}\hat{h}_{\mathbf{k}}^{\ast }\omega _{\mathbf{k}}$, is a function of the unresolved variables alone. The noises $\eta_t$ are i.i.d. standard Gaussian while the noises $\xi_t$ are Gaussian with mean zero and variance $\Xi$ determined from the training residual using Eq.~\eqref{noisebalancedmain}. We will approximate the conditional expectation $\mathbb{E}^\epsilon$ in \eqref{Eqn:topo_closure} with the LSTM model with $m=19$, which involves an approximation of a forty dimensional function. 
We should point out that the results become inaccurate when $m$ is too small. Taking the efficiency of computation into consideration, we empirically found that $m=19$ is a convenient choice. In fact, in all LSTM examples in the remaining of this paper, we use this same number of $m$ on the purpose of showing that our proposed framework is not sensitive to $m$. We will also include an experiment mimicking the existing approach in \cite{ma2018model,vlachas2018data,maulik2019time}, where the conditional expectation in \eqref{Eqn:topo_closure} is replaced with $\mathbb{E}^\epsilon\Big[ \Theta_{t+1}|\hat{u}_{t-m_u:t}]$, a function that depends only on the memory of the resolved scale with memory length $m_u = 19$. In this case, the conditional expectation is a twenty dimensional function. In addition, we also report the RKHS approximation to \eqref{Eqn:topo_closure} with $m=2$, which involves only an approximation of a six-dimensional function $(\hat{u}_{t-2:t},\hat\theta_{t-2:t})\mapsto\mathbb{E}^\epsilon[\Theta_{t+1}|\hat{u}_{t-2:t},\hat\theta_{t-2:t}]$. Here, we use a tensor product of Hermite polynomials to represent this six dimensional function. The curse of dimensionality makes the RKHS model with orthogonal polynomials prohibitive in higher dimensions. 

To compare the pathwise trajectories for short-time forecasting in the verification phase, we need to drive the closure model in \eqref{Eqn:topo_closure} with the appropriate realization of noises $\xi_t$ and $\eta_t$, that respects the realization of the noises, $\dot{W}_{t+1}, \dot{W}_{\mathbf{k},t}$ of \eqref{Eqn:u_k}, which drive the verification trajectory. Since the closure model depends on the identifiable unresolved variable, $\theta$, one can, in principle, account for the realization of the noise $\sigma_k\dot{W}_{\mathbf{k},t}$ corresponding to the verification trajectory by first mapping it to $\mathcal{W}$-space via the linear mapping $\theta(\cdot)$ (defined short after \eqref{Eqn:topo_closure}) and use it as a realization of $\xi_t$. Based on our numerical inspection, we found that the variance of this noise (in $\mathcal{W}$-space) is on the order of $10^{-8}$, whereas the variance of $\theta$ is on the order of $10^{-2}$, so there is a large signal-to-noise ratio in the true $\theta$-dynamics. 
With such a distinct signal-to-noise in the true $\theta$-dynamics, one can expect that the trajectory of $u$ can be recovered up to some accuracy even if this noise realization is neglected in the closure model so long as the conditional expectation model in \eqref{Eqn:topo_closure} can accurately approximate the deterministic part of the unresolved dynamics. Furthermore, we found that the residuals in the training phase are small (the variances, $\Xi$, are on the order of $10^{-5}$) in most of the dynamical regimes that we tested. Based on these observations, we will show numerical results obtained without additional noise $\xi_t$ in \eqref{Eqn:topo_closure}, to verify the prediction skill of the conditional expectation model alone. We should note that in separate experiments (not reported here), we found that the differences between the results without and with noise, $\xi_t$, where the latter uses a completely random realization, are negligible.

As for the stochastic noise, $\dot{W}_t$, corresponds to the resolved dynamics in \eqref{Eqn:u_k}, we notice that the full model and the closure
models are integrated with different time steps. The full model is
integrated with a relatively small time step $\delta t=2.5$E$-3$ in order to
resolve all the small-scale vorticity modes. Nevertheless, the closure
models are integrated with a relatively large time step $\Delta =0.05$ for
resolving only the large-scale mean flow $u\left( t\right) $. To compare the
pathwise trajectories, we first generate a realization of the noise $\dot{W}%
_{t+1}$ from identifiable variables using finite difference method
\begin{equation*}
\eta_{t+1}=\frac{u_{t+1}-\left[ u_{t}+\Delta \theta _{t}-\Delta d\left(
u_{t}-u_{eq}\right) \right] }{\sqrt{\Delta }\sigma \mu ^{-1/2}},
\end{equation*}%
where $u_{t}$ and $\theta _{t}$\ are the identifiable variables from the
dataset for verification of the full model observed with time interval $%
\Delta =0.05$. Using such realization of $\dot{W}_{t+1}$ for the noise $\eta
_{t+1}$ in \eqref{Eqn:topo_closure}, we can now compare the path of the closure model
with $\Delta =0.05$\ to the path of the true trajectory, starting at the same initial condition.

In Fig.~\ref{Fig2_topoLongtime}, we present the short-time predictions of $u$ in three regimes: weak coupling ($H=3\sqrt{2}/4$, $\bar{d}=0.5$), intermediate coupling ($H=5\sqrt{2}/4, \bar{d}=0.1$), and strong coupling ($H=7\sqrt{2}/4$, $\bar{d}=0.1$). These three regimes were considered in \cite{qi2017low}. We would like to emphasize that we verify the closure model in \eqref{Eqn:topo_closure} using initial conditions not in the training data set. In each regime, we compare the trajectories of the full model and the three closure models discussed above. In the weak coupling regime, one can see that the short-time predictions are all excellent among three closure models. For the other two regimes, the LSTM without the unresolved variables ($m_u=19$) proposed in \cite{ma2018model,vlachas2018data,maulik2019time} produces the worst prediction, which justifies the importance of considering both the resolved and identifiable unresolved variables in the closure model proposed in this paper. Surprisingly, the RKHS prediction is quite accurate considering that it requires less computational effort relative to LSTM. As we discussed before, given the large signal-to-noise ratio in the $\theta$-dynamics, the reasonably accurate trajectory recovery of the closure model in \eqref{Eqn:topo_closure} without incorporating the noise realization $\sigma_k\dot{W}_{\mathbf{k},t}$ suggests that the conditional expectation model alone has captured the bulk of the deterministic part of the unresolved dynamics.

In Fig. \ref{Fig3_topoLongtime}, we show the comparison of the equilibrium density and Auto-Correlation Functions (ACFs) of the full model and three prediction methods. The auto-correlation function (ACF) for the large-scale mean flow $u$
is calculated as $\left\langle U_{t}U_{0}\right\rangle / \left\langle
U_{0}U_{0}\right\rangle $, where $U_{t}=u_{t}-\left\langle
u_{t}\right\rangle $ with $\left\langle \cdot \right\rangle $ being the
temporal average over $1.25\times 10^{6}$ data for verification, which is
different from the $N=1.25\times 10^{6}$ training dataset. The probability
density function (PDF) for $u$ is obtained from the same verification
dataset using the kernel density estimation (KDE) method.
For the long-time statistics, both the LSTM methods provide a better approximation than the RKHS model. This is
because not enough memory terms are used in the RKHS model. In the strong coupling regime, one can see that the LSTM method with $m=19$ is the best
approximation. This is because the variance of the training residual is
small about $10^{-5}$ for the LSTM with $m=19$ and the variance is
relatively large about $10^{-1}$ for the LSTM with $m_{u}=19$. This result confirms the robustness
of our framework with a closure model that depends on, both, the memories of the resolved and identifiable unresolved variables.


\begin{figure}
\minipage{1\textwidth}
\minipage{0.3333\textwidth}
\subcaption{Weak coupling}
\vspace{-5pt}
\includegraphics[width=\textwidth,height=5cm]{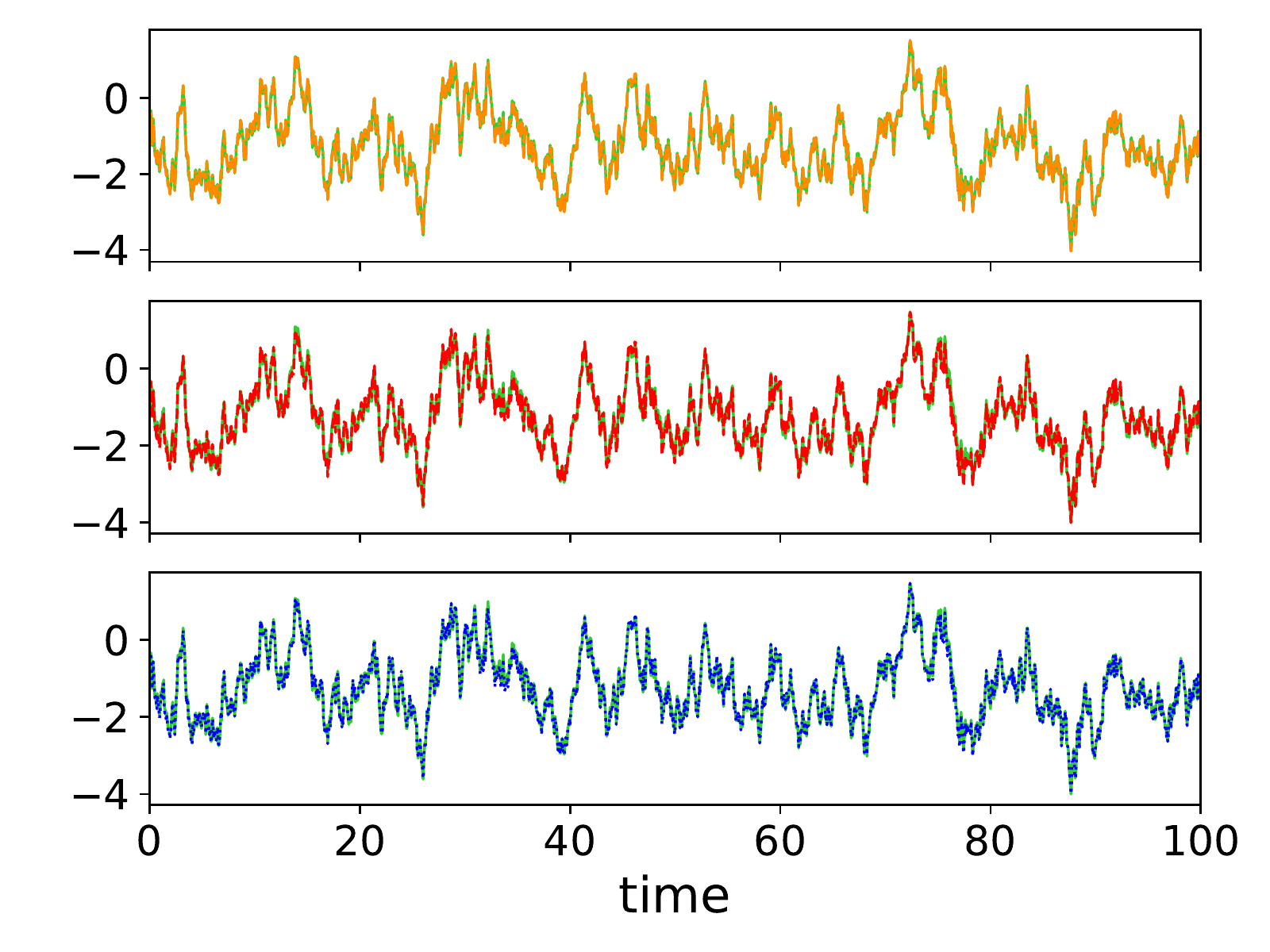}
\endminipage\hfill
\minipage{0.3333\textwidth}
\subcaption{Intermediate coupling}
\vspace{-5pt}
\includegraphics[width=\textwidth,height=5cm]{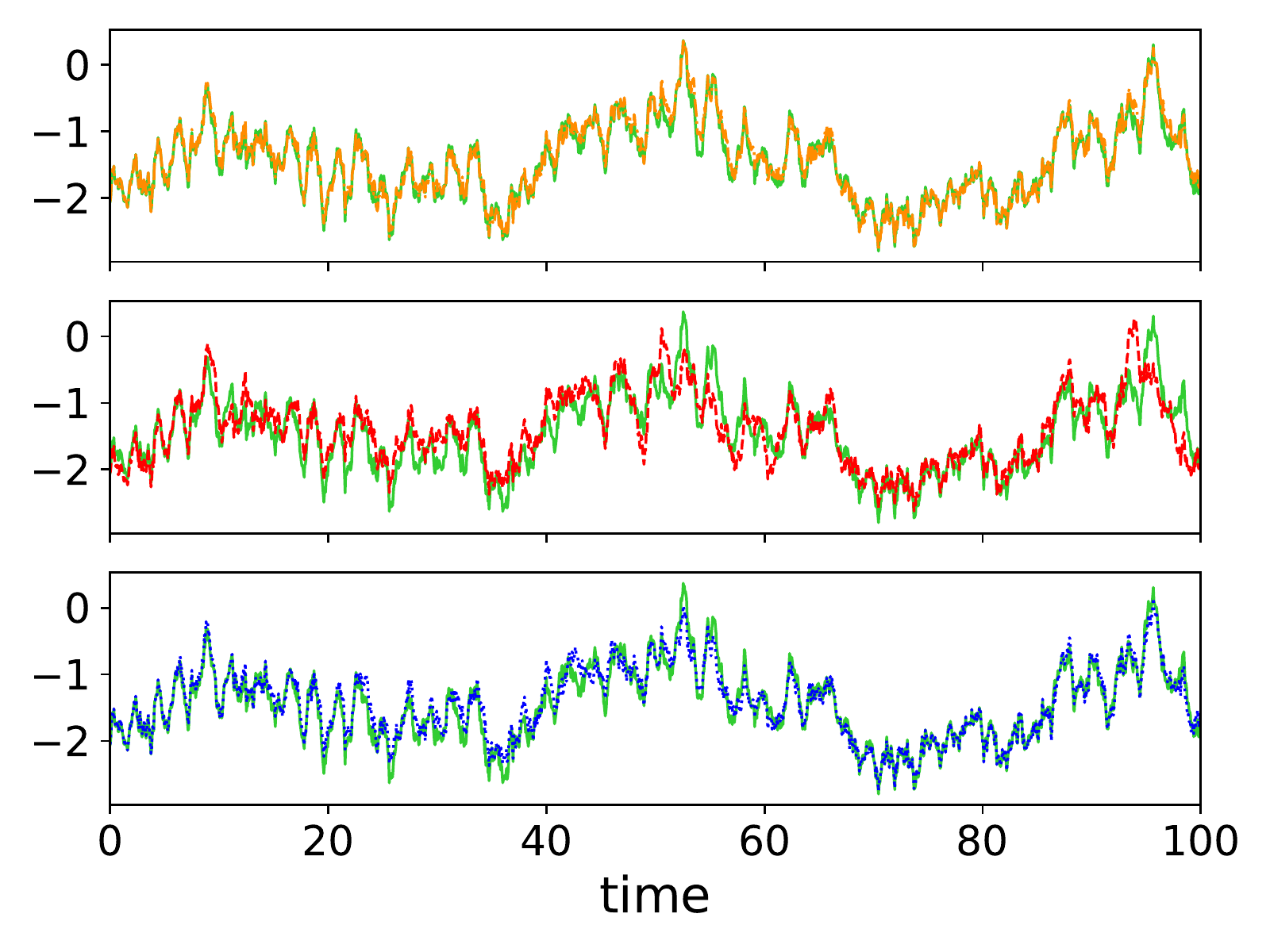}
\endminipage\hfill
\minipage{0.3333\textwidth}
\subcaption{Strong coupling}
\vspace{-5pt}
\includegraphics[width=\textwidth,height=5cm]{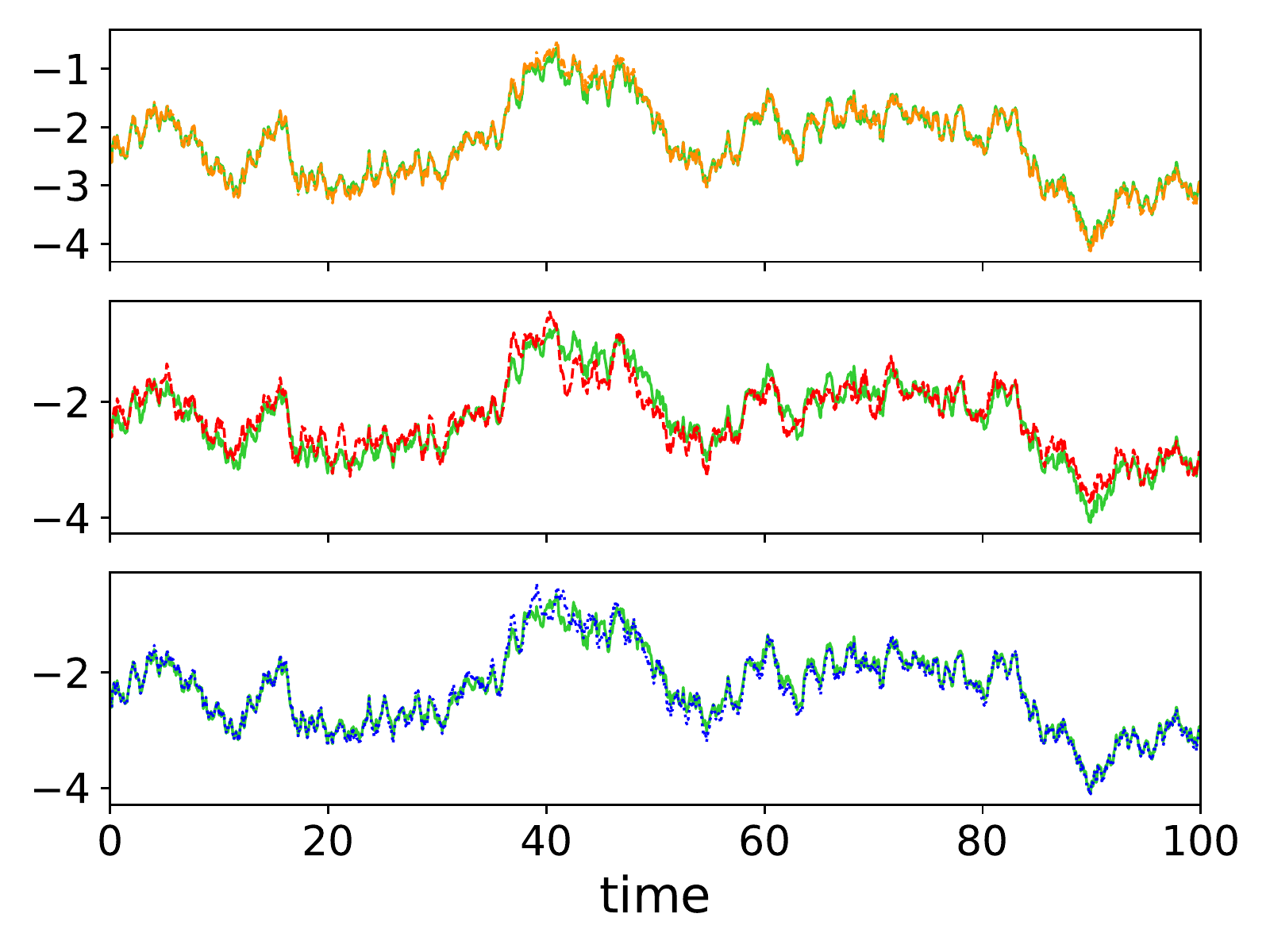}
\endminipage
\endminipage\hfill
\caption{Comparison of trajectories between the full and closure models for the
mean velocity $u$. The true trajectory (green solid) in all panels; RKHS $m=2$ (yellow dash-dotted line) in the first row; LSTM $m_u=19$ (red dashes) in the second row; LSTM $m=19$ (blue dotted line) in the third row.}
\label{Fig2_topoLongtime}
\end{figure}

\begin{figure}
\minipage{1\textwidth}
\minipage{0.3333\textwidth}
\subcaption{Weak coupling}
\vspace{-5pt}
\includegraphics[width=\textwidth,height=5cm]{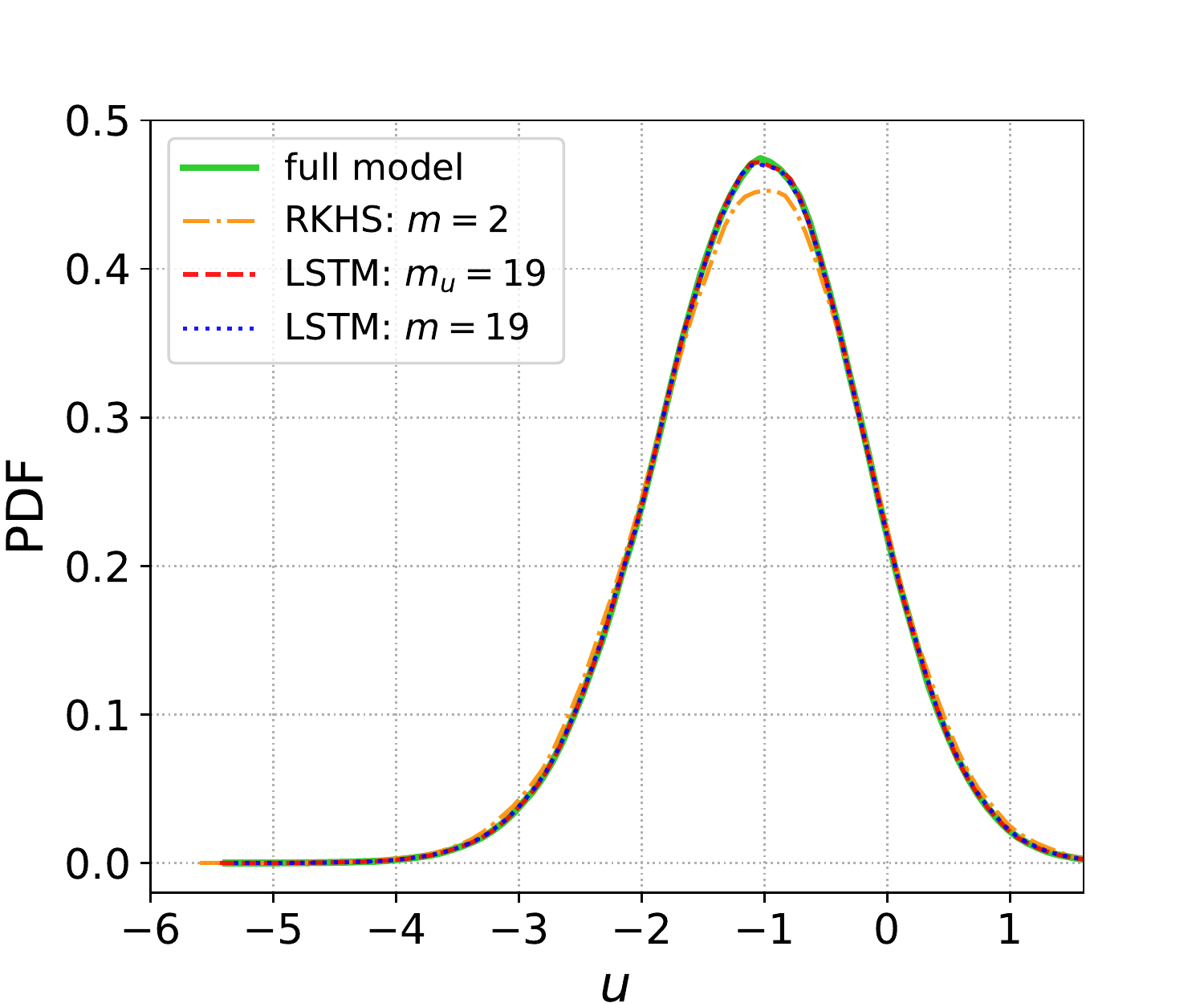}
\endminipage\hfill
\minipage{0.3333\textwidth}
\subcaption{Intermediate coupling}
\vspace{-5pt}
\includegraphics[width=\textwidth,height=5cm]{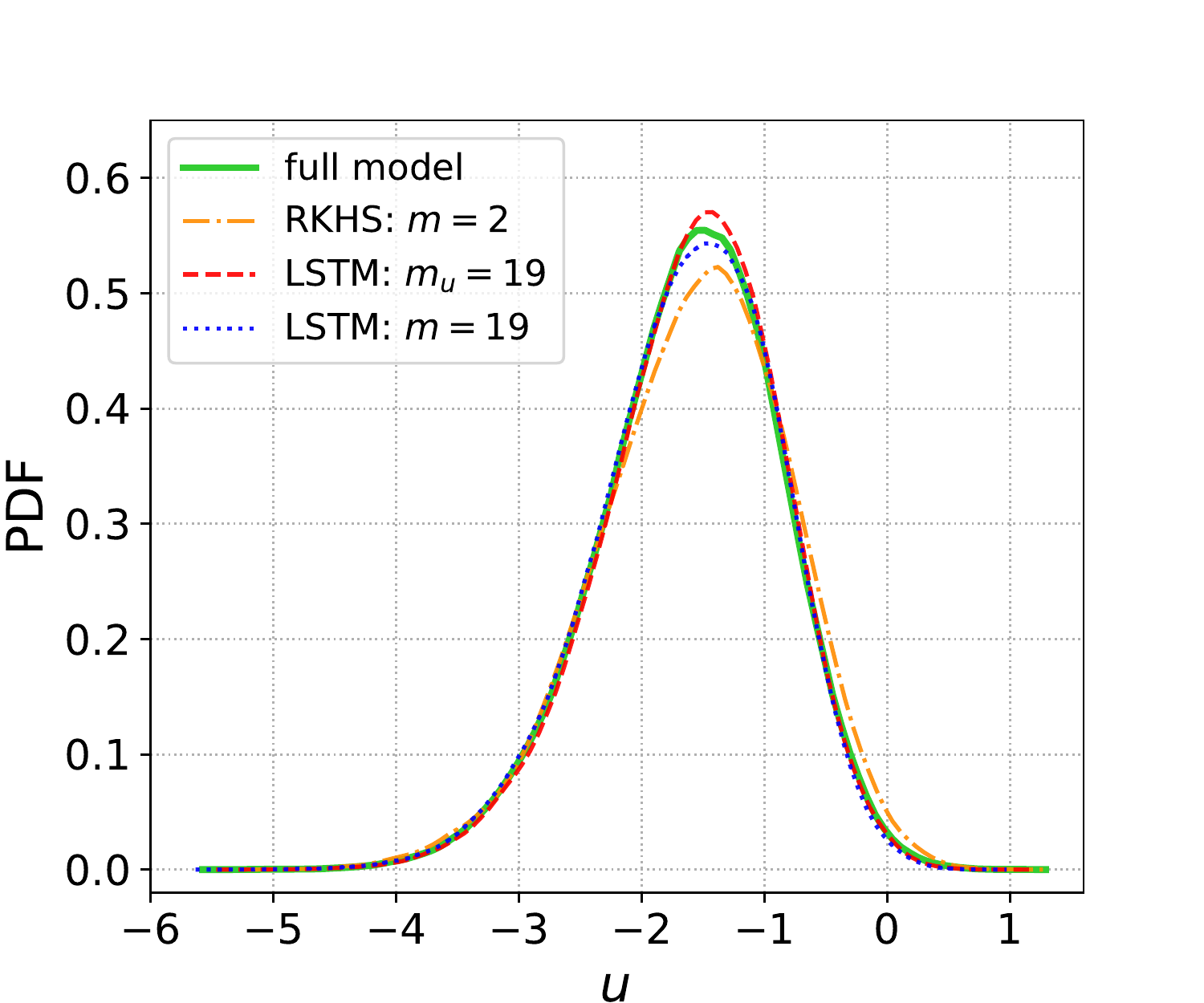}
\endminipage\hfill
\minipage{0.3333\textwidth}
\subcaption{Strong coupling}
\vspace{-5pt}
\includegraphics[width=\textwidth,height=5cm]{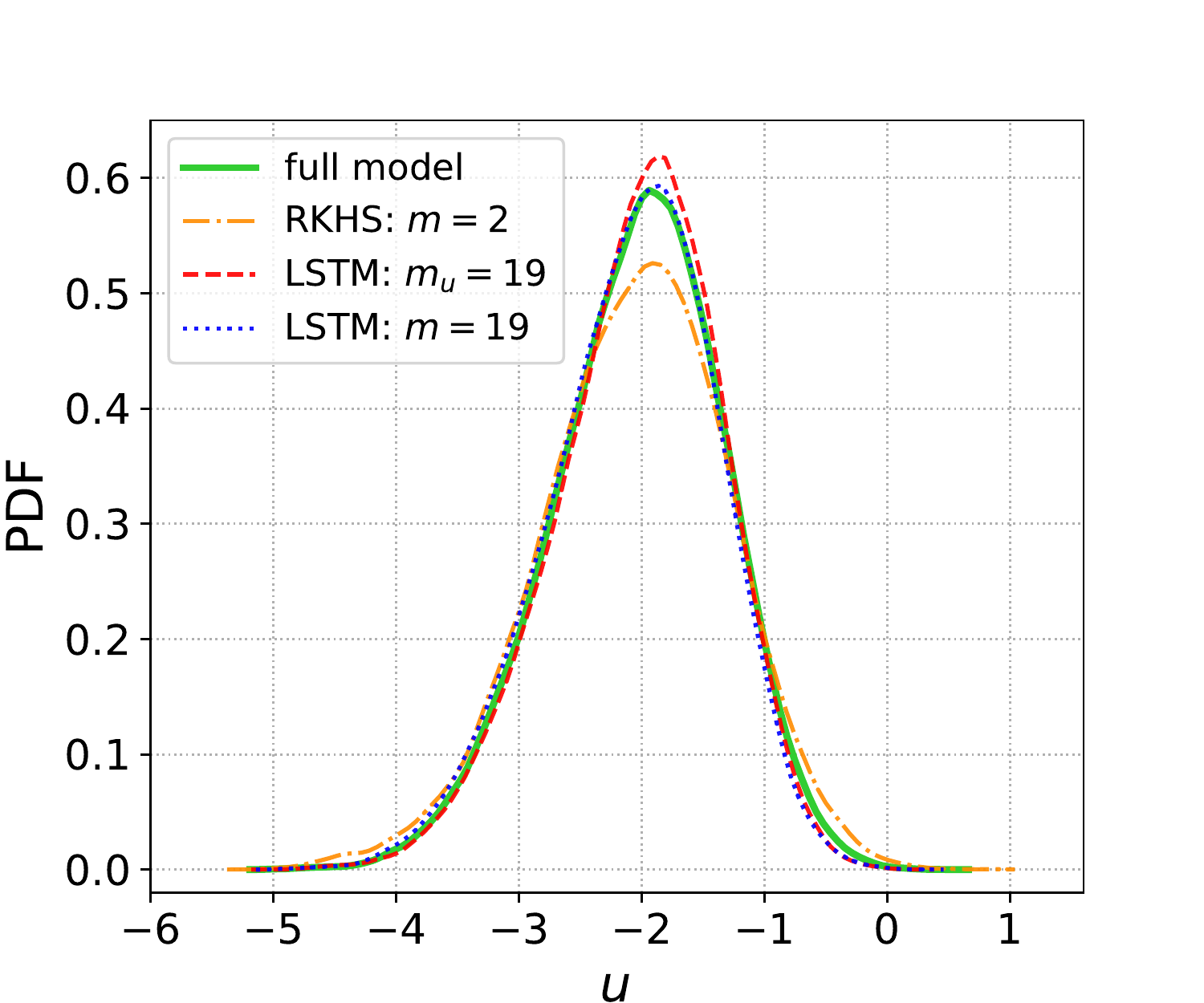}
\endminipage
\endminipage\hfill
\minipage{1\textwidth}
\minipage{0.3333\textwidth}
\subcaption{Weak coupling}
\vspace{-5pt}
\includegraphics[width=\textwidth,height=5cm]{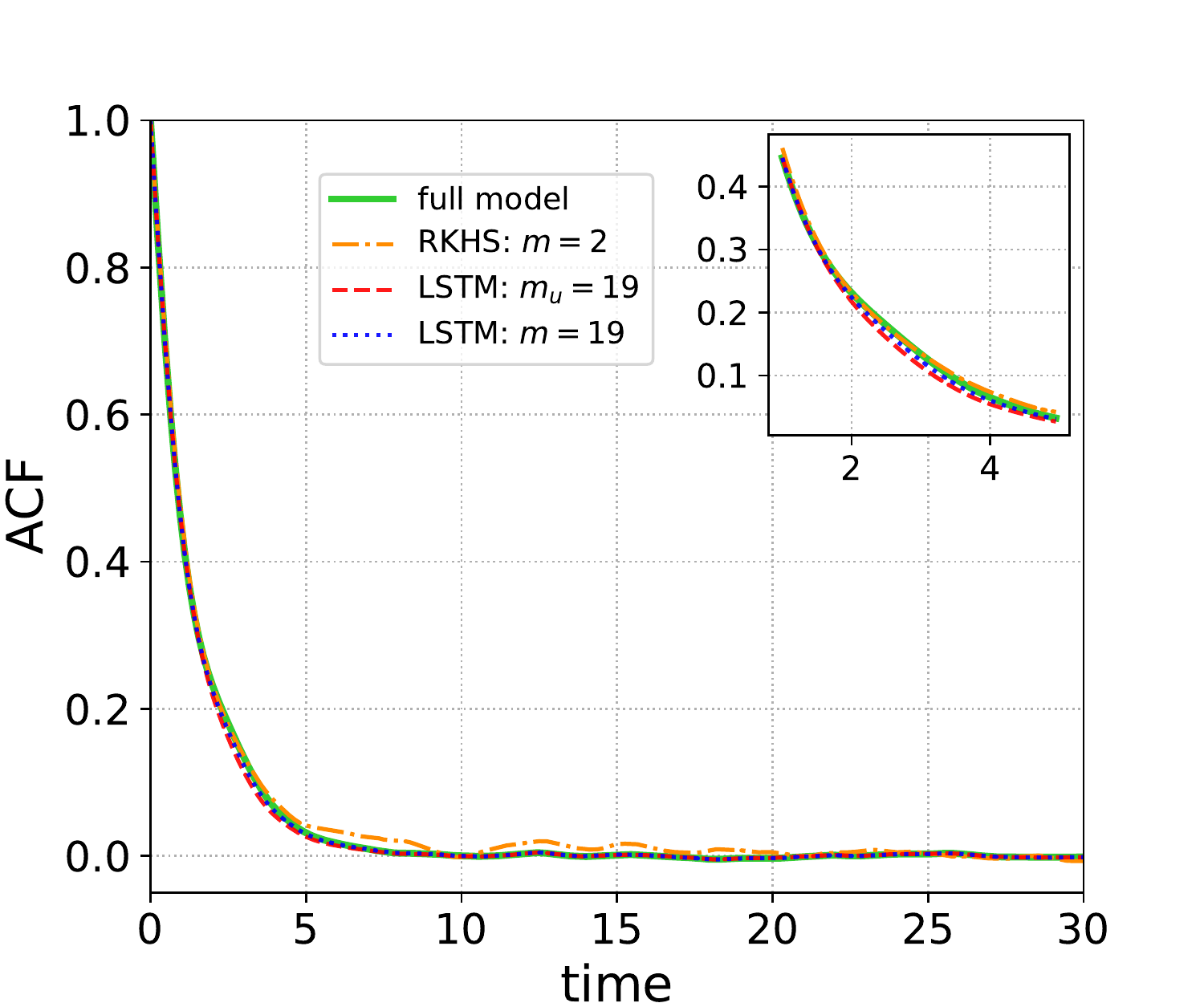}
\endminipage\hfill
\minipage{0.3333\textwidth}
\subcaption{Intermediate coupling}
\vspace{-5pt}
\includegraphics[width=\textwidth,height=5cm]{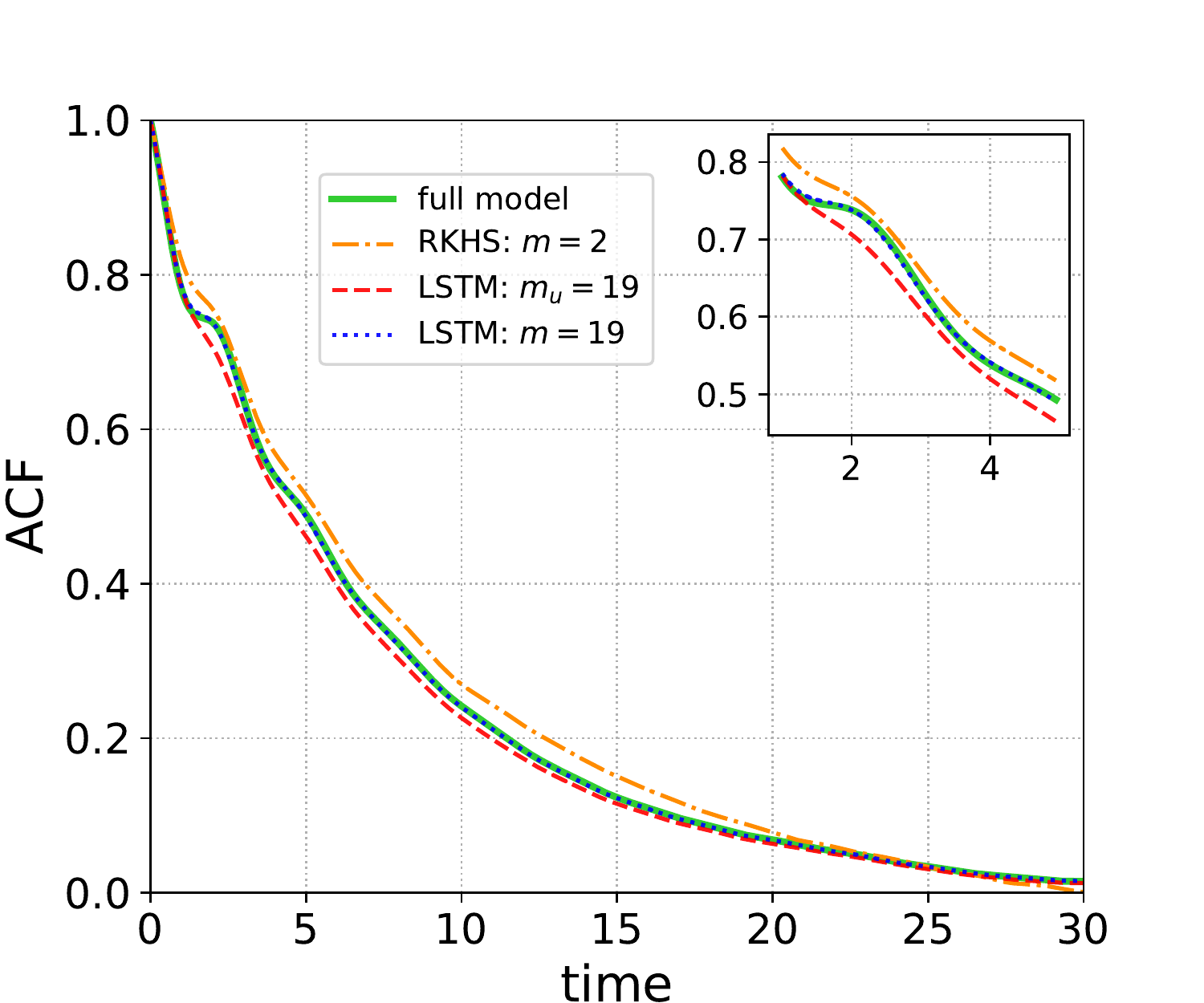}
\endminipage\hfill
\minipage{0.3333\textwidth}
\subcaption{Strong coupling}
\vspace{-5pt}
\includegraphics[width=\textwidth,height=5cm]{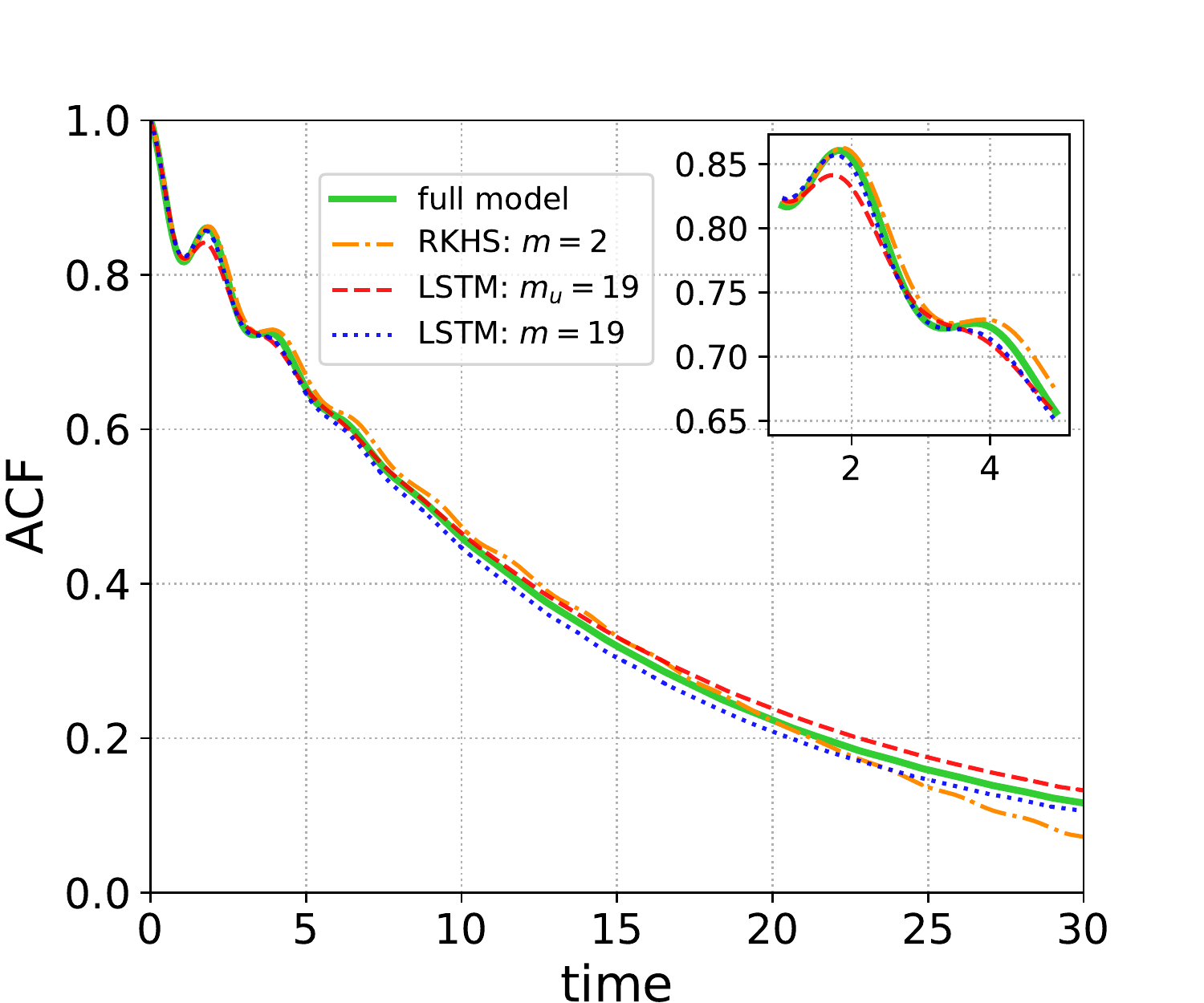}
\endminipage
\endminipage\hfill
\caption{Comparison of densities and Auto-Correlation functions (ACF) between the full and closure models for the mean velocity $u$.}
\label{Fig3_topoLongtime}
\end{figure}

\subsection{\label{sec:level1}Nonlinear Schr\"{o}dinger equation}

We consider the cubic nonlinear Schr\"{o}dinger (NLS) equation defined on a periodic boundary condition $[0,2\pi]$, which dynamical equation
can be written in terms of the Fourier modes as,
\begin{equation}
\frac{du_{k}}{dt}=-ik^2u_{k}-i\sum_{k_{1}\in \mathbb{Z}%
}\sum_{k_{2}\in \mathbb{Z}}u_{k_{1}}u_{k_{2}}u_{k_{1}+k_{2}-k}^{\ast }.
\label{Eqn:uk}
\end{equation}
Numerically, we generate the truth by integrating \eqref{Eqn:uk} on finite
wavenumbers $\left\vert k\right\vert \leq K$ and Strang's splitting method in time
\cite{bao2003numerical}. Here, the number of modes $K=32$ and
the observation time interval $\Delta=0.02$. The observation data length
is $10^6$, obtained from a single trajectory. Taking half of this data set for training, $N=5\times 10^5$ samples.

We simulate the initial conditions by sampling from the Gibbs distribution $\pi = \exp(-\frac{E}{k_BT})$, where $E$ denotes the Hamiltonian of the ODE system resulting from the Fourier representation \eqref{Eqn:uk}; $k_B$ and $T$ denote the Boltzmann constant and temperature, respectively. In this case, the Hamiltonian is given by $E=E_{2}+E_{4}$,%
\begin{eqnarray*}
E_{2} =\sum_{k\in \mathbb{Z}}k^2\left\vert
u_{k}\right\vert ^{2}, \quad\quad
E_{4} =\frac{1}{2}\sum_{k_{1}\in \mathbb{Z}}\sum_{k_{2}\in
\mathbb{Z}}\sum_{k_{3}\in \mathbb{Z}}u_{k_{1}}u_{k_{2}}u_{k_{3}}^{\ast
}u_{k_{1}+k_{2}-k_{3}}^{\ast }.
\end{eqnarray*}
We should point out that the qualitative solutions for higher temperature have larger amplitudes and frequencies. Since smaller time steps are required for accurate solutions with higher amplitude as well as the faster frequency, the problem is numerically stiff as the temperature increases. To keep the presentation short, we only show the numerical results for the zeroth mode $u_{0}$ in a high-temperature regime with $k_{B}T=10$. Our numerical test on lower temperature regime (not shown) do not change the conclusion in this section. In fact, a parametric closure proposed in \cite{hl:15} has shown accurate recovery for extremely low temperature regime, $k_BT=10^{-4}$, and less accurate as the temperature increases. The stiffness of high-temperature regime will also be manifested in the numerical scheme that is used in integrating the closure model as we will explained below.

In this example, we are interested in constructing a closure model for the dynamics of the
zeroth mode $u_{0}$ of the NLS equation. Given the dynamical equation of the resolved variable, $u_0$, we can rewrite it in the form of \eqref{Fadditive} as,
\BEA
\frac{du_{0}}{dt}=-i\sum_{k_{1}\in \mathbb{Z}%
}\sum_{k_{2}\in \mathbb{Z}}u_{k_{1}}u_{k_{2}}u_{k_{1}+k_{2}}^{\ast } := -i\Big(|u_0|^2u_0 + \theta \Big),\label{NLSreduced}
\EEA
where $\theta$ is basically the full vector field without the cubic term that involves only $u_0$ in the right-hand-side of \eqref{NLSreduced}.
The closure model is obtained by concatenating a discretization of \eqref{NLSreduced} with time step $\Delta$ with a map,
$\mathbb{E}^\epsilon \Big[ \Theta_{t+1}|\cdot\Big] : \mathbb{R}^{(m+1)\times 4} \to \mathbb{R}^2$, defined as,
\BEA
{\theta}_{t+1} &= \mathbb{E}^\epsilon \Big[ \Theta_{t+1}|u_{0,t-m:t},{\bm\theta}_{t-m:t}\Big].\label{closureNLS}
\EEA
Here ${\theta}_{t+1}$ denotes the unresolved identifiable component at discrete time $t+1$. To train this model, we need a time series of $\{\theta_t\}$ in addition to $\{u_{0,t}\}$. Based on the form of the resolved dynamics in  \eqref{NLSreduced}, given a training time series of $\{u_{0,t}\}$, we extract $\{\theta_t\}$ by a direct subtraction and a finite difference approximation to the derivative. However, we should point out that if we reverse-engineer this step, that is, solve \eqref{NLSreduced} with the true initial condition of $u_0(0)$ using a lower-order scheme (such as Euler method) and directly use the data $\{\theta_t\}$ that we just obtained from direct subtraction, the solution for $u_0(t)$ will blow up in finite-time. This is a manifestation of the stiffness of this problem. To avoid this issue in the closure model, we apply the following time-splitting method in our numerical discretization of \eqref{NLSreduced}. That is, we use the Euler scheme to solve the linear ODE, $du_{0}/dt= -i \theta$, since we only have discrete estimates of $\theta$, and we use the explicit solution $u_0(t) = u_0(t_0)\exp(-i |u_0(t_0)|^2t)$ for the nonlinear ODE, $du_0/dt =  -i|u_0|^2u_0$.

In this numerical experiment, we fix the memory length to be $m=19$ in the LSTM method, resulting in an approximation of 80-dimensional function $\mathbb{E}^\epsilon \Big[ \Theta_{t+1}|\cdot\Big]$. No residual term is added in \eqref{closureNLS}. For the short-time forecasting, we observe from Fig.~\ref{Fig_NLShighT}(a) that the path-wise solution of ${\rm Re}(u_{0})$ is well captured for a sufficiently long time; the discrepancies in the frequencies are noticeable as time increases. In Fig.~\ref{Fig_NLShighT}(b) and (c), we also reported the ACF for the ${\rm Re}(u_{0})$, calculated by a temporal average over $5\times 10^{5}$ verification data, which is different from the $N=5\times 10^{5}$ training dataset. The PDF for ${\rm Re}(u_{0})$ is obtained from the same verification dataset using the kernel density estimation (KDE) method. Notice that both the ACF (below time 40 unit) and the density of the true $u_{0}$ are also well reproduced. Therefore, for the first mode $u_{0}$\ of the NLS equation, the proposed closure model using the LSTM method can reasonably replicate the short-time forecasting skill and long-time statistics in the high-temperature regime.

We should point out that the resulting model is only valid in predicting the evolution of the system on the same energy level since the underlying Hamiltonian system is not ergodic. This implies that the verification will only be valid to predict the evolution of the system with initial conditions sampled from the same Gibbs distribution where the training data is generated from.

\begin{figure}
\centering
\minipage{1\textwidth}
\minipage{0.5\textwidth}
\subcaption{Trajectory}
\vspace{-5pt}
\includegraphics[width=\textwidth]{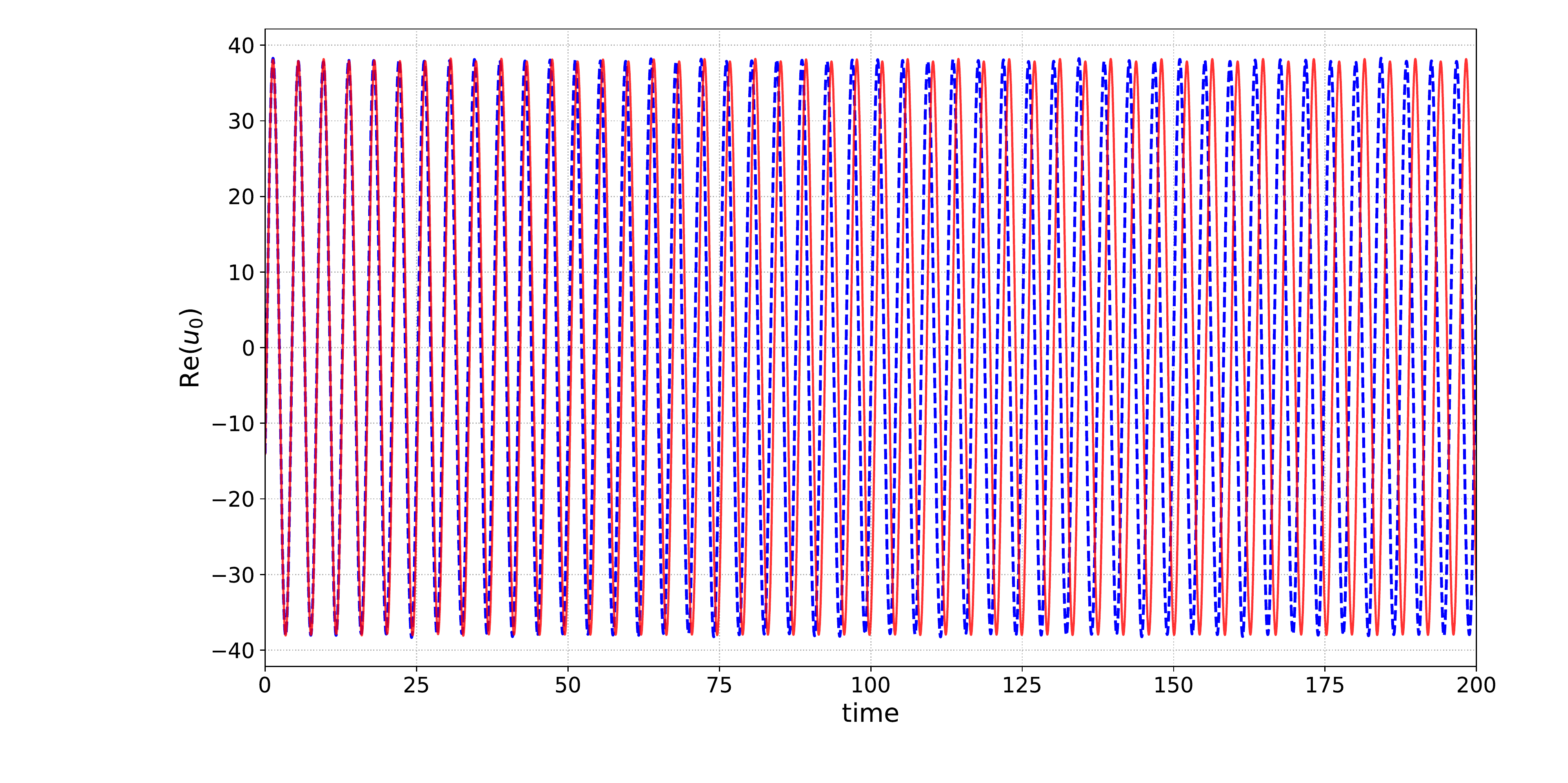}
\endminipage\hfill
\minipage{0.5\textwidth}
\subcaption{ACF}
\vspace{-5pt}
\includegraphics[width=\textwidth]{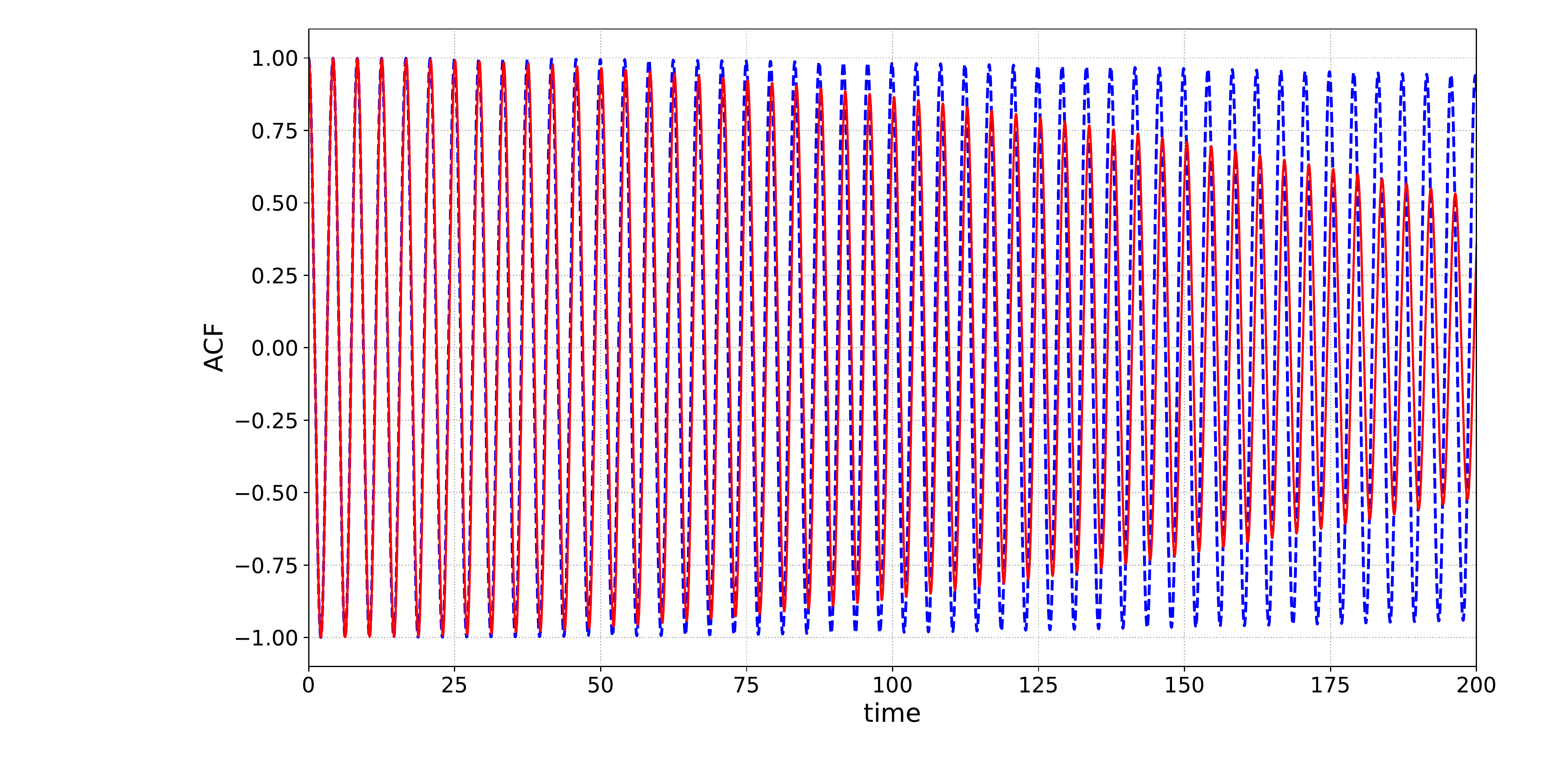}
\endminipage\hfill
\endminipage\hfill
\minipage{.5\textwidth}
\subcaption{Density}
\vspace{-5pt}
\includegraphics[width=\textwidth,height=5cm]{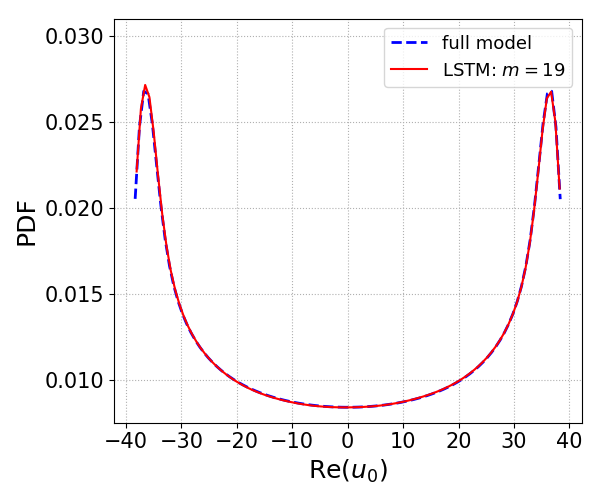}
\endminipage
\caption{Comparison of (a)
trajectories, (b) ACFs, and (c) densities  of the full and closure LSTM models in the high temperature
regime with $k_{B}T=10$. The time step for the closure model is $\Delta
=0.02$ and the LSTM method uses $m=19$ memory terms for $(\bm{\theta},u_{0})$ in Eq.~\eqref{closureNLS}. Full model (blue dashes); LSTM $m=19$ (red solid).}
\label{Fig_NLShighT}
\end{figure}

\subsection{The Kuramoto-Sivashinsky equation} We consider the Kuramoto-Sivashinsky equation (KSE) on an $L-$periodic domain, the Fourier representation of which can be written as

\comment{
which describes the evolution of a
chemical system \cite{10.1143/PTP.55.356}, is given
in the form of%
\begin{eqnarray}
&&\frac{\partial v}{\partial t}+v\frac{\partial v}{\partial x}+\frac{%
\partial ^{2}v}{\partial x^{2}}+\frac{\partial ^{4}v}{\partial x^{4}}=0,
\label{Eqn:ks_real} \\
&&v\left( x,t\right) =v\left( x+L,t\right) ;\text{ \ }v\left( x,0\right)
=g\left( x\right) ,  \notag
\end{eqnarray}%
where $t$ is time, $x$ is space, $v$ is the solution, $L$ is the spatial
length, and $g\left( x\right) $ is the initial condition. We consider a $N$

modes Galerkin--Fourier representation of the solution and then we can write
the KSE in terms of Fourier modes:}
\begin{equation}
\frac{d}{dt}v_{k}=\left( q_{k}^{2}-q_{k}^{4}\right) v_{k}-\frac{iq_{k}}{2}%
\sum_{l=-\infty }^{\infty }v_{l}v_{k-l},  \label{Eqn:ks_fourier}
\end{equation}%
where $q_{k}=2\pi k/L$ with $k\in \mathbb{Z}$, and $v_{k}$ denotes the $k$th Fourier
mode.

In our numerical implementation, we let the full dynamics to be the Galerkin truncation of \eqref{Eqn:ks_fourier} for $|k| \leq K/2$, where $K=96$.
Notice that in the linearized equations of \eqref{Eqn:ks_fourier}, each Fourier mode has an eigenvalue $%
q_{k}^{2}-q_{k}^{4}$ so that high $k$ modes with $\left\vert q_{k}\right\vert >1$ are linearly stable whereas low $k$\ modes with $%
\left\vert q_{k}\right\vert \leq 1$ are not. We set the spatial length $L=2\pi /\sqrt{0.085}$ so that the number of linearly unstable modes is $\left\lfloor 1/\sqrt{0.085}\right\rfloor =3$. In this case, the energy is transferred from the linearly unstable low $3$ modes to the damped high $K/2-3 = 45$ modes through the nonlinear terms so that the KSE is well-posed and the solutions remain globally bounded in time \cite{goodman1994stability}. This regime is exactly the same as the one considered in \cite{lu2017data,lin2019data}.

We predict the six leading modes of the KSE with the following partial dynamics,
\begin{equation}
\frac{d}{dt}\hat{v}_{k}=\left( q_{k}^{2}-q_{k}^{4}\right) \hat{v}_{k}-\frac{iq_{k}}{2}%
\sum_{1\leq |l|,|k-l|\leq 6}\hat{v}_{l}\hat{v}_{k-l} + \hat{\theta}_k, \quad k=1,\ldots,6.  \label{kstruncated}
\end{equation}%
In this case, since the nonlinear terms in \eqref{kstruncated} only involve summation of terms that are restricted to $1\leq |l|,|k-l|\leq 6$, the identifiable unresolved variables, $\theta_k$, depends also on the resolved modes, in addition to the unresolved modes. The proposed closure model is to concatenate the numerical discretization of \eqref{kstruncated} with the discrete nonparametric closure model,
\BEA
\hat{\bm{\theta}}_{t+1} =\mathbb{E}^\epsilon \Big[ \Theta_{t+1}|\hat{\bm{v}}_{t-m:t},\hat{\bm\theta}_{t-m:t}\Big],\label{closure2}
\EEA
where $\bm{\hat{\theta}}_t=(\hat{\theta}_{1,t},\ldots,\hat{\theta}_{6,t})\in\mathbb{C}^6$ and $\bm{\hat{v}}_t=(\hat{v}_{1,t},\ldots,\hat{v}_{6,t})\in\mathbb{C}^6$. In our numerical experiment, we set $m=19$ such that $\mathbb{E}^\epsilon$ in \eqref{closure2} is a function that maps a real-valued vector of size $(19+1)\times 12\times 2 = 480$ to a 12-dimensional vector consisting of the real and imaginary values of $\bm{\hat{\theta}}_{t+1}$. To evolve the dynamics in \eqref{kstruncated}-\eqref{closure2}, we discretize \eqref{kstruncated} with the midpoint rule and a time step $\Delta$.

In our numerical experiment, the true time series for training are obtained by integrating the full dynamics, that is, \eqref{Eqn:ks_fourier} truncated on $1\leq |k|\leq 48$ with a time step $\delta t=0.005$. We observe only the first 6 modes at a time step $\Delta =0.05$. The size of the training data set is $N=2.5\times 10^{5}$. The identifiable unresolved variable, $\bm{\theta}_t$, is estimated by fitting the time series $\bm{v}_t$ to the dynamics in \eqref{kstruncated}. Subsequently, we use the pair $\{\bm{\theta}_t,\bm{v}_t\}$ to train the LSTM model for \eqref{closure2}; for training, we add Gaussian noises of variance 1\% relative to that of the original data to avoid overfitting that tends to occur when the hypothesis space is rather complex and the amount of data is finite.

Fig.~\ref{Fig_ks}(a) displays the difference of the short-time spatiotemporal manifestation between the full and the closure models. One can see that the spatio-temporal pattern of the proposed closure model is consistent with that of the full KS model up to roughly time $t=54$. A close inspection shows an accurate path-wise prediction of the real component of the leading six Fourier modes up to time $54$ (see Fig.~\ref{Fig_KSsup}). In Fig.~\ref{KSrmse}, we report the root-mean-square-error (RMSE) and (anomaly correlation) ANCR as defined in \cite{crommelin2008subgrid} that characterize the lead-time prediction skill, averaged over 1000 initial conditions out-of-samples and the spatial domain. Notice that both metrics show a substantial improvement in the prediction skill relative to that of the bare truncated model which is a result of using only \eqref{kstruncated} with $\hat{\theta}_k=0$.

For this regime $L\approx 21.55$, the leading Lyapunov exponent is roughly $\lambda_1\approx 0.04$ \cite{edson_bunder_mattner_roberts_2019}, which suggests that the accurate prediction length is roughly $54\times \lambda_1=2.16$ Lyapunov time units. In other words, the length of the prediction is on the same order as the Lyapunov time. While this empirical result suggests that the constant $a$ in Theorem~\ref{thm3} is roughly $e^{\lambda_1}$, a theoretical justification for such a tighter bound will require more thorough investigation with possibly additional assumptions on the dynamics.

In Fig.~\ref{Fig_ks}(b), we show the accurate recovery of the energy spectra. Fig.~\ref{Fig_KSsup} also displays the results for the comparison of ACFs and PDFs for all the Fourier modes $v_{1},\ldots
,v_{6}$ and CCF's defined as the cross-correlation functions between $|v_k|^2$ and $|v_4|^2$. All of these long-time statistics are computed using the Monte-Carlo estimation over $2.5\times 10^5$ data samples. We can see that ACFs, CCFs, and PDFs can be well reproduced by the LSTM for all modes. Therefore, the closure model using the LSTM can provide an accurate recovery for both the short-time forecasting and the long-time statistics of the KSE.

To summarize, we should also mention that while such an accurate recovery in path-wise and statistical prediction has also been achieved with the NARMAX parametric closure in \cite{lu2017data,lin2019data}, careful choice of parametric ansatz is necessary with the NARMAX model. Here, an accurate recovery is obtained with a much simpler nonparametric model in \eqref{closure2}.

\begin{figure}[tbp]
\centering
\minipage{1\textwidth}
\minipage{0.55\textwidth}
\subcaption{Difference of the full and closure models solutions}
\vspace{-5pt}
\includegraphics[width=\textwidth,height=5cm]{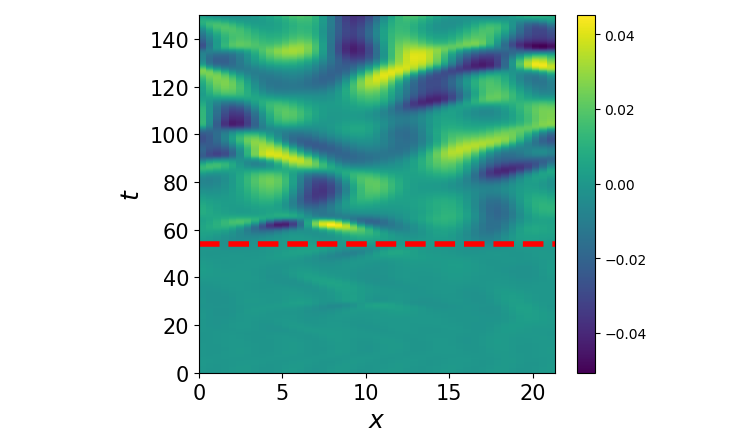}
\endminipage
\minipage{0.35\textwidth}
\subcaption{Energy spectrum}
\vspace{-5pt}
\includegraphics[width=\textwidth,height=5cm]{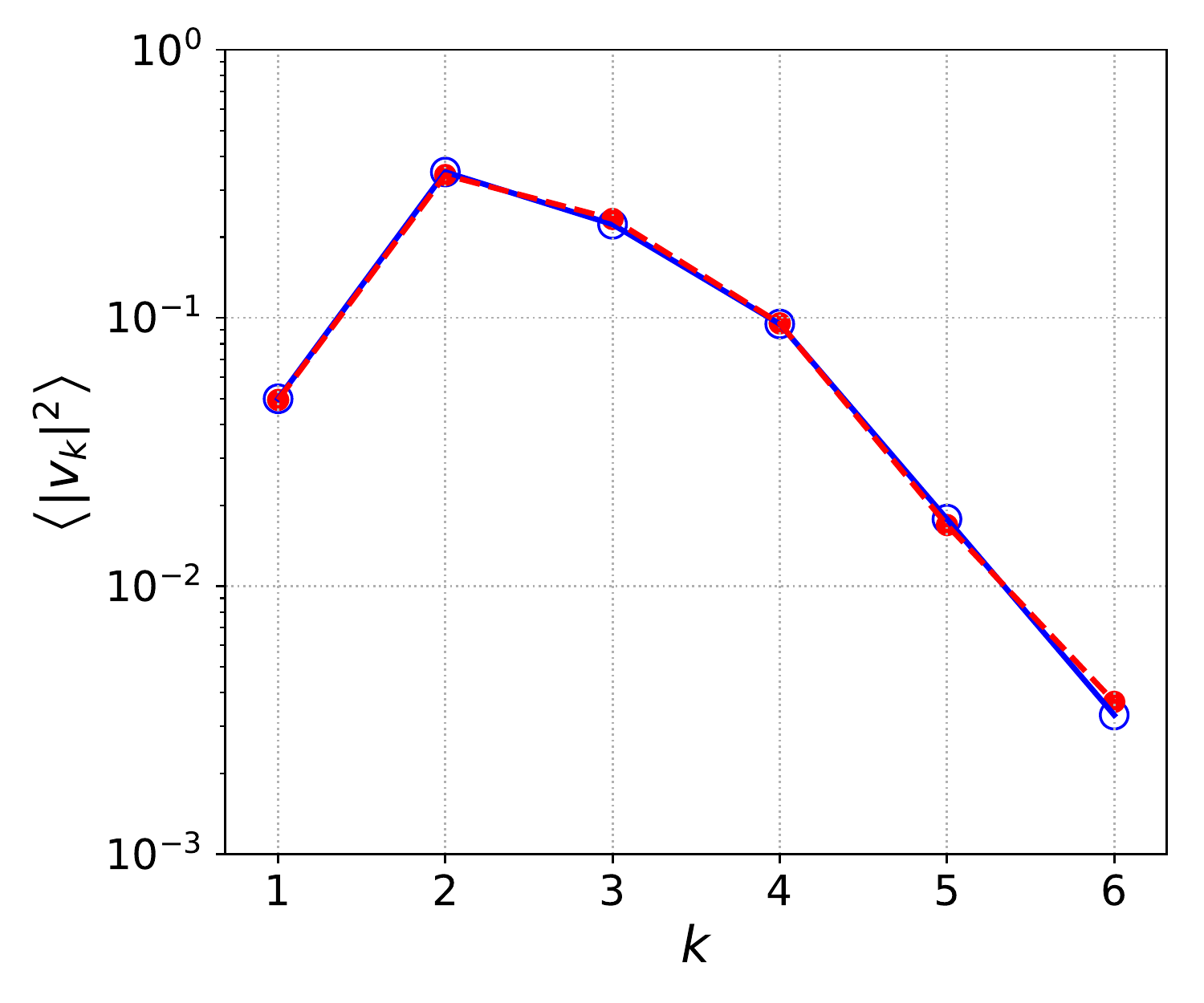}
\endminipage
\endminipage\hfill
\caption{(a) Difference of spatiotemporal manifestation of KS solutions starting from the same initial conditions between the full model and the closure model using the LSTM method; (b) The energy spectra $\left\langle\left\vert v_{k}\right\vert\right\rangle ^2$ for the KS solutions between the full (blue solid) and the closure (red dashes) LSTM models.}
\label{Fig_ks}
\end{figure}

\begin{figure}[tbp]
\centering
\minipage{1\textwidth}
\minipage{0.5\textwidth}
\subcaption{RMSE}
\vspace{-5pt}
\includegraphics[width=\textwidth]{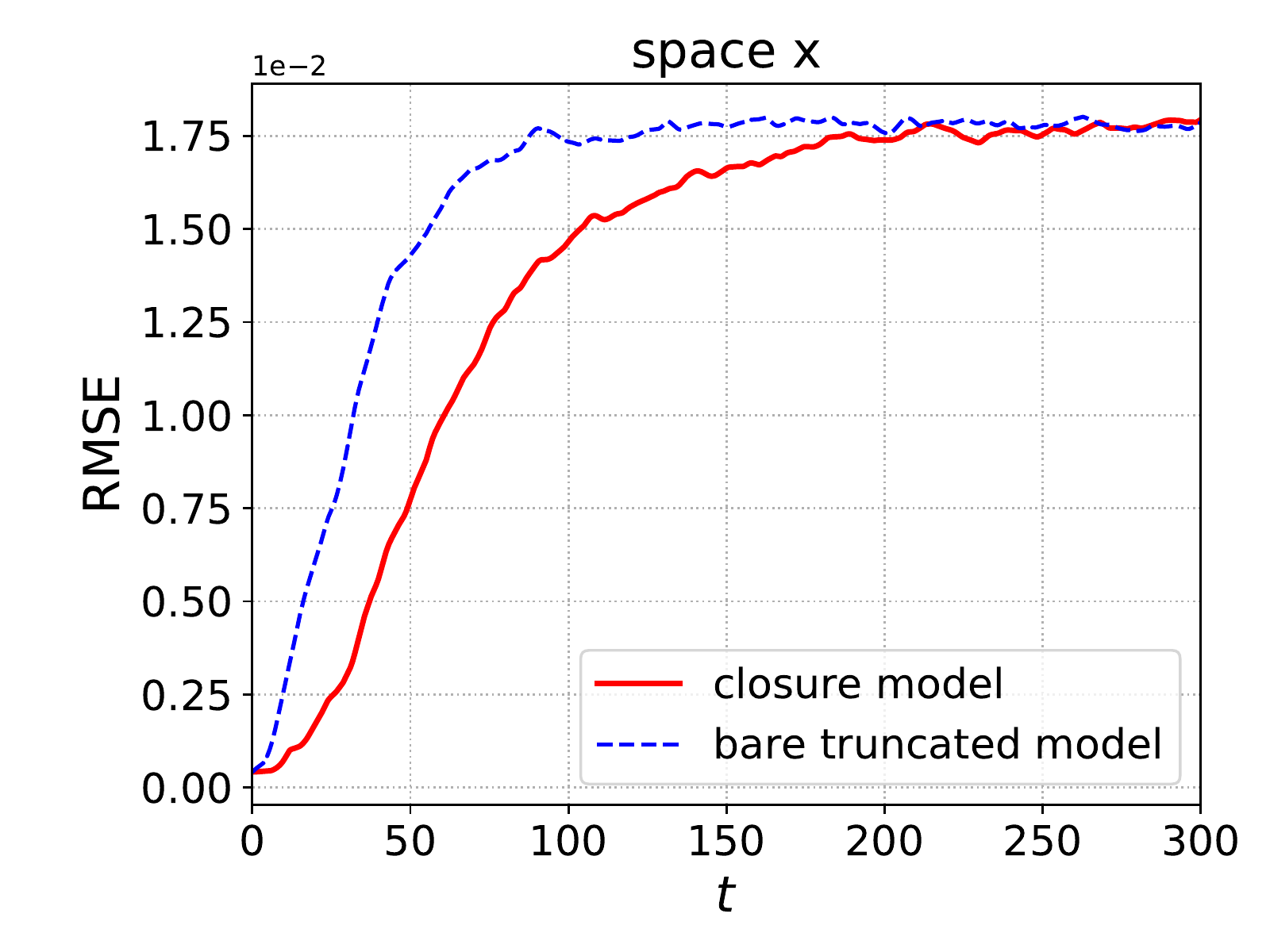}
\endminipage\hfill
\minipage{0.5\textwidth}
\subcaption{ANCR}
\vspace{-5pt}
\includegraphics[width=\textwidth]{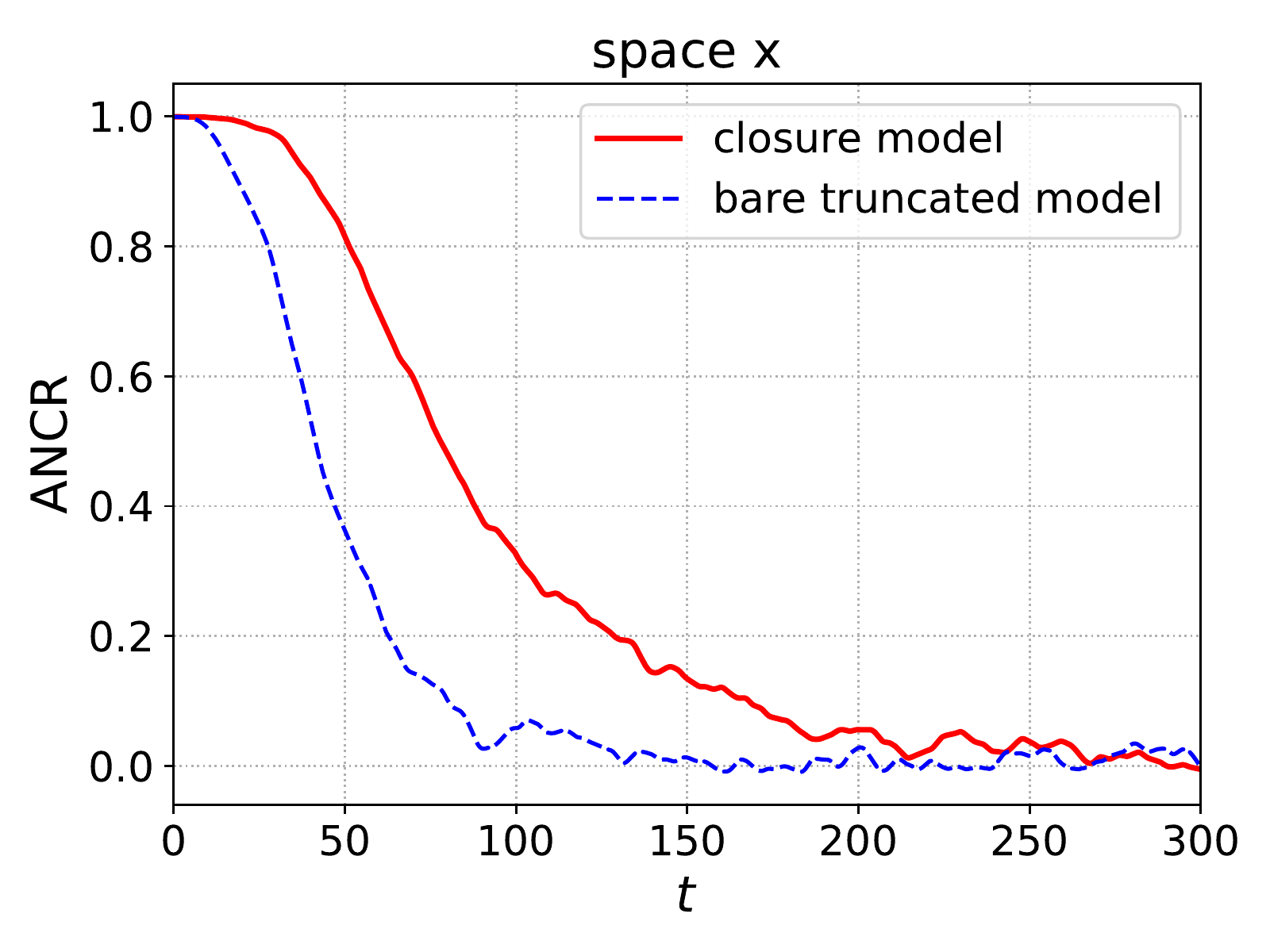}
\endminipage\hfill
\endminipage\hfill
\caption{Prediction error: RMSE and ANCR as a function of time. These metric are estimated by a spatial and temporal average over 1000 initial conditions out-of-sampling.}
\label{KSrmse}
\end{figure}

\begin{figure*}[tbp]
\centering
\includegraphics[scale=0.4]{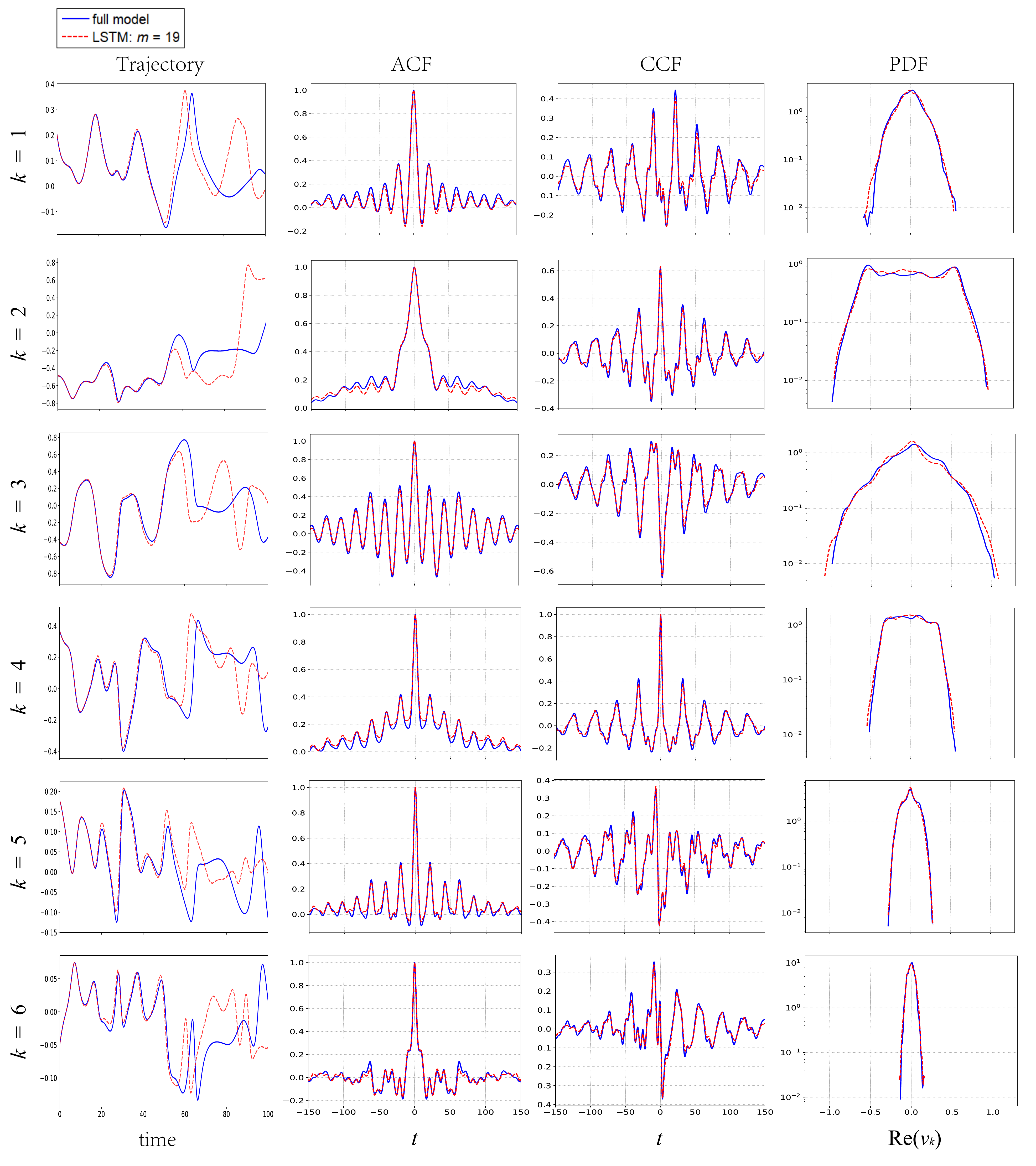}
\caption{(Color online) Comparison of trajectories for Re$(v_{k})$, ACFs for
Re$(v_{k})$, CCFs between $\left\vert v_{k}\right\vert ^2$ and $\left\vert
v_{4}\right\vert ^2$, and PDFs for Re$(v_{k})$ between the full model and
the closure model. Solid blue line corresponds to the full model and dashed
red line corresponds to the closure model. }
\label{Fig_KSsup}
\end{figure*}

\section{Summary} \label{summary}
We have presented a general nonparametric framework for prediction with missing dynamics. The proposed framework reformulates the closure model as a supervised learning problem in which the task is to approximate a high-dimensional map that takes the history of resolved and identifiable unresolved variables to the missing components in the resolved dynamics. Mathematically, we validate the approach with an error bound which implies that the closure framework converges when a consistent learning algorithm is used. Numerically, we demonstrate the effectiveness of our framework in replicating severely truncated complex nonlinear problems arising in many applications. While the framework can be realized using any machine learning technique, we found that the LSTM as a special class of RNN is robust for this particular task.

From the positive numerical tests, several open questions deserve further investigation. For example, justifying the existence of the equilibrium distribution of the closure model; demonstrating the convergence to the underlying equilibrium distribution; characterizing the prediction error using Lyapunov exponents for chaotic dynamics; clarifying the condition under which we can achieve a stable closure model.

{\bf Acknowledgments.} The research of J.H. was partially supported by the ONR Grant N00014-16-1-2888 and NSF Grant DMS-1854299. J.H. thanks Di Qi for sharing the codes for the 57-mode barotropic stress model. S. L. and H. Y. gratefully acknowledge the support of National Supercomputing Center Singapore (NSCC)
and High-Performance Computing (HPC) of the National University of Singapore for providing computational resources, and the support of NVIDIA Corporation with the donation of the Titan Xp GPU used for this research. H. Y. was partially supported by the US National Science Foundation under award DMS-1945029.

\appendix

\section{Proof of Theorem~1}\label{app:A}

Before we prove Theorem~1 in the main text, let us review the following bound which will be used below as well as in the proof of Theorem~3.
\begin{lemma}
\label{lemma1} Let $\alpha, c >0$ be real numbers and $m, T \geq 0$ be integers. Suppose that,
\BEA
E_{T+1} \leq \alpha \sum_{j=T-m}^T {E_j} + c,\nonumber
\EEA
If $E_j = 0$ for $j=-m,\ldots,0$, then for all integer $T\geq 0$.
\BEA
E_{T+1} \leq c (1+\alpha)^T. \nonumber
\EEA
\end{lemma}

\begin{proof}
We proceed by induction. one can verify that, $E_1 \leq c$, $E_2 \leq c(1+\alpha)$ and so on. In fact, we can verify for $j=0,\ldots, m$ one by one that
\BEA
E_j \leq c(1+\alpha)^{j-1}.
\EEA
By induction, for $T\geq m$, we have
\BEA
E_{T+1} \leq c \alpha \sum_{j=T-m}^{T} (1+\alpha)^{j-1} + c \leq c \alpha \sum_{j=1}^{T} (1+\alpha)^{j-1} + c  = c(1+\alpha)^T. \label{gronwallbound}
\EEA
\end{proof}

Now we proceed with the proof of Theorem~1. As we mentioned before, since
\BEA
\mathbb{E}[y_{t+1}|x_t,y_t]=\mathcal{G}(x_t,y_t),\label{Ey}
\EEA
we can rewrite the full dynamics as,
\BEA
x_{t+1} &=& \mathcal{F}(x_t,y_t), \nonumber\\
y_{t+1} &=& \mathbb{E}[Y_{t+1}|x_t,y_t].\nonumber
\EEA
We consider an approximate dynamics given as,
\BEA
\hat{x}_{t+1} &=& \mathcal{F}(\hat{x}_t,\hat{y}_t)  \nonumber\\
\hat{y}_{t+1} &=& \mathbb{E}^\epsilon[y_{t+1}|\hat{x}_t,\hat{y}_t]  + \xi_{t+1},\nonumber
\EEA
where $\xi_{t+1} \sim \Xi$ are Gaussian white noises with variance,
\BEA
\mathbb{E}[\Xi^2] := \mathbb{E}\left[\left(Y_{t+1}-\mathbb{E}^\epsilon[Y_{t+1}|{X}_t,{Y}_t] \right)^2 \right]
= \mathbb{E}\left[\left(\mathbb{E}[Y_{t+1}|X_t,Y_t]-\mathbb{E}^\epsilon[Y_{t+1}|{X}_t,{X}_t] \right)^2 \right] = \mathcal{O}(\epsilon^2).\label{noisebalanced}
\EEA

Define $E_{x,t}: = |x_{t+1} - \hat{x}_{t+1}|$ and $E_{y,t}: = |y_{t+1} - \hat{y}_{t+1}|$, using the consistency in \eqref{Ey} and the Lipschitz conditions of $\mathcal{F}$ and $\mathcal{G}$, we deduce
\begin{align}
E_{y,t+1} &\leq \left\vert\mathbb{E}[Y_{t+1}|x_t,y_t]  - \mathbb{E}^\epsilon[Y_{t+1}|\hat{x}_t,\hat{y}_t]\right\vert  + \left\vert\xi_{t+1}\right\vert \nonumber\\
&\leq \left\vert\mathbb{E}[Y_{t+1}|x_t,y_t]  - \mathbb{E}[Y_{t+1}|\hat{x}_t,\hat{y}_t]\right\vert + \left\vert\mathbb{E}[Y_{t+1}|\hat{x}_t,\hat{y}_t] - \mathbb{E}^\epsilon[Y_{t+1}|\hat{x}_t,\hat{y}_t]\right\vert  +|\xi_{t+1}| \nonumber\\
& < \left\vert\mathcal{G}(x_t,y_t) - \mathcal{G}(\hat{x}_t,\hat{y}_t)\right\vert  + \left\vert\mathbb{E}[Y_{t+1}|\hat{x}_t,\hat{y}_t] - \mathbb{E}^\epsilon[Y_{t+1}|\hat{x}_t,\hat{y}_t]\right\vert   + | \xi_{t+1}| \nonumber\\
&\leq L_1 E_{y,t} +  L_2 E_{x,t}  + \left\vert\mathbb{E}[Y_{t+1}|\hat{x}_t,\hat{y}_t] - \mathbb{E}^\epsilon[Y_{t+1}|\hat{x}_t,\hat{y}_t]\right\vert  + |\xi_{t+1}|\nonumber
\end{align}

Define $E_{x,T+1}^*:=\mathbb{E}[\max_{t=\{0,\ldots,T+1\}}E_{x,t}]$ and $E_{y,T+1}^*:=\mathbb{E}[\max_{t=\{0,\ldots,T+1\}}E_{y,t}]$. Then, by the Burkholder-Davis-Gundy inequality \cite{ps:00},
\BEA
E_{y,T+1}^* \leq L_1 E^*_{y,T} + L_2 E^*_{x,T}  + C \epsilon,\label{Eystarxx}
\EEA
in which we have used \eqref{noisebalanced} to bound the last two terms. This bound can be explicitly written as,
\BEA
E_{y,T+1}^* &\leq& L_1^{T+1} E^*_{y,0} + \sum_{j=0}^T L_1^{j} (L_2 E^*_{x,T-j}  + C \epsilon) \nonumber \\
&\leq& L_1^{T+1} E^*_{y,0} + (L_2 E^*_{x,T}  + C \epsilon) \sum_{j=0}^T L_1^{j} \nonumber\\
&=& (L_2 E^*_{x,T}  + C \epsilon) \frac{L_1^{T+1}-1}{L_1-1}.\label{Eystar2}
\EEA
where we have used the fact that $L_2 E^*_{x,t}  + C \epsilon$ is non-decreasing to get the second inequality and $E^*_{y,0}=0$ to obtain the last equality.

Using similar algebra, we have
\BEA
E_{x,T+1}^*\leq L_3 E^*_{x,T} + L_4 E^*_{y,T}  \label{Ex}
\EEA

Inserting \eqref{Eystar2} into \eqref{Ex}, we obtain
\BEA
E_{x,T+1}^* \leq L_3 E^*_{x,T} + L_4 (L_2 E^*_{x,T-1}  + C \epsilon) \frac{L_1^{T}-1}{L_1-1}\nonumber\\
\leq \alpha (E^*_{x,T} + E^*_{x,T-1}) + CL_4\epsilon \frac{L_1^{T}-1}{L_1-1}
\EEA
where we have define $\alpha = \max\{L_3,L_2L_4\frac{L_1^{T}-1}{L_1-1}\}$.

Given $E_{x,0}^*=0$, we apply the bound in Lemma~\ref{lemma1} for $m=1$,
\BEA
E_{x,T+1}^* &\leq&  CL_4 \epsilon \frac{L_1^{T}-1}{L_1-1} (1+\alpha)^{T} = \mathcal{O}(a^T\epsilon),
\nonumber
\EEA
for some constant $a>1$ and the proof is complete.

\section{Proof of Theorem~2}\label{app:B}

Let $Y_{t+1}^\Delta:= \frac{Y_{t+1}-Y_t}{\Delta}$ such that,
\BEA
\mathbb{E}\Big[Y_{t+1}^\Delta |x_t,y_t\Big]=g(x_t,y_t).\label{Ey2}
\EEA
With this definition, we can rewrite the full dynamics as,
\BEA
x_{t+1} &=& x_t + f(x_t,y_t) \Delta + \Delta^{1/2}\sigma_x \xi_{x,t+1},\nonumber \\
y_{t+1} &=&y_t + \mathbb{E}[Y^\Delta_{t+1}|x_t,y_t] \Delta + \Delta^{1/2}\sigma_y \xi_{y,t+1}.\nonumber
\EEA
We consider an approximate dynamics given as,
\BEA
\hat{x}_{t+1} &=& \hat{x}_t + f(\hat{x}_t,\hat{y}_t) \Delta + \Delta^{1/2}\sigma_x \xi_{x,t+1}, \nonumber\\
\hat{y}_{t+1} &=& \hat{y}_t + \mathbb{E}^\epsilon[Y^\Delta_{t+1}|\hat{x}_t,\hat{y}_t]\Delta  + \Delta^{1/2}\hat\sigma_y \xi_{y,t+1}.\nonumber
\EEA
First, notice that
\BEA
\Delta \hat\sigma_y^2 &=& \mathbb{E}\left[\left(Y_{t+1}-Y_{t}-\Delta \mathbb{E}^\epsilon[Y^\Delta_{t+1}|X_t,Y_t]\right)^2\right]\nonumber \\
&\leq&\mathbb{E}\left[\left(Y_{t+1}-Y_t-\Delta \mathbb{E}[Y^\Delta_{t+1}|X_t,Y_t]\right)^2\right]  + \Delta^2\mathbb{E}\left(\mathbb{E}[Y^\Delta_{t+1}|X_t,Y_t] -\mathbb{E}^\epsilon[Y^\Delta_{t+1}|X_t,Y_t]\right)^2 \ldots \nonumber \\
 && +2\Delta \mathbb{E}\left[\left(Y_{t+1}-Y_t-\Delta \mathbb{E}[Y^\Delta_{t+1}|X_t,Y_t]\right)\left(\mathbb{E}[Y^\Delta_{t+1}|X_t,Y_t] -\mathbb{E}^\epsilon[Y^\Delta_{t+1}|X_t,Y_t]\right)\right] \nonumber \\
&=& \Delta\sigma_y^2 + \mathcal{O}(\Delta^2\epsilon^2), \label{noisebalanced2}
\EEA
where the last term vanishes since the mean of $y_{t+1}-y_t-\Delta\mathbb{E}[Y^\Delta_{t+1}|x_t,y_t]=\Delta^{1/2}\sigma_y \xi_{x,t+1}$ is zero.


Define $E_{x,t+1}: = |x_{t+1} - \hat{x}_{t+1}|$ and $E_{y,t+1}: = |y_{t+1} - \hat{y}_{t+1}|$, using the consistency in \eqref{Ey2} and the Lipschitz conditions of $f$ and $g$, we deduce
\begin{align}
E_{y,t+1} &\leq E_{y,t}+ \Delta\left\vert\mathbb{E}[Y^\Delta_{t+1}|x_t,y_t]  - \mathbb{E}^\epsilon[Y^\Delta_{t+1}|\hat{x}_t,\hat{y}_t]\right\vert  + \Delta^{1/2} \left\vert\sigma_y-\hat\sigma_y\right\vert |\xi_{y,t+1}| \nonumber\\
&\leq E_{y,t} +\Delta \left\vert\mathbb{E}[Y^\Delta_{t+1}|x_t,y_t]  - \mathbb{E}[Y^\Delta_{t+1}|\hat{x}_t,\hat{y}_t]\right\vert + \Delta\left\vert\mathbb{E}[Y^\Delta_{t+1}|\hat{x}_t,\hat{y}_t] - \mathbb{E}^\epsilon[Y^\Delta_{t+1}|\hat{x}_t,\hat{y}_t]\right\vert  + \Delta^{1/2} \left\vert\sigma_y-\hat\sigma_y\right\vert |\xi_{y,t+1}| \nonumber\\
& < E_{y,t} + \Delta \left\vert g(x_t,y_t) - g(\hat{x}_t,\hat{y}_t)\right\vert  +\Delta \left\vert\mathbb{E}[Y^\Delta_{t+1}|\hat{x}_t,\hat{y}_t] - \mathbb{E}^\epsilon[Y^\Delta_{t+1}|\hat{x}_t,\hat{y}_t]\right\vert   + \Delta^{1/2} \left\vert\sigma_y-\hat\sigma_y\right\vert |\xi_{y,t+1}| \nonumber\\
&\leq (1+\Delta \ell) E_{y,t} + \Delta  \ell E_{x,t}  + \Delta \left\vert\mathbb{E}[Y^\Delta_{t+1}|\hat{x}_t,\hat{y}_t] - \mathbb{E}^\epsilon[Y^\Delta_{t+1}|\hat{x}_t,\hat{y}_t]\right\vert  + \Delta^{1/2} \left\vert\sigma_y-\hat\sigma_y\right\vert |\xi_{y,t+1}|.\nonumber
\end{align}
where $\ell=\mathcal{O}(1)$ denotes the largest Lipschitz constant in all directions. Define $E_{x,T+1}^*:=\mathbb{E}[\max_{t=\{0,\ldots,T+1\}}E_{x,t}]$ and $E_{y,T+1}^*:=\mathbb{E}[\max_{t=\{0,\ldots,T+1\}}E_{y,t}]$. Then, by the Burkholder-Davis-Gundy inequality \cite{ps:00}, we have
\BEA
E_{y,T+1}^* \leq (1+\Delta \ell) E^*_{y,T} + \Delta  \ell E^*_{x,T}  + C \Delta \epsilon,\label{Eystar}
\EEA
where we have used \eqref{noisebalanced2}.
Concatenate this with
\BEA
E_{x,T+1}^* \leq (1+\Delta \ell) E^*_{x,T} + \Delta  \ell E^*_{y,T},  \nonumber
\EEA
we have
\BEA
E_{T+1}^* \leq (I+ A + A^2 + \ldots + A^{T})b,\nonumber
\EEA
where
\BEA
E_{T+1}^* = \begin{pmatrix} E_{x,T+1}^* \\ E_{y,T+1}^*
\end{pmatrix},\quad\quad
A = \begin{pmatrix} 1+\Delta \ell & \Delta \ell \\ \Delta \ell & 1+\Delta \ell
\end{pmatrix},\quad\quad b = \begin{pmatrix} 0  \\ C\Delta\epsilon
\end{pmatrix}.\nonumber
\EEA
Using the fact that,
\BEA
A = \frac{1}{2}\begin{pmatrix} 1 & 1 \\ -1 & 1 \end{pmatrix}\begin{pmatrix} 1 & 0 \\ 0 & 1+2\ell\Delta\end{pmatrix}\begin{pmatrix} 1 & -1 \\ 1 & 1 \end{pmatrix},\nonumber
\EEA
one can deduce that,
\BEA
E_{x,T+1}^* &\leq& C\Delta \epsilon \Big(-(T+1) + \frac{1-(1+2\ell\Delta)^{T+1}}{-2\ell\Delta} \Big) \nonumber\\
&\leq &C\Delta \epsilon \left( -(T+1)+\left[ \frac{-1+\left( 1+2\ell \Delta
\left( T+1\right) +\frac{T\left( T+1\right) }{2}\left( 2\ell \Delta \right)
^{2}+O(\Delta ^{3}T^{3})\right) }{2\ell \Delta }\right] \right)  \nonumber\\
&=&C\Delta \epsilon \frac{T\left( T+1\right) }{2}2\ell \Delta +O(\epsilon
\Delta ^{3}T^{3})\nonumber \\
&=& \mathcal{O}(\epsilon\Delta^2 T^2).\nonumber
\EEA
where we use Taylor expansion over small $2\ell\Delta$ and the proof is complete.

\comment{
Recall that discrete Gronwall Lemma for explicit finite difference \cite{emmrich1999discrete}, that is, if
\BEA
a_{T+1} \leq (1+\lambda \Delta)a_T + \Delta g_T\nonumber
\EEA
where $g_T$ is non-decreasing, we have
\BEA
a_T \leq (1+\lambda \Delta)^{T} a_0 + \frac{g_{T-1}}{\lambda}((1+\lambda\Delta)^T-1)\nonumber
\EEA
In our case, we apply it to \eqref{Eystar} such that $\lambda=\ell_1$, $a_0=0$, and $g_t =  \ell_2 E^*_{x,t}  + C \epsilon \Delta^{-1}$ is monotonically increasing and obtain,
\BEA
E_{y,T}^* &\leq& \frac{1}{\ell_1} ( \ell_2 E^*_{x,T-1}  + C \epsilon\Delta^{-1}) ((1+\ell_1\Delta )^{T}-1)\nonumber\\
&=& (\ell_2 E^*_{x,T-1}  + C \epsilon \Delta^{-1}) (T\Delta + \mathcal{O}(T^2\Delta^2)) \nonumber\\
&=& T \Delta \ell_2 E^*_{x,T1}  + C \epsilon T + \mathcal{O}(T^2\Delta^2,\epsilon \Delta T^2). \nonumber
\EEA

Using similar algebra, we have
\BEA
E_{x,T+1}^* \leq (1+\Delta \ell_3) E^*_{x,T} + \Delta  \ell_4 E^*_{y,T}  \label{Exxx}
\EEA

Inserting this bound to \eqref{Exxx}, we obtain
\BEA
E_{x,T+1}^* &\leq& (1+\Delta \ell_3) E^*_{x,T} + \Delta  \ell_4 (T \Delta \ell_2 E^*_{x,T-1}  + C \epsilon T + \mathcal{O}(T^2\Delta^2,\epsilon \Delta T^2) ).\nonumber
\EEA

Now, setting $\alpha=\max\{ 1+\Delta \ell_3, \Delta^2  \ell_2\ell_4 T\}$, by induction,
\BEA
E_{x,T+1}^* &\leq&  C (\epsilon T\Delta+\mathcal{O}(T^2\Delta^3,\epsilon \Delta T^3)) (1+\alpha)^{T-1} = \mathcal{O}(\epsilon T\Delta (1+\alpha)^{T-1}).\nonumber
\EEA
Since the Lipschitz constants are order-one, we have for $T\sim\Delta^{-1}$, $\alpha=\mathcal{O}(1)$.

and the proof is complete.
}

\section{Proof of Theorem~3}\label{app:C}
In this case, we have
\BEA
\mathbb{E}[\Theta_{t+1}|\bm{z}_{t,m}] = \bar{\mathcal{G}}_0(x_t,\theta_t) + \sum_{k=1}^m \bar{\mathcal{G}}_k (x_{t-k},\theta_{t-k})+(QS)^{m+1}\pi(x_{t-m},y_{t-m}),
\EEA
where $\bm{z}_{t,m} = (\bm{x}_{t-m:t},\bm{\theta}_{t-m:t})$.

Define $E_{\theta,t+1}: = |\theta_{t+1} - \hat{\theta}_{t+1}|$ and $E_{x,t+1}: = |x_{t+1} - \hat{x}_{t+1}|$. By the Assumption~\ref{assumption2}, $\mathcal{F}$ and $\mathcal{G}_k$ are Lipschitz continuous on $x$ and $\theta$, and $(QS)^{m+1}$ is a bounded linear operator in uniform sense.
Thus, we have
\BEA
E_{\theta,t+1} &\leq& \left\vert\mathbb{E}[\Theta_{t+1}|\bm{z}_{t,m}]  -\mathbb{E}^\epsilon[\Theta_{t+1}|\bm{\hat{z}}_{t,m}]\right\vert  + |\xi_{t+1}| \nonumber\\
&\leq& \left\vert\mathbb{E}[\Theta_{t+1}|\bm{z}_{t,m}]  -\mathbb{E}[\Theta_{t+1}|\bm{\hat{z}}_{t,m}]\right\vert+ \left\vert\mathbb{E}[\Theta_{t+1}|\bm{\hat{z}}_{t,m}]  -\mathbb{E}^\epsilon[\Theta_{t+1}|\bm{\hat{z}}_{t,m}]\right\vert  +|\xi_{t+1}| \nonumber\\
&\leq& \sum_{k=0}^m \left\vert \mathcal{G}_k (x_{t-k},\theta_{t-k})-\mathcal{G}_k (\hat{x}_{t-k},\hat{\theta}_{t-k})\right\vert
+ \left\vert(QS)^{m+1}\pi(x_{t-m},y_{t-m})-(QS)^{m+1}\pi(\hat{x}_{t-m},\hat{y}_{t-m})\right\vert
\nonumber  \\ && +\left\vert\mathbb{E}[\Theta_{t+1}|\bm{\hat{z}}_{t,m}] - \mathbb{E}^\epsilon[\Theta_{t+1}|\bm{\hat{z}}_{t,m}]\right\vert   + |\xi_{t+1}|\nonumber\\
&\leq &\sum_{s=t-m}^t K_{s-(t-m)} E_{\theta,s} +   \sum_{s=t-m}^t L_{s-(t-m)} E_{x,s} + \left\vert\mathbb{E}[\Theta_{t+1}|\bm{\hat{z}}_{t,m}] - \mathbb{E}^\epsilon[\Theta_{t+1}|\bm{\hat{z}}_{t,m}]\right\vert  + |\xi_{t+1}|,\label{EBB}
\EEA
where $K_s, L_s$ are Lipschitz constants.

Define $E_{\theta,T+1}^*:=\mathbb{E}[\max_{t=\{0,\ldots,T+1\}}E_{\theta,t}]$ and $E_{x,T+1}^*:=\mathbb{E}[\max_{t=\{0,\ldots,T+1\}}E_{x,t}]$. Then we have,
\begin{align}
E^*_{\theta,T+1} &\leq\sum_{s=T-m}^T K_{s-(T-m)} E^*_{\theta,s} +   \sum_{s=T-m}^T  L_{s-(T-m)} E^*_{x,s} + C\epsilon,\label{thm1theta}
\end{align}
where the expectation of the last term in \eqref{EBB} is bounded using the Burkholder-Davis-Gundy inequality \cite{ps:00}. We should point out that since the expectation in
\BEA
\mathbb{E}\left[\left(\mathbb{E}[\theta_{t+1}|{Z}_{t,m}]-\mathbb{E}^\epsilon[\Theta_{t+1}|{Z}_{t,m}] \right)^2 \right] = \mathcal{O}(\epsilon^2)\nonumber,
\EEA
is defined with respect to the pushforward measure $\nu := Z_{t,m*}\mu$, that is, $\nu(B)= \mu(Z_{t,m}^{-1}(B))$, for all $B\in \mathcal{B}(\mathcal{Z})$ in the $\sigma$-algebra, associated to the random variable $Z_{t,m}:\mathcal{X}\times\mathcal{Y} \to\mathcal{Z}$. Since  $\nu(\mathcal{Z})=\int_{\mathcal{Z}}d\nu(z)<\infty$, it is clear that expectation of the third term in \eqref{EBB},
\BEA
\mathbb{E}\Big[\left\vert\mathbb{E}[\theta_{t+1}|{Z}_{t,m}]-\mathbb{E}^\epsilon[\Theta_{t+1}|{Z}_{t,m}]\right\vert \Big]\leq \mathbb{E}\Big[\left\vert\mathbb{E}[\theta_{t+1}|{Z}_{t,m}]-\mathbb{E}^\epsilon[\Theta_{t+1}|{Z}_{t,m}]\right\vert^2 \Big]^{1/2} \nu(\mathcal{Z})^{1/2} = C\epsilon.
\EEA
is also bounded by order-$\epsilon$.

Let $0<K := \max\{K_0,\ldots,K_m\}$, applying the bound in Lemma~\ref{lemma1}, we can obtain from \eqref{thm1theta}
\BEA
E^*_{\theta,T+1} &\leq& \Big(  \sum_{s=T-m}^T  L_{s-(T-m)} E^*_{x,s}  + C\epsilon\Big) (1+K)^T.\label{Etheta}
\EEA
Using similar algebra, we have
\BEA
E_{x,T+1}^*\leq L_{m+1} E^*_{x,T} + K_{m+1} E^*_{\theta,T},  \label{Ex3}
\EEA
for some constants $K_{m+1},L_{m+1}>0$.
Inserting \eqref{Etheta} into \eqref{Ex3}, let $0<L := \max_{j=0,\ldots,m}\{L_{m+1}, K_{m+1}L_j(1+K)^{T-1} \}$, applying the bound \eqref{gronwallbound}, we obtain
\BEA
E_{x,T+1}^* &\leq& L_{m+1} E^*_{x,T} + K_{m+1}\Big(  \sum_{s=T-m-1}^{T-1}  L_{s-(T-m-1)} E^*_{x,s}  + C\epsilon\Big)(1+K)^{T-1} \\
&\leq&  L  \sum_{s=T-m-1}^T  E^*_{x,s}  + K_{m+1}C\epsilon (1+K)^{T-1} \\
&\leq& K_{m+1}C\epsilon (1+K)^{T-1}(1+L)^T \\ \nonumber
&=& \mathcal{O}(a^T \epsilon),
\EEA
for some $a>1$ and the proof is completed.

\bibliographystyle{abbrv}
\bibliography{ref,refhz}

\begin{thebibliography}{10}

\bibitem{DBLP:journals/corr/abs-1902-01028}
Z.~Allen-Zhu and Y.~Li.
\newblock Can {SGD} learn recurrent neural networks with provable
  generalization?
\newblock {\em CoRR}, abs/1902.01028, 2019.

\bibitem{DBLP:journals/corr/abs-1811-04918}
Z.~Allen-Zhu, Y.~Li, and Y.~Liang.
\newblock Learning and generalization in overparameterized neural networks,
  going beyond two layers.
\newblock {\em CoRR}, abs/1811.04918, 2018.

\bibitem{RNNconvergence}
Z.~Allen-Zhu, Y.~Li, and Z.~Song.
\newblock On the convergence rate of training recurrent neural networks.
\newblock {\em CoRR}, abs/1810.12065, 2018.

\bibitem{bao2003numerical}
W.~Bao, S.~Jin, and P.~A. Markowich.
\newblock Numerical study of time-splitting spectral discretizations of
  nonlinear schr{\"o}dinger equations in the semiclassical regimes.
\newblock {\em SIAM Journal on Scientific Computing}, 25(1):27--64, 2003.

\bibitem{barron1993}
A.~R. Barron.
\newblock Universal approximation bounds for superpositions of a sigmoidal
  function.
\newblock {\em IEEE Transactions on Information Theory}, 39(3):930--945, May
  1993.

\bibitem{bh:14}
T.~Berry and J.~Harlim.
\newblock {Linear Theory for Filtering Nonlinear Multiscale Systems with Model
  Error}.
\newblock {\em Proc. Roy. Soc. A 20140168}, 2014.

\bibitem{bh:16jcp}
T.~Berry and J.~Harlim.
\newblock Semiparametric modeling: Correcting low-dimensional model error in
  parametric models.
\newblock {\em J. Comput. Phys.}, 308:305--321, 2016.

\bibitem{DBLP:journals/corr/abs-1710-10174}
A.~Brutzkus, A.~Globerson, E.~Malach, and S.~Shalev{-}Shwartz.
\newblock {SGD} learns over-parameterized networks that provably generalize on
  linearly separable data.
\newblock {\em CoRR}, abs/1710.10174, 2017.

\bibitem{DBLP:journals/corr/abs-1902-01384}
Y.~Cao and Q.~Gu.
\newblock A generalization theory of gradient descent for learning
  over-parameterized deep relu networks.
\newblock {\em CoRR}, abs/1902.01384, 2019.

\bibitem{carnevale1987nonlinear}
G.~F. Carnevale and J.~S. Frederiksen.
\newblock Nonlinear stability and statistical mechanics of flow over
  topography.
\newblock {\em Journal of Fluid Mechanics}, 175:157--181, 1987.

\bibitem{chen2019on}
M.~Chen, X.~Li, and T.~Zhao.
\newblock On generalization bounds of a family of recurrent neural networks,
  2019.

\bibitem{chen2019spatial}
N.~Chen, A.~J. Majda, and X.~T. Tong.
\newblock Spatial localization for nonlinear dynamical stochastic models for
  excitable media.
\newblock {\em arXiv preprint arXiv:1901.07318}, 2019.

\bibitem{chk:02}
A.~Chorin, O.~Hald, and R.~Kupferman.
\newblock Optimal prediction with memory.
\newblock {\em Physica D: Nonlinear Phenomena}, 166(3):239--257, 2002.

\bibitem{chorin2007problem}
A.~Chorin and P.~Stinis.
\newblock Problem reduction, renormalization, and memory.
\newblock {\em Communications in Applied Mathematics and Computational
  Science}, 1(1):1--27, 2007.

\bibitem{chui2012linear}
C.~K. Chui and G.~Chen.
\newblock {\em Linear systems and optimal control}, volume~18.
\newblock Springer Science \& Business Media, 2012.

\bibitem{crommelin2008subgrid}
D.~Crommelin and E.~Vanden-Eijnden.
\newblock Subgrid-scale parameterization with conditional markov chains.
\newblock {\em Journal of the Atmospheric Sciences}, 65(8):2661--2675, 2008.

\bibitem{darve2009computing}
E.~Darve, J.~Solomon, and A.~Kia.
\newblock Computing generalized langevin equations and generalized
  fokker--planck equations.
\newblock {\em Proceedings of the National Academy of Sciences},
  106(27):10884--10889, 2009.

\bibitem{Weinan2018}
W.~E, C.~Ma, and L.~Wu.
\newblock A priori estimates of the generalization error for two-layer neural
  networks.
\newblock {\em Communications in Mathematical Sciences}, 17(5):1407--1425,
  2019.

\bibitem{edson_bunder_mattner_roberts_2019}
R.~A. EDSON, J.~E. BUNDER, T.~W. MATTNER, and A.~J. ROBERTS.
\newblock Lyapunov exponents of the kuramoto?sivashinsky pde.
\newblock {\em The ANZIAM Journal}, 61(3):270?285, 2019.

\bibitem{franzke2009systematic}
C.~Franzke, I.~Horenko, A.~J. Majda, and R.~Klein.
\newblock Systematic metastable atmospheric regime identification in an agcm.
\newblock {\em Journal of the Atmospheric Sciences}, 66(7):1997--2012, 2009.

\bibitem{gdls:2014}
T.~Gao, J.~Duan, X.~Li, and R.~Song.
\newblock Mean exit time and escape probability for dynamical systems driven by
  lévy noises.
\newblock {\em SIAM Journal on Scientific Computing}, 36(3):A887--A906, 2014.

\bibitem{givon2004extracting}
D.~Givon, R.~Kupferman, and A.~Stuart.
\newblock Extracting macroscopic dynamics: model problems and algorithms.
\newblock {\em Nonlinearity}, 17(6):R55, 2004.

\bibitem{goodman1994stability}
J.~Goodman.
\newblock Stability of the kuramoto-sivashinsky and related systems.
\newblock {\em Communications on Pure and Applied Mathematics}, 47(3):293--306,
  1994.

\bibitem{gouasmi2017priori}
A.~Gouasmi, E.~J. Parish, and K.~Duraisamy.
\newblock A priori estimation of memory effects in reduced-order models of
  nonlinear systems using the mori--zwanzig formalism.
\newblock {\em Proceedings of the Royal Society A: Mathematical, Physical and
  Engineering Sciences}, 473(2205):20170385, 2017.

\bibitem{grote1999dynamic}
M.~J. Grote, A.~J. Majda, and C.~G. Ragazzo.
\newblock Dynamic mean flow and small-scale interaction through topographic
  stress.
\newblock {\em Journal of Nonlinear Science}, 9(1):89--130, 1999.

\bibitem{hl:15}
J.~Harlim and X.~Li.
\newblock {Parametric reduced models for the nonlinear Schr\"{o}dinger
  equation}.
\newblock {\em Phys. Rev. E.}, 91:053306, 2015.

\bibitem{hmm:14}
J.~Harlim, A.~Mahdi, and A.~Majda.
\newblock {An ensemble Kalman filter for statistical estimation of physics
  constrained nonlinear regression models}.
\newblock {\em J. Comput. Phys.}, 257, Part A:782--812, 2014.

\bibitem{hirsch2012differential}
M.~W. Hirsch, S.~Smale, and R.~L. Devaney.
\newblock {\em Differential equations, dynamical systems, and an introduction
  to chaos}.
\newblock Academic press, 2012.

\bibitem{Hochreiter_1997}
S.~Hochreiter and J.~Schmidhuber.
\newblock Long short-term memory.
\newblock {\em Neural Computation}, 9(8):1735--1780, Nov. 1997.

\bibitem{jh:19}
S.~Jiang and J.~Harlim.
\newblock Modeling of missing dynamical systems: Deriving parametric models
  using a nonparametric framework.
\newblock {\em arXiv:1905.08082}, 2019.

\bibitem{kevrekidis2016solitons}
P.~Kevrekidis and D.~Frantzeskakis.
\newblock Solitons in coupled nonlinear schr{\"o}dinger models: a survey of
  recent developments.
\newblock {\em Reviews in Physics}, 1:140--153, 2016.

\bibitem{kbm:10}
B.~Khouider, J.~A. Biello, and A.~J. Majda.
\newblock A stochastic multicloud model for tropical convection.
\newblock {\em Comm. Math. Sci.}, 8:187--216, 2010.

\bibitem{kondrashov2015data}
D.~Kondrashov, M.~D. Chekroun, and M.~Ghil.
\newblock Data-driven non-markovian closure models.
\newblock {\em Physica D: Nonlinear Phenomena}, 297:33--55, 2015.

\bibitem{kunita1997stochastic}
H.~Kunita.
\newblock {\em Stochastic flows and stochastic differential equations},
  volume~24.
\newblock Cambridge university press, 1997.

\bibitem{10.1143/PTP.55.356}
Y.~Kuramoto and T.~Tsuzuki.
\newblock {Persistent Propagation of Concentration Waves in Dissipative Media
  Far from Thermal Equilibrium}.
\newblock {\em Progress of Theoretical Physics}, 55(2):356--369, 02 1976.

\bibitem{kwasniok2012data}
F.~Kwasniok.
\newblock Data-based stochastic subgrid-scale parametrization: an approach
  using cluster-weighted modelling.
\newblock {\em Philosophical Transactions of the Royal Society A: Mathematical,
  Physical and Engineering Sciences}, 370(1962):1061--1086, 2012.

\bibitem{laquey1975nonlinear}
R.~E. LaQuey, S.~Mahajan, P.~Rutherford, and W.~Tang.
\newblock Nonlinear saturation of the trapped-ion mode.
\newblock {\em Physical Review Letters}, 34(7):391, 1975.

\bibitem{lin2019data}
K.~K. Lin and F.~Lu.
\newblock Data-driven model reduction, wiener projections, and the mori-zwanzig
  formalism.
\newblock {\em arXiv preprint arXiv:1908.07725}, 2019.

\bibitem{lu2016comparison}
F.~Lu, K.~Lin, and A.~Chorin.
\newblock Comparison of continuous and discrete-time data-based modeling for
  hypoelliptic systems.
\newblock {\em Communications in Applied Mathematics and Computational
  Science}, 11(2):187--216, 2016.

\bibitem{lu2017data}
F.~Lu, K.~K. Lin, and A.~J. Chorin.
\newblock Data-based stochastic model reduction for the kuramoto--sivashinsky
  equation.
\newblock {\em Physica D: Nonlinear Phenomena}, 340:46--57, 2017.

\bibitem{lu2017accounting}
F.~Lu, X.~Tu, and A.~J. Chorin.
\newblock Accounting for model error from unresolved scales in ensemble kalman
  filters by stochastic parameterization.
\newblock {\em Monthly Weather Review}, 145(9):3709--3723, 2017.

\bibitem{Shen3}
J.~{Lu}, Z.~{Shen}, H.~{Yang}, and S.~{Zhang}.
\newblock {Deep Network Approximation for Smooth Functions}.
\newblock {\em arXiv e-prints}, page arXiv:2001.03040, Jan. 2020.

\bibitem{DBLP:journals/corr/abs-1709-02540}
Z.~Lu, H.~Pu, F.~Wang, Z.~Hu, and L.~Wang.
\newblock The expressive power of neural networks: {A} view from the width.
\newblock {\em CoRR}, abs/1709.02540, 2017.

\bibitem{ma2018model}
C.~Ma and J.~Wang.
\newblock Model reduction with memory and the machine learning of dynamical
  systems.
\newblock {\em arXiv preprint arXiv:1808.04258}, 2018.

\bibitem{mh:13}
A.~Majda and J.~Harlim.
\newblock {Physics constrained nonlinear regression models for time series.}
\newblock {\em Nonlinearity}, 26:201--217, 2013.

\bibitem{mtv:03}
A.~Majda, I.~Timofeyev, and E.~Vanden-Eijnden.
\newblock {Systematic strategies for stochastic mode reduction in climate}.
\newblock {\em Journal of the Atmospheric Sciences}, 60:1705--1722, 2003.

\bibitem{mw:06}
A.~Majda and X.~Wang.
\newblock {\em {Nonlinear dynamics and statistical theories for basic
  geophysical flows}}.
\newblock Cambridge University Press, UK, 2006.

\bibitem{majda1999models}
A.~J. Majda, I.~Timofeyev, and E.~V. Eijnden.
\newblock Models for stochastic climate prediction.
\newblock {\em Proceedings of the National Academy of Sciences},
  96(26):14687--14691, 1999.

\bibitem{majda2001mathematical}
A.~J. Majda, I.~Timofeyev, and E.~Vanden~Eijnden.
\newblock A mathematical framework for stochastic climate models.
\newblock {\em Communications on Pure and Applied Mathematics: A Journal Issued
  by the Courant Institute of Mathematical Sciences}, 54(8):891--974, 2001.

\bibitem{maulik2019time}
R.~Maulik, A.~Mohan, B.~Lusch, S.~Madireddy, and P.~Balaprakash.
\newblock Time-series learning of latent-space dynamics for reduced-order model
  closure.
\newblock {\em arXiv preprint arXiv:1906.07815}, 2019.

\bibitem{Hadrien}
H.~Montanelli and Q.~Du.
\newblock New error bounds for deep networks using sparse grids.
\newblock 2017.

\bibitem{HadrienYang}
H.~Montanelli and H.~Yang.
\newblock Error bounds for deep relu networks using the kolmogorov–arnold
  superposition theorem.
\newblock {\em Neural Networks}, 129:1 -- 6, 2020.

\bibitem{HadrienYangDu}
H.~Montanelli, H.~Yang, and Q.~Du.
\newblock Deep relu networks overcome the curse of dimensionality for
  bandlimited functions.
\newblock {\em arXiv:1903.00735 [math.NA]}, 2019.

\bibitem{mori:65}
H.~Mori.
\newblock Transport, collective motion, and {Brownian} motion.
\newblock {\em Prog. Theor. Phys.}, 33:423 -- 450, 1965.

\bibitem{Mozer:1995:FBA:201784.201791}
M.~C. Mozer.
\newblock Backpropagation.
\newblock chapter A Focused Backpropagation Algorithm for Temporal Pattern
  Recognition, pages 137--169. L. Erlbaum Associates Inc., Hillsdale, NJ, USA,
  1995.

\bibitem{Masaaki}
R.~Nakada and M.~Imaizumi.
\newblock Adaptive approximation and estimation of deep neural network to
  intrinsic dimensionality.
\newblock {\em arXiv:1907.02177 [stat.ML]}, 2019.

\bibitem{pan2018data}
S.~Pan and K.~Duraisamy.
\newblock Data-driven discovery of closure models.
\newblock {\em SIAM Journal on Applied Dynamical Systems}, 17(4):2381--2413,
  2018.

\bibitem{parish2017dynamic}
E.~J. Parish and K.~Duraisamy.
\newblock A dynamic subgrid scale model for large eddy simulations based on the
  mori--zwanzig formalism.
\newblock {\em Journal of Computational Physics}, 349:154--175, 2017.

\bibitem{Pascanu:2013}
R.~Pascanu, T.~Mikolov, and Y.~Bengio.
\newblock On the difficulty of training recurrent neural networks.
\newblock In {\em Proceedings of the 30th International Conference on
  International Conference on Machine Learning - Volume 28}, ICML'13, pages
  III--1310--III--1318. JMLR.org, 2013.

\bibitem{ps:00}
G.~Pavliotis and A.~Stuart.
\newblock {\em {Multiscale Methods: Averaging and Homogenization}}, volume~53
  of {\em Texts in Applied Mathematics}.
\newblock Springer, 2000.

\bibitem{qi2017low}
D.~Qi and A.~J. Majda.
\newblock Low-dimensional reduced-order models for statistical response and
  uncertainty quantification: Barotropic turbulence with topography.
\newblock {\em Physica D: Nonlinear Phenomena}, 343:7--27, 2017.

\bibitem{robinson:utility}
A.~J. Robinson and F.~Fallside.
\newblock The utility driven dynamic error propagation network.
\newblock Technical Report CUED/F-INFENG/TR.1, Engineering Department,
  Cambridge University, Cambridge, UK, 1987.

\bibitem{ShenYangZhang2}
Z.~Shen, H.~Yang, and S.~Zhang.
\newblock Deep network approximation characterized by number of neurons.
\newblock {\em arXiv:1906.05497 [math.NA]}, 2019.

\bibitem{Shen4}
Z.~Shen, H.~Yang, and S.~Zhang.
\newblock Deep network approximation with discrepancy being reciprocal of width
  to power of depth.
\newblock {\em ArXiv}, abs/2006.12231, 2020.

\bibitem{song2013}
L.~Song, K.~Fukumizu, and A.~Gretton.
\newblock Kernel embeddings of conditional distributions: A unified kernel
  framework for nonparametric inference in graphical models.
\newblock {\em IEEE Signal Process. Mag.}, 30(4):98--111, 2013.

\bibitem{song2009}
L.~Song, J.~Huang, A.~Smola, and K.~Fukumizu.
\newblock Hilbert space embeddings of conditional distributions with
  applications to dynamical systems.
\newblock In {\em Proceedings of the 26th Annual International Conference on
  Machine Learning}, pages 961--968. ACM, 2009.

\bibitem{steinwart2001influence}
I.~Steinwart.
\newblock On the influence of the kernel on the consistency of support vector
  machines.
\newblock {\em Journal of machine learning research}, 2(Nov):67--93, 2001.

\bibitem{stone1982optimal}
C.~J. Stone.
\newblock Optimal global rates of convergence for nonparametric regression.
\newblock {\em The annals of statistics}, pages 1040--1053, 1982.

\bibitem{vanden2010transition}
E.~Vanden-Eijnden.
\newblock Transition-path theory and path-finding algorithms for the study of
  rare events.
\newblock {\em Annual review of physical chemistry}, 61:391--420, 2010.

\bibitem{vlachas2018data}
P.~R. Vlachas, W.~Byeon, Z.~Y. Wan, T.~P. Sapsis, and P.~Koumoutsakos.
\newblock Data-driven forecasting of high-dimensional chaotic systems with long
  short-term memory networks.
\newblock {\em Proceedings of the Royal Society A: Mathematical, Physical and
  Engineering Sciences}, 474(2213):20170844, 2018.

\bibitem{eweinan2007heterogeneous}
E.~Weinan, B.~Engquist, X.~Li, W.~Ren, and E.~Vanden-Eijnden.
\newblock Heterogeneous multiscale methods: a review.
\newblock {\em Commun. Comput. Phys}, 2(3):367--450, 2007.

\bibitem{WERBOS1988339}
P.~J. Werbos.
\newblock Generalization of backpropagation with application to a recurrent gas
  market model.
\newblock {\em Neural Networks}, 1(4):339 -- 356, 1988.

\bibitem{wilks2005effects}
D.~S. Wilks.
\newblock Effects of stochastic parametrizations in the lorenz'96 system.
\newblock {\em Quarterly Journal of the Royal Meteorological Society},
  131(606):389--407, 2005.

\bibitem{ZHL:19}
H.~Zhang, J.~Harlim, and X.~Li.
\newblock {Computing linear response statistics using orthogonal polynomial
  based estimators: An RKHS formulation}.
\newblock {\em arXiv:1912.11110}, 2019.

\bibitem{zwanzig:61}
R.~Zwanzig.
\newblock Statistical mechanics of irreversiblity.
\newblock {\em Lectures in Theoretical Physics}, 3:106--141, 1961.

\bibitem{Zwanzig2001}
R.~Zwanzig.
\newblock {\em Nonequilibrium statistical mechanics}.
\newblock Oxford University Press, 2001.

\end{thebibliography}

\end{document}